\documentclass[11pt]{amsart}
\usepackage{times}
\usepackage{amsaddr}
 \usepackage[margin=1in]{geometry}
\usepackage{url}
\usepackage{amssymb}
\usepackage{amsthm}
\usepackage[normalem]{ulem}
\usepackage{mathabx}
\usepackage{bbm}
\usepackage{array}
\usepackage{hyperref}
\usepackage{amsmath}
\usepackage{booktabs}
\usepackage[new]{old-arrows}

\usepackage{etoolbox}
\let\bbordermatrix\bordermatrix
\patchcmd{\bbordermatrix}{8.75}{4.75}{}{}
\patchcmd{\bbordermatrix}{\left(}{\left[}{}{}
\patchcmd{\bbordermatrix}{\right)}{\right]}{}{}

\usepackage[inline]{enumitem}

\usepackage{mathrsfs}
\usepackage{framed}
\usepackage{latexsym}
\usepackage{pb-diagram}
\usepackage[mathscr]{euscript}
\usepackage{graphicx}
\usepackage[all]{xy} 
\usepackage{color}
\usepackage{tikz}
\usepackage{algorithm}
\usepackage[noend]{algpseudocode}
\usepackage{thmtools,thm-restate} 

\usepackage{thmtools,thm-restate}
\usepackage{subcaption}
\usetikzlibrary{backgrounds,calc,positioning,shapes,shadows,arrows,fit,through}
\usepackage{tikz-cd}
\usepackage{cleveref}

\newcommand{\N}{\mathbb{N}}
\newcommand{\Z}{\mathbb{Z}}
\newcommand{\R}{\mathbb{R}}
\newcommand{\Sp}{\mathbb{S}}

\newcommand{\seta}{\mathbf{A}}
\newcommand{\sete}{\mathbf{E}}
\newcommand{\seth}{\mathbf{H}}

\makeatletter
\DeclareFontFamily{U}{tipa}{}
\DeclareFontShape{U}{tipa}{m}{n}{<->tipa10}{}
\newcommand{\arc@char}{{\usefont{U}{tipa}{m}{n}\symbol{62}}}%

\newcommand{\arc}[1]{\mathpalette\arc@arc{#1}}

\newcommand{\arc@arc}[2]{%
  \sbox0{$\m@th#1#2$}%
  \vbox{
    \hbox{\resizebox{\wd0}{\height}{\arc@char}}
    \nointerlineskip
    \box0
  }%
}
\makeatother

\newcommand{\cov}[2]{{N}_{#2}\left({#1}\right)}  
\newcommand{\capa}[2]{C_{#2}(#1)} 

\newcommand{\volmeas}[2]{{\mathbf{vol}_{#1}}\left({#2}\right)}             
\newcommand{\vol}[1]{{\mathbf{Vol}}\left({#1}\right)}

\renewcommand{\emptyset}{\varnothing}

\newcommand{\diam}{\operatorname{diam}}

\DeclareRobustCommand\longtwoheadrightarrow
     {\relbar\joinrel\twoheadrightarrow}

\newcommand{\dis}{\operatorname{dis}}	

\newcommand{\codis}{\operatorname{codis}}	

\newcommand{\dgh}{d_{\operatorname{GH}}}		
\renewcommand{\dh}{d_{\operatorname{H}}}

\newtheorem{theorem}{Theorem}

\newtheorem{question}{Question}

\newtheorem{proposition}{Proposition}[section]
\newtheorem{lemma}[proposition]{Lemma}
\newtheorem{corollary}[proposition]{Corollary}

\theoremstyle{definition}
\newtheorem{example}[proposition]{Example}
\newtheorem{remark}[proposition]{Remark}
\newtheorem{conjecture}{Conjecture}
\newtheorem{defn}{Definition}
\newtheorem{claim}{Claim}
\newtheorem{fact}{Fact}

\title{The Gromov-Hausdorff distance between spheres}

\begin{document}
\date{\today}

\begin{abstract}
We provide general upper and lower bounds for the Gromov-Hausdorff distance $\dgh(\Sp^m,\Sp^n)$ between spheres $\Sp^m$ and $\Sp^n$ (endowed with the round metric) for $0\leq m< n\leq \infty$. Some of these lower bounds are based on certain topological ideas related to the Borsuk-Ulam theorem. Via explicit constructions of (optimal) correspondences we prove that our lower bounds are tight in the cases of  $\dgh(\Sp^0,\Sp^n)$, $\dgh(\Sp^m,\Sp^\infty)$, $\dgh(\Sp^1,\Sp^2)$, $\dgh(\Sp^1,\Sp^3)$ and $\dgh(\Sp^2,\Sp^3)$. We also formulate a number of open  questions.
\end{abstract}

\author{Sunhyuk Lim}
\address{Max Planck Institute for Mathematics in the Sciences, Leipzig.}
\email{sulim@mis.mpg.de}

\author{Facundo M\'emoli}
\address{Department of Mathematics,
The Ohio State University.}
\email{facundo.memoli@gmail.com}

\author{Zane Smith}
\address{Department of Computer Science,
The University of Minnesota.}
\email{smit9474@umn.edu}

\maketitle
\tableofcontents
\section{Introduction}

Despite being widely used in Riemannian geometry \cite{burago-book,petersen-book}, very little is known in terms of the \emph{exact} value of the Gromov-Hausdorff distance between two given spaces.  In a closely related vein, \cite[p.141]{gromov-book} Gromov poses the question of computing/estimating the value of the \emph{box distance}  $\underline{\Box}_1(\Sp^m,\Sp^n)$ (a close relative of $\dgh$) between spheres (viewed as metric measure spaces). In \cite{funano-estimates}, Funano provides asymptotic bounds for this distance via an idea due to Colding (see the discussion preceding Proposition \ref{prop:colding} below).

The Gromov-Hausdorff distance is also a natural choice for expressing the stability of invariants in applied algebraic topology \cite{carlsson-09,carlsson2010characterization,carlsson_2014} and has also been invoked in applications related to shape matching \cite{bbk-book,ms,memoli-dghlp-long} as a notion of dissimilarity between shapes.

In this paper, we consider the problem of estimating the Gromov-Hausdorff distance $\dgh(\Sp^m,\Sp^n)$ between spheres (endowed with their round/geodesic distance). In particular we show that in some cases, topological ideas related to the Borsuk-Ulam theorem yield lower bounds which turn out to be tight.

\subsection{Basic definitions} The Gromov-Hausdorff distance \cite{edwards1975structure,gromov-book} between two bounded metric spaces $(X,d_X)$ and $(Y,d_Y)$ is defined as $$\dgh(X,Y) := \inf\dh(f(X),g(Y)),$$
where $\dh$ denotes the Hausdorff distance between subsets of the ambient space $Z$ and the infimum is taken over all $f,g$ isometric embeddings of $X$ and $Y$ into $Z$, respectively, and over all metric spaces $Z.$ We will henceforth denote by $\mathcal{M}_b$ the collection of all bounded metric spaces.

It is known that $\dgh$ defines a metric on compact metric spaces up to isometry \cite{gromov-book}.  A standard reference is \cite{burago-book}.   A useful property is that whenever $(X,d_X)$ is a compact metric space and for some $\delta>0$ a subset $A\subseteq X$ is a $\delta$-net for $X$, then $\dgh\big((X,d_X),(A,d_X|_{A\times A})\big)\leq \delta.$

Given two sets $X$ and $Y$, a correspondence between them is any relation $R\subseteq X\times Y$ such that $\pi_X(R)=X$ and $\pi_Y(R)=Y$ where $\pi_X:X\times Y\rightarrow X$ and $\pi_Y:X\times Y\rightarrow Y$ are the canonical projections. Given two bounded metric spaces $(X,d_X)$ and $(Y,d_Y)$, and any non-empty relation $R\subseteq X\times Y$, its distortion is defined as $$\dis(R):=\sup_{(x,y),(x',y')\in R}\big|d_X(x,x')-d_Y(y,y')\big|.$$

\begin{remark}\label{rem:dgh-dis}
In particular,  the graph of any map $\psi:X\rightarrow Y$ is a relation $\mathrm{graph}(\psi)$ between $X$ and $Y$ and this relation is a correspondence whenever $\psi$ is surjective. The distortion of the relation induced by $\psi$ will be denoted by $\dis(\psi)$.
\end{remark}

A theorem of Kalton and Ostrovskii \cite{banach} proves that the Gromov-Hausdorff distance  between any two bounded metric spaces $(X,d_Y)$ and $(Y,d_Y)$ is equal to
\begin{equation}\label{eqn:dghaltdef}
    \dgh(X,Y):=\frac{1}{2}\inf_R \dis(R),
\end{equation}
where $R$ ranges over all correspondences between $X$ and $Y$. It was also observed in \cite{banach} that \begin{equation}\label{eq:dgh-functions}\dgh(X,Y) = \frac{1}{2}\inf_{\varphi,\psi}\max\big\{\dis(\varphi),\dis(\psi),\codis(\varphi,\psi)\big\},
\end{equation} 
where $\varphi:X\rightarrow Y$ and $\psi:Y\rightarrow X$ are any (not necessarily continuous) maps, and $$\codis(\varphi,\psi):=\sup_{x\in X, y\in Y}\big|d_X(x,\psi(y))-d_Y(\varphi(x),y)\big|$$ is the \emph{codistortion} of the pair $(\varphi,\psi).$

\subsection*{Known results on  $\dgh(\Sp^m,\Sp^n)$.}
We will find it useful to refer to the infinite matrix $\mathfrak{g}$ such that for $m,n\in\overline{\N}:=\N\cup\{\infty\}$,  $$\mathfrak{g}_{m,n}:=\dgh(\Sp^m,\Sp^n);$$ see Figure \ref{fig:g}.

The following lower bound for $\mathfrak{g}_{m,n}$, obtained via simple estimates for covering and packing numbers based on volumes of  balls, is in the same spirit as a result by Colding, \cite[Lemma 5.10]{colding}.\footnote{Funano used a similar idea in \cite{funano-estimates} to estimate Gromov's box distance between metric measure space representations of spheres.}  By $v_n(\rho)$ we denote the \emph{normalized volume}  of an open ball of radius $\rho\in(0,\pi]$ on $\Sp^n$ (so that the entire sphere has volume $1$). Colding's approach yields:

\begin{proposition}\label{prop:colding}
For all integers $0<m<n$, we have
$$\dgh(\Sp^m,\Sp^n)\geq\mu_{m,n}:=\frac{1}{2} \sup_{\rho\in(0,\pi]}\left(v_n^{-1}\circ v_m\left(\frac{\rho}{2}\right)-\rho\right).$$
\end{proposition}

We relegate the proof of this proposition to \S\ref{sec:general}. 

    \begin{example}[Lower bound for $\mathfrak{g}_{1,2}$ via Colding's idea]\label{ex:s1-s2}
In this case, $m=1$ and $n=2$, the lower bound provided by Proposition \ref{prop:colding} above is  $\sup_{\rho\in (0,\pi]}\big(\arccos(1-\frac{\rho}{\pi})-\rho\big)$, which is approximately equal to and bounded below by $0.1605$. Thus,  $\mathfrak{g}_{1,2}\geq 0.0802$. See Remark \ref{rem:new-comp} for a comparison with a new lower bound which also arises from covering/packing arguments via the Lyusternik-Schnirelmann theorem.
\end{example}

In contrast, in this paper, via techniques which include both certain topological ideas leading to lower bounds and the precise construction of correspondences with matching (and hence optimal) distortion, we prove results which imply (see Proposition \ref{prop:ub} below) that  in particular $\mathfrak{g}_{1,2} = \frac{\pi}{3}\simeq 1.0472$ which is about 13 times larger than the value obtained by the method above. In \cite[Example 5.3]{dgh-props} the lower bound $\mathfrak{g}_{1,2}\geq \frac{\pi}{12}$ was obtained via a calculation involving Gromov's \emph{curvature sets} $\mathbf{K}_3(\Sp^1)$ and $\mathbf{K}_3(\Sp^2)$. Finally, via considerations based on Katz's precise calculation \cite{katz1983filling} of the filling radius of spheres \cite[Corollary 9.3]{vrfr} yields that $\mathfrak{g}_{1,n}\geq \frac{\pi}{6}$ for all $n\geq 2$ as well as other lower bounds for $\mathfrak{g}_{m,n}$ for general $m< n$ which are non-tight. In a related vein, in \cite{ji2021gromov} the authors determine the precise value of $\dgh([0,\lambda],\Sp^1)$ between an interval of length $\lambda>0$ and the circle (with geodesic distance).

\subsection{Overview of our results}

The diameter of a bounded metric space $(X,d_X)$ is the number $\diam(X):=\sup_{x,x'\in X}d_X(x,x').$

For $m\in\overline{\N}$ we view the $m$-dimensional sphere $$\Sp^m:=\{(x_1,\ldots,x_{m+1})\in \R^{m+1},\,x_1^2+\cdots+x_{m+1}^2=1\}$$ as a metric space  by endowing it with the geodesic distance: for any two points $x,x'\in \Sp^m$,
$$d_{\Sp^m}(x,x'):=\arccos\left(\langle x,x'\rangle\right) =2\arcsin\left(\frac{d_{\mathrm{E}}(x,x')}{2}\right)$$
where $d_{\mathrm{E}}$ denotes the canonical Euclidean metric inherited from $\R^{m+1}$.

Note that for $m=0$ this definition yields that $\Sp^0$ consists of two points at distance $\pi$, and that $\Sp^\infty$ is the unit sphere in $\ell^2$ with distance given in the expression above.

\begin{remark}\label{remark:ub}
First recall \cite[Chapter 7]{burago-book} that for any two  bounded metric spaces $X$ and $Y$ one always has $\dgh(X,Y)\leq \frac{1}{2}\max\{\diam(X),\diam(Y)\}.$ This means that 
\begin{equation}
\label{eq:easy-ub}
\dgh(\Sp^m,\Sp^n)\leq \frac{\pi}{2}\,\,\mbox{for all $0\leq m\leq n\leq\infty$.}  
\end{equation}

\end{remark}

We first prove the following two propositions which establish that the above upper bound is tight in certain extremal cases:

\begin{proposition}[{Distance to $\Sp^0$, \cite[Prop. 1.2]{dgh-geodesics}}]\label{prop:s0-sd}
For any integer $n\geq  1$,
$$\dgh(\Sp^0,\Sp^n)=\frac{\pi}{2}.$$
\end{proposition}

\begin{proposition}[Distance to $\Sp^\infty$]\label{prop:sinfty-sd}
For any integer $m\geq 0$,
$$\dgh(\Sp^m,\Sp^\infty)=\frac{\pi}{2}.$$
\end{proposition}
Proposition \ref{prop:s0-sd} can be proved as follows: any correspondence between $\Sp^0$ and $\Sp^n$ induces a closed cover  of $\Sp^n$ by two sets. Then, necessarily, by the Lyusternik-Schnirelmann theorem, one of these blocks must contain two antipodal points. Proposition \ref{prop:sinfty-sd} can be proved in a similar manner. See Figure \ref{fig:equilateral}. 

\begin{figure}
    \centering
    \includegraphics[width=0.5\linewidth]{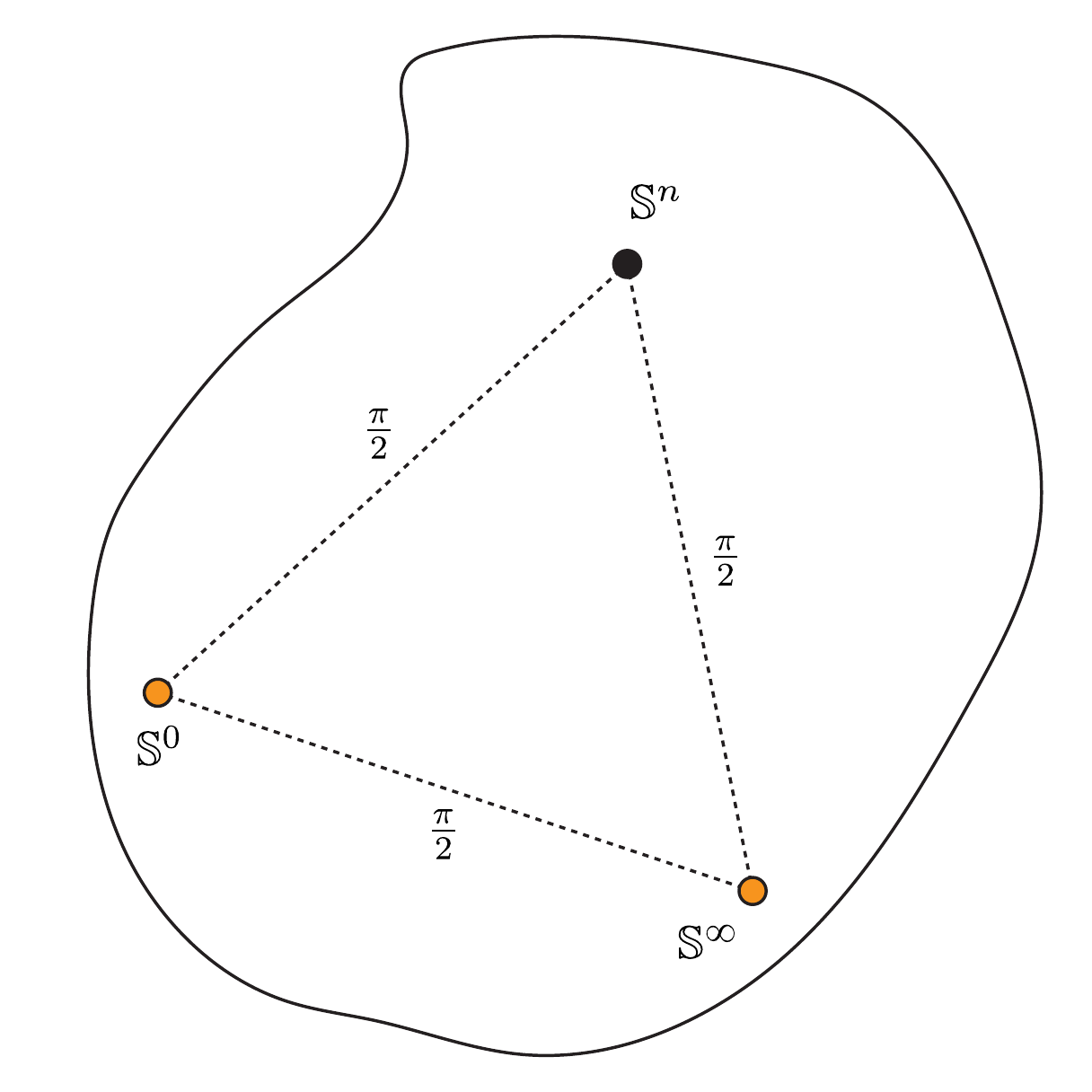}
    \caption{Propositions \ref{prop:s0-sd} and \ref{prop:sinfty-sd} encode the peculiar fact that all triangles in $(\mathcal{M}_b,\dgh)$ with vertices $\Sp^0,\Sp^\infty$, and $\Sp^n$ (for $0<n<\infty)$ are equilateral.}
    \label{fig:equilateral}
\end{figure}

\begin{remark}
When taken together, Remark \ref{remark:ub}, Propositions \ref{prop:s0-sd} and \ref{prop:sinfty-sd} above might suggest that the Gromov-Hausdorff distance between \emph{any} two spheres of different dimension is $\frac{\pi}{2}$. In fact, this is true for the following \emph{continuous} version of $\dgh$:
$$\dgh^\mathrm{cont}(X,Y):=\frac{1}{2}\inf_{\varphi',\psi'}\max\big\{\dis(\varphi'),\dis(\psi'),\codis(\varphi',\psi')\big\},$$
where $\varphi':X\rightarrow Y$ and $\psi':Y\rightarrow X'$ are \emph{continuous} maps.

Indeed, suppose that $n>m\geq 1$. Then,  the Borsuk-Ulam theorem   (cf. \cite[Theorem 1]{munkholm1969borsuk} or \cite[p. 29]{matousek2003using}), it must be that for any $\varphi':\Sp^n\rightarrow \Sp^m$ continuous there must be two antipodal points with the same image under $\varphi'$: that is, there is $x\in \Sp^n$ such that $\varphi'(x)=\varphi'(-x)$.  This implies that $\dis(\varphi')=\pi$ and consequently $\dgh^\mathrm{cont}(\Sp^n,\Sp^m)\geq \frac{\pi}{2}.$ The reverse inequality can be obtained by choosing constant maps $\varphi'$ and $\psi'$ in the above definition, thus implying that $$\dgh^\mathrm{cont}(\Sp^m,\Sp^n) = \frac{\pi}{2}.$$
\end{remark}

In contrast, we prove the following result for the standard Gromov-Hausdorff distance:
\begin{theorem}\label{thm:sn-sm-ub}
 $\dgh(\Sp^m,\Sp^n)<\frac{\pi}{2}$, for all $0<m<n<\infty$.
\end{theorem}
The Borsuk-Ulam theorem implies that, for any positive integers $n>m$ and for any given continuous function $\varphi:\Sp^n\rightarrow \Sp^m$, there exist two antipodal points in the higher dimensional sphere which are mapped to the \emph{same} point in the lower dimensional sphere. This   forces the distortion of any such continuous map to be $\pi$.  
In contrast, in order to prove  Theorem \ref{thm:sn-sm-ub}, we exhibit, for every positive numbers $m$ and $n$ with $m<n$, a  \emph{continuous antipode preserving surjection} from $\Sp^m$ to $\Sp^n$ with distortion \emph{strictly} bounded above by $\pi$, which implies the claim since the graph of any such surjection is a correspondence between $\Sp^m$ and $\Sp^n$ (cf. Remark \ref{rem:dgh-dis}). 
The proof relies on ideas related to space filling curves and spherical suspensions.

The standard Borsuk-Ulam theorem is however still useful for obtaining additional information about the Gromov-Hausdorff distance between spheres. Indeed, via Lemma \ref{lemma:ls} and the triangle inequality for $\dgh$, one can prove the following general lower bound:

\begin{proposition}\label{piS2-nm}
For any $1\leq m<n<\infty$,
$$\dgh(\Sp^m,\Sp^n)\geq \nu_{m,n}:=\frac{\pi}{2}-\mathrm{cov}_{\Sp^m}(n+1).$$
\end{proposition}

Above, for any integer $k\geq 1$, and any compact metric space $X$, $\mathrm{cov}_X(k)$ denotes the $k$-th \emph{covering radius} of $X$: 
\begin{equation}\label{eq:cov}
\mathrm{cov}_X(k):=\inf\{\dh(X,P)|\,P\subset X\,\mbox{s.t.}\,|P|\leq k\}.
\end{equation}

\begin{remark}\label{rem:new-comp} Both the lower bound $\mu_{m,n}$ from Proposition \ref{prop:colding} and $\nu_{m,n}$ from Proposition \ref{piS2-nm}  implement covering/packing ideas and as such it is interesting to compare them:
\begin{enumerate}
    \item Note that, since $\mathrm{cov}_{\Sp^1}(3)=\frac{\pi}{3}$, we have $\nu_{1,2} = \frac{\pi}{6}$ which is about 6.5 times larger than $\mu_{1,2}\approx 0.0802$ (cf. Example \ref{ex:s1-s2}).
    
    \item Computing $\nu_{m,n}$ in general requires  knowledge of the covering radius $\mathrm{cov}_{\Sp^m}(k)$ of spheres which is currently only  known for $k\leq m+2$ \cite[Theorem 3.2]{cho1997optimal}. In contrast, computing $\mu_{m,n}$ can be done (in principle) for any $m<n$ given that we have the explicit formula  $v_m(\rho)=\frac{\vol{\Sp^{m-1}}}{\vol{\Sp^m}}\int_0^\rho (\sin\theta)^{m-1}\,d\theta$ which is valid for every positive integer $m$ and $\rho\in[0,\pi]; see $\cite{tubes}. 
    
    \item The lower bound  $\mu_{m,n}$ is more widely applicable than $\nu_{m,n}$, which originates from the  Lyusternik-Schnirelman theorem (see below)  and the underlying ideas are in principle only applicable when one of the two metric spaces is a sphere.\footnote{This can be ascertained by inspecting the proof Proposition \ref{piS2-nm} in \S\ref{sec:other-lbs}.} Indeed, see \cite{colding,funano-estimates} for estimates of the Gromov-Hausdorff distance between Riemmanian manifolds satisfying upper and lower bounds on curvature obtained by combining volume comparison theorems with techniques similar to those used in proving Proposition \ref{prop:colding}.
    \item Through \cite[Theorem 3.2]{cho1997optimal} it is known that $\mathrm{cov}_{\Sp^m}(m+2)=\pi-\arccos(-1/(m+1))$ for  $m\geq 1$. Therefore, when $n=m+1$, the lower bound  $\nu_{m,m+1}$ given by Proposition \ref{piS2-nm} becomes $\arccos(-1/(m+1))-\frac{\pi}{2}$ for $m\geq 1$ which  tends to zero as $m$ goes to infinity. It is not known whether or not $\mu_{m,m+1}$ has the same behaviour.
\end{enumerate}
\end{remark}

As an immediate corollary, we obtain the following result which complements both  Proposition \ref{prop:sinfty-sd} and Theorem \ref{thm:sn-sm-ub}:
\begin{corollary}
Given any positive integer $m$ and $\epsilon>0$, there exists an integer $n = n(m,\epsilon)>m$ such that $$\dgh(\Sp^m,\Sp^n)\geq \frac{\pi}{2}-\epsilon.$$
\end{corollary}
\begin{remark}
For small $\epsilon>0$  one can estimate the value of $n$ above as $n=n(m,\epsilon) = O(\epsilon^{-m}).$
\end{remark}
The results above motivate the following two questions:

\begin{question}\label{question:mono}
Is it true that for fixed $m\geq 1$, $\dgh(\Sp^m,\Sp^n)$ is non-decreasing for all $n\geq m$?
\end{question}

\begin{figure}
    \centering
    \includegraphics[width=0.8\linewidth]{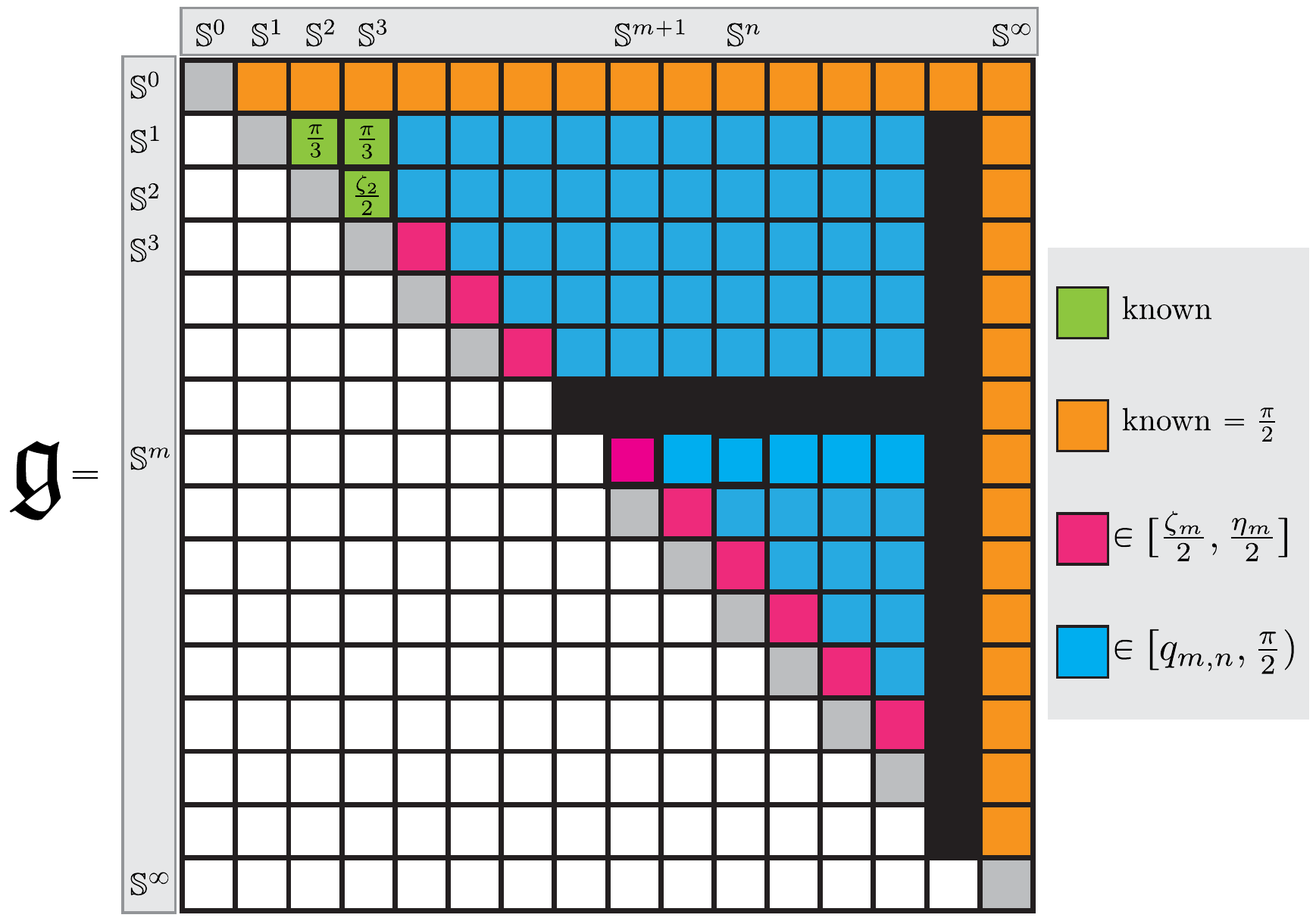}
    \caption{\textbf{The matrix $\mathfrak{g}$ such that $\mathfrak{g}_{m,n}:=\dgh(\Sp^m,\Sp^n).$}
    According to Remark \ref{remark:ub} and Corollary \ref{cor:sn-sm-univlb}, all non-zero entries of the matrix $\mathfrak{g}$ are in the range $[\frac{\pi}{4},\frac{\pi}{2}]$. In the figure, $\zeta_m = \arccos\left(\frac{-1}{m+1}\right)$ is the edge length of the regular geodesic simplex inscribed in $\Sp^m$, $\eta_m$ is the diameter of a face of the regular geodesic simplex in $\Sp^m$ (see equation (\ref{eq:etam})), and $q_{m,n} = \max\left\{\frac{\zeta_m}{2},\frac{\pi}{2}-\mathrm{cov}_{\Sp^m}(n+1)\right\}$.}
    \label{fig:g}
\end{figure}

\begin{question}
Fix $m\geq 1$ and $\epsilon>0$. Find (optimal) estimates for:
$$k_m(\epsilon):=\inf\left\{k\geq 1\big|\, \dgh(\Sp^m,\Sp^{m+k})\geq \frac{\pi}{2}-\epsilon\right\}.$$ 
\end{question}

Via the Lyusternik-Schnirelmann theorem, Proposition \ref{piS2-nm} above depends on the classical Borsuk-Ulam theorem which, in one of its guises \cite[Theorem 2.1.1]{matousek2003using}, states that there is no \emph{continuous} antipode preserving map $g:\Sp^n\rightarrow \Sp^{n-1}$. As a consequence, if  $g:\Sp^n\rightarrow \Sp^{n-1}$ is any antipode preserving map as above, then $g$ cannot be continuous. 
A natural question is \emph{how discontinuous} is any such $g$  forced to be. This question was actually tackled in 1981 by Dubins and Schwarz \cite{dubins1981equidiscontinuity} who proved that the \emph{modulus of discontinuity} $\delta(g)$ of any such $g$ needs to be suitably bounded below. These results are instrumental for proving Theorem \ref{thm:sn-sm-lb-DS} below; see \S\ref{sec:lb} and Appendix \ref{app:proof-BU} for details and for a concise proof of the main theorem from \cite{dubins1981equidiscontinuity} (following a strategy outlined by Matou\v{s}ek in  \cite{matousek2003using}).

For each $m\in\N$ let $\zeta_m$ denote the edge length (with respect to the geodesic distance) of a regular  $m+1$ simplex inscribed in $\Sp^m$:
$$\zeta_m := \arccos\left(\frac{-1}{m+1}\right),$$
which is monotonically decreasing in $m$. For example $\zeta_0 = \pi$, $\zeta_1 = \frac{2\pi}{3}$, $\zeta_2 = \arccos\big(\frac{-1}{3}\big)\simeq 0.608 \pi,$ and $\lim_{m\rightarrow\infty}\zeta_m=\frac{\pi}{2}.$  Then, we have the following lower bound which will turn out to be optimal in some cases:

\begin{theorem}[Lower bound via geodesic simplices]\label{thm:sn-sm-lb-DS}
For all integers $0<m<n$, $$\dgh(\Sp^m,\Sp^n)\geq \frac{1}{2}\zeta_m.$$
Moreover, for any map $\varphi:\Sp^n\rightarrow\Sp^m$, we have that $\dis(\varphi)\geq \zeta_m.$
\end{theorem}

From the above, we have the following general lower bound:
\begin{corollary}\label{cor:sn-sm-univlb}
For all integers $0<m<n$, $\dgh(\Sp^m,\Sp^n)\geq \frac{\pi}{4}.$
\end{corollary}
This corollary  of course implies that the sequence of compact metric spaces $(\Sp^n)_{n\in\overline{\N}}$ is not Cauchy.

\begin{remark}
Note that the lower bound for $\dgh(\Sp^m,\Sp^n)$ given by Theorem \ref{thm:sn-sm-lb-DS} coincides with the \emph{filling radius} of $\Sp^m$; cf. \cite[Theorem 2]{katz1983filling}.  This lower bound is twice the one obtained via the  stability of Vietoris-Rips persistent homology \cite[Corollary 9.3]{vrfr}.
\end{remark}

Note that $\mathrm{cov}_{\Sp^1}(k) \leq  \frac{\pi}{k}$, which can be seen by considering the vertices of a regular polygon inscribed in $\Sp^1$ with $k$ sides.  Combining this fact with Proposition \ref{piS2-nm}, Theorem \ref{thm:sn-sm-lb-DS}, and the fact that $\zeta_1=\frac{2\pi}{3}$ one obtains the following special claim for the entries in the first row of the matrix $\mathfrak{g}$:

\begin{corollary}\label{coro:s1-sn}
For all $n>1$, $\dgh(\Sp^1,\Sp^n)\geq \pi\cdot \max\big\{\frac{1}{3},\frac{1}{2}\frac{n-1}{n+1}\big\}.$
\end{corollary}

\begin{remark}\label{rem:s1-sn}
Notice that this implies that, whereas $\dgh(\Sp^1,\Sp^n)\geq \frac{\pi}{3}$ for $n\in\{2,3,4,5\}$, one has the larger lower bound $\dgh(\Sp^1,\Sp^6)\geq \frac{5\pi}{14}>\frac{\pi}{3}.$ Propositions \ref{prop:ub} and \ref{prop:ub13} below establish that actually $\dgh(\Sp^1,\Sp^2) = \dgh(\Sp^1,\Sp^3) = \frac{\pi}{3}.$
\end{remark}

Finally, in order to prove that $\dgh(\Sp^1,\Sp^2)=\frac{\pi}{3}$, we combine Theorem \ref{thm:sn-sm-lb-DS} with an explicit construction of a correspondence between $\Sp^1$ and $\Sp^2$ as follows. Let $\mathbf{H}_{\geq0}(\Sp^2)$ denote the closed upper   hemisphere of $\Sp^2$. Then, the following proposition shows that there exists a correspondence between $\Sp^1$ and $\mathbf{H}_{\geq0}(\Sp^2)$ with distortion at most $\frac{2\pi}{3}$. A correspondence between $\Sp^1$ and $\Sp^2$ (see Figure \ref{fig:phi21}) with the same distortion is then obtained via a certain \emph{odd} (i.e. antipode preserving) extension of the aforementioned correspondence (cf. Lemma \ref{lemma:distortion}): 
\begin{restatable}{proposition}{ub}\label{prop:ub}
There exists (1) a correspondence between $\Sp^1$ and $\mathbf{H}_{\geq0}(\Sp^2)$, and (2) a correspondence between $\Sp^1$ and $\Sp^2$, both of which have distortion at most $\frac{2\pi}{3}$. In particular, together with Theorem \ref{thm:sn-sm-lb-DS}, this implies $\dgh(\Sp^1,\Sp^2) = \frac{\pi}{3}$.
\end{restatable}

Even though we do not state it explicitly, in a manner similar to Proposition \ref{prop:ub}, all correspondences constructed in Propositions \ref{prop:ub13}, \ref{prop:ub23} and \ref{prop:ub-n-m} below also arise from odd extensions of correspondences between the lower dimensional sphere and the upper hemisphere of the larger dimensional sphere (cf. their respective proofs).

\begin{remark}\label{rem:compareS1HS^2} 
Also, by combining the first claim of Proposition \ref{prop:ub} and item (4) of Example \ref{example:rstofstrongthm} below (which is analogous to the claim of Theorem \ref{thm:sn-sm-lb-DS} but tailored to the case of $\Sp^m$ versus $\mathbf{H}_{\geq 0}(\Sp^m)$), one  concludes that $\dgh(\Sp^1,\mathbf{H}_{\geq 0}(\Sp^2))=\frac{1}{2}\zeta_1=\frac{\pi}{3}$.
\end{remark}

Via a construction somewhat reminiscent of the Hopf fibration, we prove that there exists a correspondence between the 3-dimensional sphere and the 1-dimensional sphere with distortion at most $\frac{2\pi}{3}$. By applying suitable rotations in $\R^4$, the proof of the following proposition extends the (a posteriori) optimal correspondence between $\Sp^1$ and $\Sp^2$ constructed in the proof of Proposition \ref{prop:ub} (see Figure \ref{fig:hopf}):

\begin{restatable}{proposition}{ub13}\label{prop:ub13}
There exists a correspondence between $\Sp^1$ and $\Sp^3$ with distortion at most $\frac{2\pi}{3}$. In particular, together with Theorem \ref{thm:sn-sm-lb-DS}, this implies $\dgh(\Sp^1,\Sp^3) = \frac{\pi}{3}$.
\end{restatable}

Finally, we were able to compute the exact value of the distance between $\Sp^2$ and $\Sp^3$ by producing a correspondence whose distortion matches the one implied by the lower bound in Theorem \ref{thm:sn-sm-lb-DS}. This correspondence is structurally different from the ones constructed in Propositions \ref{prop:ub} and \ref{prop:ub13} and arises by partitioning $\Sp^3$ into 32 regions whose  diameter is (necessarily) bounded above by $\zeta_2$ and also satisfy suitable pairwise constraints (cf. \S\ref{sec:gen-corresp}):

\begin{restatable}{proposition}{ub23}\label{prop:ub23}
There exists a correspondence between $\Sp^2$ and $\Sp^3$ with distortion at most $\zeta_2$. In particular, together with Theorem \ref{thm:sn-sm-lb-DS}, this implies $\dgh(\Sp^2,\Sp^3) = \frac{1}{2}\zeta_2$.
\end{restatable}

Keeping in mind Remark \ref{rem:s1-sn} and Propositions \ref{prop:ub} and \ref{prop:ub13} we pose the following:
\begin{question} \label{question:s1-sn}
Is it true that $\dgh(\Sp^1,\Sp^n)=\frac{\pi}{3}$ for $n\in\{4,5\}$?
\end{question}

Theorem \ref{thm:sn-sm-lb-DS} and Propositions \ref{prop:ub} and \ref{prop:ub23} lead to formulating the following conjecture: 

\begin{conjecture}\label{conj:eq}
For all $m\in\N$, $\dgh(\Sp^m,\Sp^{m+1})=\frac{1}{2}\zeta_m$.
\end{conjecture}
Note that when $m=1$ and $m=2$, Conjecture \ref{conj:eq} reduces to Propositions \ref{prop:ub} and \ref{prop:ub23}, respectively. Moreover, the conjecture would imply that $\lim_{m\rightarrow\infty}\dgh(\Sp^m,\Sp^{m+1})=\frac{\pi}{4}$.

While trying to prove Conjecture \ref{conj:eq}, we were able to prove the following weaker result via an explicit construction of a certain correspondence  generalizing the one constructed in the proof of Proposition \ref{prop:ub}:

\begin{proposition}\label{prop:ub-n-m}
For any positive integer $m>0$, there exists a correspondence between $\Sp^m$ and $\Sp^{m+1}$ with distortion at most $\eta_m$, where 
\begin{equation}\label{eq:etam}
\eta_m:=\begin{cases}\arccos\left(-\frac{m+1}{m+3}\right)&\text{for }m\text{ odd}\\\arccos\left(-\sqrt{\frac{m}{m+4}}\right)&\text{for }m\text{ even}.\end{cases}\end{equation}
Here, $\eta_m$ is the diameter of a face of the regular geodesic  $m$-simplex in $\Sp^m$; see Figure \ref{fig:eta-m} and the discussion in \S\ref{sec:proof-gen-ub}.
\end{proposition}

This correspondence arises from a partition of $\Sp^{m+1}$ into $2(m+2)$ regions which are induced by two antipodal regular simplices inscribed in $\Sp^m$, the equator of $\Sp^{m+1}$ (see Figure \ref{fig:phi21} for the case $m=1$, a case in which this correspondence turns out to be optimal). 

\begin{corollary}\label{coro:general-m}
For any positive integer $m>0$, $\dgh(\Sp^m,\Sp^{m+1}) \leq\frac{1}{2}\eta_m.$
\end{corollary}

\begin{remark}
Note that $\eta_m\geq\zeta_m$ for any $m>0$ and the equality holds for $m=1$, namely: $\eta_1 = \zeta_1$, so Proposition \ref{prop:ub-n-m} generalizes Proposition \ref{prop:ub}. However, by Proposition \ref{prop:ub23} we see that, since $1.9106\approx \zeta_2 < \eta_2 \approx    2.1863$, Corollary \ref{coro:general-m} is not tight when $m=2$.  Also, since $\eta_m < \pi$, Corollary \ref{coro:general-m} gives a quantitative version of the claim in Theorem \ref{thm:sn-sm-ub} when $n=m+1$. 
\end{remark}

\begin{remark}
Note that by combining Theorem \ref{thm:sn-sm-lb-DS} and Proposition \ref{piS2-nm} we obtain a generalization of the bound given in Corollary \ref{coro:s1-sn}: for all $1\leq m < n$,
\begin{equation}\label{eq:qmn}
    \dgh(\Sp^m,\Sp^n)\geq \max\left\{\frac{\zeta_m}{2},\frac{\pi}{2}-\mathrm{cov}_{\Sp^m}(n+1)\right\} =: q_{m,n}.\end{equation}
\end{remark}

\begin{question}
Formula (\ref{eq:qmn}) and Remark \ref{rem:s1-sn} motivate the following question: For $m\geq 1$ large, find the \emph{rate} at which the number\footnote{Note that $\zeta_m = \pi - \arccos\big(\frac{1}{m+1}\big)$.}
$$n_\mathrm{diag}(m):=\max\left\{n>m\bigg|\,\mathrm{cov}_{\Sp^m}(n+1)\geq \frac{1}{2}\arccos\left(\frac{1}{m+1}\right)\,\right\}$$
grows with $m$. The reason for the notation $n_\mathrm{diag}(m)$ is that this number provides an estimate for a band around the principal diagonal of the matrix $\mathfrak{g}$ (see Figure \ref{fig:g}) inside of which one would hope to prove that $$\dgh(\Sp^m,\Sp^n) = \frac{\zeta_m}{2} \,\,\,\mbox{for all $n\in\{m+1,\ldots,n_\mathrm{diag}(m)$}\}.$$
\end{question}

\subsection{Additional results and questions.}
Besides what we have described so far, the paper includes a number of
other results about Gromov-Hausdorff distances between spaces closely related to spheres.

\subsubsection{Spheres with Euclidean distance} \label{sec:euclidean} Some of the ideas described above (for spheres with geodesic distance) can be easily adapted to provide bounds for the distance between half spheres with geodesic distance, and between spheres with Euclidean distance. However, there is evidence that this phenomenon is subtle and to the best of our knowledge, there is no complete translation between the geodesic and Euclidean cases. This is exemplified by the following.

Let $\Sp^n_{\mathrm{E}}$ denote the unit sphere with the Euclidean metric $d_{\mathrm{E}}$ inherited from $\R^{n+1}$. Then, via Remark \ref{rem:compareS1HS^2} and item (2) of Corollary \ref{cor:EucGeo} (which provides a bridge between geodesic distortion and Euclidean distortion via the $\sin(\cdot)$ function) we have that $$\dgh\big(\Sp^1_\mathrm{E},\mathbf{H}_{\geq 0}(\Sp^2_\mathrm{E})\big)\leq \sin\big(\dgh(\Sp^1,\mathbf{H}_{\geq 0}(\Sp^2)\big) = \frac{\sqrt{3}}{2}.$$ Despite this, in Proposition \ref{prop:dgh-s12-s2pe} we were able to construct a correspondence between these two spaces with distortion \emph{strictly smaller} than $\sqrt{3}$. This suggests that Euclidean analogues of Theorem \ref{thm:sn-sm-lb-DS} may not be direct consequences; see  \S\ref{sec:dgh-eucl} for other related results.

This motivates posing the following question:  

\begin{question}
Determine $\mathfrak{g}^\mathrm{E}_{m,n}:=\dgh(\Sp^m_{\mathrm{E}},\Sp^n_{\mathrm{E}})$ for all integers $1\leq m<n$.
\end{question}

It should however be noted that by  Corollary \ref{cor:EucGeo} we have 
$\mathfrak{g}_{m,n}^\mathrm{E}\leq \sin\big(\mathfrak{g}_{m,n}\big)$, which renders Proposition \ref{prop:ub-n-m} immediately applicable, yielding $\mathfrak{g}_{m,m+1}^\mathrm{E}\leq \sin\left(\frac{\eta_m}{2}\right)$.

\subsubsection{A stronger version of Theorems \ref{thm:sn-sm-lb-DS}}

By inspecting the proof of Theorems \ref{thm:sn-sm-lb-DS}, we actually have Theorem \ref{thm:stronggeneralization} which subsumes these results in a much greater degree of generality. Indeed, via this theorem one can obtain the following estimates:

\begin{example}\label{example:rstofstrongthm}
The following lower bounds hold:
\begin{enumerate}

    \item $\dgh([0,\pi],\Sp^n)\geq \frac{\pi}{3}$ for any $n\geq 2$.

    \item $\dgh(\Sp^1, \Sp^2\times \cdots \times \Sp^2 )\geq \frac{\pi}{3}$ for any number of $\Sp^2$ factors.

    \item $\dgh(\Sp^m,\mathbf{H}_{\geq 0}(\Sp^n))\geq \frac{1}{2}\zeta_m$ whenever $0<m<n<\infty$.
    
    \item $\dgh\big(\mathbf{H}_{\geq 0}(\Sp^m),\mathbf{H}_{\geq 0}(\Sp^n)\big)\geq\frac{1}{2}\zeta_m$ whenever $0<m<n<\infty$.

    \item $\dgh(P,\Sp^2)\geq \frac{\pi}{3}$ for any finite $P\subset \Sp^1$. Compare to the $\frac{\pi}{2}$ lower bound given in Lemma \ref{lemma:ls}.

    \item $\dgh(P_3,\seth_{\geq0}(\Sp^2))=\frac{\pi}{3}$ where $P_3$ is the 3 point metric space with all interpoint distances equal to  $\frac{2\pi}{3}$. Also $\dgh(P_6,\Sp^2)=\frac{\pi}{3}$, where $P_6$ is the six point metric space corresponding to a regular hexagon inscribed in $\Sp^1$. These are consequences of item (5) and  small modifications of the correspondences constructed in Proposition \ref{prop:ub}.
    
\end{enumerate}
\end{example}

\begin{theorem}\label{thm:stronggeneralization}
Let bounded metric spaces $X$ and $Y$ be such that for some positive integer $m$: (i)  $X$ can be isometrically embedded into $\Sp^m$ and (ii) $\seth_{\geq 0}(\Sp^{m+1})$ can be isometrically embedded into $Y$.
Then, 

\begin{enumerate}
    \item $\dgh(X,Y)\geq\frac{1}{2}\zeta_m$. 
    
    \item Moreover,  $\dis(\phi)\geq\zeta_m$ for any map $\phi:Y\longrightarrow X$.
\end{enumerate}
\end{theorem}

\begin{remark}
Item (1) in Example \ref{example:rstofstrongthm} also holds for $n=1$ albeit this is not implied by Theorem \ref{thm:stronggeneralization}. The fact that $\dgh([0,\pi],\Sp^1)\geq \frac{\pi}{3}$  follows from \cite[Theorem 4.10]{ji2021gromov} and it also follows from the proof of \cite[Lemma 2.3]{katz2020torus}; see Appendix \ref{app:katz-lb}. 
\end{remark}

\subsection{Organization}
In \S\ref{sec:preliminaries} we review some preliminaries.

The proof of Proposition \ref{prop:colding} on a lower bound for $\mathfrak{g}_{m,n}$ involving the normalized volume of open balls is given in \S\ref{sec:colding}, whereas those of Propositions \ref{prop:s0-sd} (establishing the precise value of $\mathfrak{g}_{0,n}$), \ref{prop:sinfty-sd} (establishing the precise value of $\mathfrak{g}_{m,\infty}$), and \ref{piS2-nm} (on a lower bound for $\mathfrak{g}_{m,n}$ involving the covering radius) are given in \S\ref{sec:other-lbs}.

The proof of Theorem \ref{thm:sn-sm-ub} establishing that $\mathfrak{g}_{m,n}<\frac{\pi}{2}$ (for any $0<m<n<\infty$) is given in \S\ref{sec:proof-univ-ub} whereas that of Theorem \ref{thm:sn-sm-lb-DS} on a lower bound for $\mathfrak{g}_{m,n}$ deduced from a discontinuous version of the Borsuk-Ulam theorem and Theorem \ref{thm:stronggeneralization} (a generalization of Theorem \ref{thm:sn-sm-lb-DS}) are given in \S\ref{sec:lb}.

The proofs of Propositions \ref{prop:ub} establishing the precise value of $\mathfrak{g}_{1,2}$ and  \ref{prop:ub-n-m} on an upper bound involving the diameter of a face of a geodesic simplex are given in \S\ref{sec:proofs-ub}.

Proposition \ref{prop:ub13} establishing the precise value of $\mathfrak{g}_{1,3}$ is proved in \S\ref{sec:S1S3uppbdd} whereas Proposition \ref{prop:ub23} establishing the precise value of $\mathfrak{g}_{2,3}$ is proved in \S\ref{sec:proof-prop-ub23}.

The case of spheres with Euclidean distance  is discussed  in \S\ref{sec:dgh-eucl}.

Finally, this paper has four appendices. Appendix \S\ref{app:proof-BU} provides a succinct and self contained proof of the version of Borsuk-Ulam's theorem due to Dubins and Schwarz \cite{dubins1981equidiscontinuity} which is instrumental for proving Theorem  \ref{thm:sn-sm-lb-DS} and related results. Appendix \S\ref{app:katz-lb} establishes that the Gromov-Hausdorff distance between the $n$-dimensional sphere and an interval is always bounded below by $\frac{\pi}{3}$ whereas Appendix \S\ref{sec:polys} provides some results about the Gromov-Hausdorff distance between regular polygons. Finally, Appendix \S\ref{sec:alternative-method} constructs an alternative optimal correspondence between $\Sp^1$ and $\Sp^2$.

Computational experiments of this project are described in \cite{dgh-spheres-github}.

\subsection{Acknowledgements}
We are grateful to Henry Adams for helpful conversations related to this work. We thank 
Gunnar Carlsson and Tigran Ishkhanov for encouraging F.M. to tackle the question about the Gromov-Hausdorff distance between spheres via topological methods.
This work was supported by NSF grants DMS-1547357, CCF-1526513, IIS-1422400, and  CCF-1740761.

\section{Preliminaries}\label{sec:preliminaries}

Given a metric space $(X,d_X)$ and $\delta>0$, a $\delta$-net for $X$ is any $A\subset X$ such that for all $x\in X$ there exists $a\in A$ with $d_X(x,a)\leq \delta$. The diameter of $X$ is $\diam(X):=\sup_{x,x'\in X} d_X(x,x').$

Recall \cite[Chapter 2]{burago-book} that complete metric space $(X,d_X)$ is a \emph{geodesic space} if and only if it admits midpoints: for all $x,x'\in X$ there exists $z\in X$ such that $$d_X(x,z)=d_X(x',z)=\frac{1}{2}d_X(x,x').$$

We henceforth use the symbol $\ast$ to denote the one point metric space. It is easy to check that $\dgh(\ast,X)=\frac{1}{2}\diam(X)$ for any bounded metric space $X$. From this, and the triangle inequality for the Gromov-Hausdorff distance, it then follows that for all bounded metric spaces $X$ and $Y$,
\begin{equation}\label{eq:easy}
\dgh(X,Y)\geq \frac{1}{2}\big|\diam(X)-\diam(Y)\big|.
\end{equation}

\subsection{Notation and conventions about spheres.} \label{sec:conventions}
Finally, let us collect and introduce  important notation and conventions which will be used throughout this paper (except for \S\ref{sec:S1S3uppbdd}). For each nonnegative integer $m\in\N$,

\begin{itemize}
    \item $\Sp^m:=\{(x_1,\dots,x_{m+1})\in\R^{m+1}:x_1^2+\dots+x_{m+1}^2=1\}$ ($m$-sphere).
    \item $\mathbf{H}_{\geq 0}(\Sp^m):=\{(x_1,\dots,x_{m+1})\in \Sp^m: x_{m+1}\geq 0\}$ (closed upper hemisphere).
    \item $\mathbf{H}_{>0}(\Sp^m):=\{(x_1,\dots,x_{m+1})\in \Sp^m: x_{m+1}> 0\}$ (open upper hemisphere).
    \item $\mathbf{H}_{\leq 0}(\Sp^m):=\{(x_1,\dots,x_{m+1})\in \Sp^m: x_{m+1}\leq 0\}$ (closed lower hemisphere).
    \item $\mathbf{H}_{<0}(\Sp^m):=\{(x_1,\dots,x_{m+1})\in \Sp^m: x_{m+1}< 0\}$ (open lower hemisphere).
    \item $\sete(\Sp^m):=\{(x_1,\dots,x_{m+1})\in \Sp^m: x_{m+1}=0\}$ (equator of sphere).
    \item $\mathbb{B}^{m+1}:=\{(x_1,\dots,x_{m+1})\in\R^{m+1} : x_1^2+\dots+x_{m+1}^2\leq1\}$ (unit closed ball).
    \item $\widehat{\mathbb{B}}^{m+1}:=\{(x_1,\dots,x_{m+1})\in\R^{m+1} : \vert x_1 \vert+\dots+\vert x_{m+1} \vert\leq1\}$ (unit cross-polytope).
\end{itemize}

Also, $\Sp^m,\mathbf{H}_{\geq 0}(\Sp^m),\mathbf{H}_{> 0}(\Sp^m),\mathbf{H}_{\leq 0}(\Sp^m),\mathbf{H}_{<0}(\Sp^m)$ and $\sete(\Sp^m)$ are all equipped with the geodesic metric $d_{\Sp^m}$. Observe that $\Sp^m$ and $\sete(\Sp^{m+1})$ are isometric. We will denote by
\begin{align}
    \iota_m:\Sp^m&\longrightarrow\Sp^{m+1}\label{eq:iota-m}\\
    (x_1,\dots,x_{m+1})&\longmapsto (x_1,\dots,x_{m+1},0)\nonumber
\end{align}
 the canonical isometric embedding from $\Sp^m$ into $\Sp^{m+1}$.
 
\subsection{A general construction of  correspondences}\label{sec:gen-corresp}
Assume $X$ and $Y$ are compact metric spaces such that $X\stackrel{\phi}{\hookrightarrow} Y$ isometrically, e.g. $\Sp^m\hookrightarrow \Sp^n$ for $m\leq n$.

As mentioned in Remark \ref{rem:dgh-dis} any surjection $\psi:Y\twoheadrightarrow X$ gives rise to a correspondence between $X$ and $Y$. The following simple construction of such a $\psi$ will be used throughout this paper. Given $k\in\N$, assume  $P_k=\{B_1,\ldots,B_i,\ldots,B_k\}$ is any partition of $Y\backslash \phi(X)$ and 
$\mathbb{X}_k=\{x_1,\ldots,x_i,\ldots,x_k\}$ are any $k$ points in $X$. Then, define $\psi:Y\twoheadrightarrow X$ by 
$\psi|_{\phi(X)} := \phi^{-1}$ and $\psi|_{B_i} := x_i$ for each $1\leq i \leq k$. It then follows that the distortion of this correspondence is:

$$\dis(\psi)=\max\{A,B,C\}$$
where
\begin{itemize}
    \item $A:=\max\limits_i\diam(B_i)$,
    
    \item $B:=\max\limits_{i\neq j}\max\limits_{\substack{y\in B_i\\y'\in B_j}}\big|d_X(x_i,x_j)-d_Y(y,y')\big|$, and
    
    \item $C:=\max\limits_i\max\limits_{\substack{x\in X\\y\in B_i}}\big|d_X(x,x_i) - d_Y(\phi(x),y)\big|$.
\end{itemize}

This pattern will be used several times in this paper.

\section{Some general lower bounds}\label{sec:general}

\subsection{The proof of Proposition \ref{prop:colding}}\label{sec:colding}

For a metric space $X$ and $\rho>0$, let $N_X(\rho)$ denote the minimal number of open balls of radius $\rho$ needed to cover $X$. Also, let $C_X(\rho)$ denote the maximal number of pairwise disjoint open balls of radius $\frac{\rho}{2}$ that can be placed in $X$. $N_X$ and $C_X$ are usually referred to as the \emph{covering number} and the \emph{packing number}, respectively.

Note that the covering radius $\mathrm{cov}_X$ (cf. equation (\ref{eq:cov})) and the covering number $N_X$ are related in the following manner: $$\mathrm{cov}_X(k)=\inf\{\rho>0:N_X(\rho)\leq k\}.$$

The following \emph{stability} property of $N_X(\cdot)$ and $C_X(\cdot)$ is classical and can be used to obtain estimates for the Gromov-Hausdorff distance between spheres:

\begin{proposition}[{\cite[pp. 299]{petersen-book}}]\label{prop:cov-cap}
If $X$ and $Y$ are metric spaces and $\dgh(X,Y)<\eta$ for some $\eta>0$, then  for all $\rho\geq 0,$
\begin{itemize}
    \item[(1)] $N_X(\rho)\geq N_Y(\rho+2\eta)$,  and
    \item[(2)] $C_X(\rho)\geq C_Y(\rho+2\eta).$
\end{itemize}
\end{proposition}

Recall that $v_n(\rho)$ is the normalized volume of an open ball or radius $\rho$ on $\Sp^n$.

\begin{proof}[Proof of Proposition \ref{prop:colding}]
 The proof that $\dgh(\Sp^m,\Sp^n)\geq\mu_{m,n}:=\frac{1}{2} \sup_{\rho\in(0,\pi]}\left(v_n^{-1}\circ v_m\left(\frac{\rho}{2}\right)-\rho\right)$ for any $0<m<n<\infty$ is by contradiction. We first state two claims that we prove at the end.

\begin{claim}\label{claim:cap}
For any $\rho> 0$ and $n\geq 1$, the packing number $\capa{\rho}{\Sp^n}\leq \big(v_{n}(\frac{\rho}{2})\big)^{-1}$.
\end{claim}

\begin{claim}\label{claim:cov}
For any $\rho> 0$ and $n\geq 1$, the covering number $\cov{\rho}{\Sp^n}$ satisfies $1\leq \cov{\rho}{\Sp^n}\cdot v_{n}(\rho)$.
\end{claim}

Assuming the claims above, suppose that $n>m\geq 1$ and $\eta:=\dgh(\Sp^m,\Sp^n)<\mu_{m,n}$. Pick $\varepsilon>0$ small enough such that $\eta+\frac{\varepsilon}{2} <\mu_{m,n}.$

Since $\dgh(\Sp^m,\Sp^n)<\eta+\frac{\varepsilon}{2}$, from Proposition \ref{prop:cov-cap}, the fact that for $N_X(\rho)\leq C_X(\rho)$ for any compact metric space $X$ and any $\rho> 0$, and Claim \ref{claim:cap} we have that 
$$\cov{\rho+2\eta+\varepsilon}{\Sp^n}\leq \cov{\rho}{\Sp^m}\leq\capa{\rho}{\Sp^m}\leq \left(v_{m}\left(\frac{\rho}{2}\right)\right)^{-1}.$$ Now, from Claim \ref{claim:cov} we obtain that for all $\rho\in[0,\pi]$
$$1\leq \cov{\rho+2\eta+\varepsilon}{\Sp^n}\cdot v_{n}(\rho+2\eta+\varepsilon)\leq \frac{v_{n}(\rho+2\eta+\varepsilon)}{v_{m}(\frac{\rho}{2})}.$$

Then,  for all $\rho\in[0,\pi]$ we must have 
$$\eta+\frac{\varepsilon}{2}\geq \frac{1}{2}\left(v_n^{-1}\circ v_m\left(\frac{\rho}{2}\right)-\rho\right).$$ Then, in particular, $\eta+\frac{\epsilon}{2}\geq \mu_{m,n}$, a contradiction.

\begin{proof}[Proof of Claim \ref{claim:cap}]
Let $k=\capa{\rho}{\Sp^n}$ and let $x_1,\ldots,x_k\in\Sp^n$ be s.t. $B(x_i,\frac{\rho}{2})\cap B(x_j,\frac{\rho}{2})=\emptyset$ for all $i\neq j$. Thus, $\bigcup_{i=1}^k B(x_i,\frac{\rho}{2})\subset \Sp^n$, and 
$$\vol{\Sp^n}\geq \volmeas{\Sp^n}{\bigcup_{i=1}^k B\left(x_i,\frac{\rho}{2}\right)} = k\cdot v_{n}\left(\frac{\rho}{2}\right)\cdot \vol{\Sp^n}.$$
\end{proof}

\begin{proof}[Proof of Claim \ref{claim:cov}]
Let $N=\cov{\rho}{\Sp^n}$ and $x_1,\ldots,x_N\in\Sp^n$ be s.t. $\bigcup_{i=1}^NB(x_i,\rho)=\Sp^n$. Then,
$$\vol{\Sp^n}\leq \volmeas{\Sp^n}{\bigcup_{i=1}^N B(x_i,\rho)} \leq  N\cdot v_{n}(\rho)\cdot \vol{\Sp^n}.$$
\end{proof}
\end{proof}

\subsection{Other lower bounds and the proofs of Propositions \ref{prop:s0-sd} and \ref{prop:sinfty-sd}}\label{sec:other-lbs}
Recall the following corollary to the Borsuk-Ulam theorem \cite{matousek2003using}:

\begin{theorem}[Lyusternik-Schnirelmann]\label{thm:ls}
Let $n\in\N$, and $\{U_1,\ldots, U_{n+1}\}$ be a closed cover of $\mathbb{S}^n$. Then there is $i_0\in\{1,\ldots,n+1\}$ such that $U_{i_0}$ contains two antipodal points.
\end{theorem}

The lemma below will be useful in what follows:

\begin{lemma}\label{lemma:ls}
For any integer $m\geq 1$ and any finite metric space $P$ with cardinality at most $m+1$ we have $\dgh(\Sp^m,P)\geq\frac{\pi}{2}.$
\end{lemma}

\begin{remark}\label{remark:sn-pnp1}
Lemma \ref{lemma:ls} and Remark \ref{remark:ub} imply that for each integer $n\geq 1$, $\dgh(\Sp^n,P) = \frac{\pi}{2}$ for any finite metric space $P$ with $|P|\leq n+1$ and $\diam(P)\leq \pi.$
\end{remark}

\begin{proof}[Proof of Lemma \ref{lemma:ls}]
 Suppose $m\geq 1$ is given. We prove that $\dgh(\Sp^m,P)\geq\frac{\pi}{2}$ for any finite set $P$ of size at most $m+1$. Assume that $R$ is an arbitrary correspondence between $\Sp^m$ and $P$.  We claim that $\dis(R)\geq \pi$ from which the proof will follow.  For each $p\in P$ let $R(p):=\{z\in \Sp^m|\,(z,p)\in R\}$. Then, $\big\{\overline{R(p)}\subseteq\Sp^m:p\in P\big\}$ is a closed cover of $\Sp^m$.  Since $|P|\leq m+1$, Theorem \ref{thm:ls} yields that for some $p_0\in P$, $\diam(R(p_0))=\pi$.  Finally, the claim follows since $\dis(R)\geq \max_{p\in P} \diam(R(p)).$ 
\end{proof}

By a refinement of the proof of Lemma \ref{lemma:ls} above one obtains:

\begin{corollary}\label{coro:sinfty-finite}
Let $R$ be any correspondence between a finite metric space $P$ and $\Sp^\infty$. Then, $\dis(R)\geq \pi.$ In particular, $\dgh(P,\Sp^\infty)\geq \frac{\pi}{2}.$
\end{corollary}
\begin{proof}
As in the proof of Lemma \ref{lemma:ls}, the correspondence $R$ induces a closed cover of $\Sp^\infty$. Thus, it induces a closed cover of any finite dimensional sphere $\Sp^{|P|-1}\subset \Sp^\infty$. The claim follows from Theorem \ref{thm:ls}.
\end{proof}

Notice that if $P$ has diameter at most $\pi$, then $\dgh(P,\Sp^\infty)=\frac{\pi}{2}$ (cf. Remark \ref{remark:ub} and Remark \ref{remark:sn-pnp1}). In  Appendix \ref{sec:polys} we consider a scenario which is thematically connected with Remark \ref{remark:sn-pnp1} and Corollary \ref{coro:sinfty-finite}, namely the determination of the Gromov-Hausdorff distance between a finite metric space and a sphere. Appendix \ref{sec:polys} fully resolves this question for the case of $\Sp^1$ and (the vertex set of) inscribed regular polygons.

By a small modification of the proof of Corollary \ref{coro:sinfty-finite}, we obtain the following stronger claim: 
\begin{proposition}\label{prop:tb-sinfty}
Let $X$ be any totally bounded metric space. Then, $\dgh(X,\Sp^\infty)\geq \frac{\pi}{2}.$
\end{proposition}
\begin{proof}
Fix any $\varepsilon>0$ and let $P_\varepsilon\subset X$ be a finite $\varepsilon$-net for $X$. Then, by the triangle inequality for $\dgh$, and Corollary \ref{coro:sinfty-finite} we have $\dgh(X,\Sp^\infty)\geq \dgh(\Sp^\infty,P_\varepsilon)-\dgh(X,P_\varepsilon)\geq \frac{\pi}{2} -\varepsilon$ which implies the claim since $\varepsilon>0$ was arbitrary. \end{proof}

\begin{proof}[Proof of Proposition \ref{prop:s0-sd}]
That $\dgh(\Sp^0,\Sp^n)=\frac{\pi}{2}$ for any integer $n\geq 1$ follows from Lemma \ref{lemma:ls} and Remark \ref{remark:ub}.
\end{proof}

\begin{proof}[Proof of Proposition \ref{prop:sinfty-sd}]
That $\dgh(\Sp^m,\Sp^\infty)=\frac{\pi}{2}$ for any nonnegative integer $m<\infty$ follows from Proposition \ref{prop:tb-sinfty} and Remark \ref{remark:ub}.
\end{proof}

\begin{proof}[Proof of Proposition \ref{piS2-nm}]
We prove that $\dgh(\Sp^m,\Sp^n)\geq\nu_{m,n}:=\frac{\pi}{2} - \mathrm{cov}_{\Sp^m}(n+1)$ for any $1\leq m<n<\infty$.

Let $P$ be any subset $\Sp^{m}$ with cardinality not exceeding $n+1$. Since the Hausdorff distance satisfies $\dh(P,\Sp^m)\geq \dgh(P,\Sp^m)$, and  by the triangle inequality for the Gromov-Hausdorff distance, we have:
$$\dh(P,\Sp^m)+\dgh(\Sp^m,\Sp^n)\geq \dgh(P,\Sp^m)+\dgh(\Sp^m,\Sp^n)\geq \dgh(P,\Sp^n)  .$$
Since $\diam(P)\leq \pi$, by Remark \ref{remark:sn-pnp1} we have that $\dgh(P,\Sp^n) = \frac{\pi}{2}.$ Hence, from the above, 
$$\dh(P,\Sp^m)+\dgh(\Sp^m,\Sp^n)\geq \frac{\pi}{2}$$
for \emph{any} $P\subset \Sp^m$ with $|P|\leq n+1.$ By the definition of  the covering radius (see equation (\ref{eq:cov})), we obtain the claim by infimizing over all possible such choices of $P$. \end{proof}

\section{The proof of Theorem \ref{thm:sn-sm-ub}}\label{sec:proof-univ-ub}

The Borsuk-Ulam theorem implies that, for any positive integers $n>m$ and for any given continuous map $\varphi:\Sp^n\rightarrow \Sp^m$, there exists two antipodal points in the higher dimensional sphere which are mapped to the same point in the lower dimensional sphere.

We now prove that, in contrast, there always exists a \emph{surjective}, antipode preserving, and continuous map $\psi_{m,n}$ from the lower dimensional sphere to the higher dimensional sphere. 

\begin{theorem}\label{thm:surj'}
For all integers $0<m<n<\infty$, there exists an \emph{antipode preserving} continuous surjection $$\psi_{m,n}:\Sp^m\longtwoheadrightarrow \Sp^n,$$
i.e., $\psi_{m,n}(-x)=-\psi_{m,n}(x)$ for every $x\in\Sp^m$. 
\end{theorem}

With this theorem, the proof of Theorem \ref{thm:sn-sm-ub}, stating that $\dgh(\Sp^m,\Sp^n)<\frac{\pi}{2}$ for all $0<m<n<\infty$, now follows:

\begin{proof}[Proof of Theorem \ref{thm:sn-sm-ub}]
Let $\psi_{m,n}:\Sp^m\longtwoheadrightarrow \Sp^n$ be the map given in Theorem \ref{thm:surj'}. Recall that the graph of a surjective map can be seen as a correspondence and let $R_{m,n}:=\mathrm{graph}(\psi_{m,n})$. In order to prove the claim, it is enough to verify that  $\dis(R_{m,n})=\dis(\psi_{m,n})<\pi$.

Since $\psi_{m,n}$ is continuous and $\Sp^m$ is compact, the supremum in the definition of distortion is a maximum:
$$\dis(\psi_{m,n}) = \max_{x,x'\in \Sp^m}\big|d_{\Sp^m}(x,x')-d_{\Sp^n}(\psi_{m,n}(x),\psi_{m,n}(x'))\big|.$$ Let $x_0,x_0'\in\Sp^m$ attain the maximum above. Note that we may assume that $x_0\neq x_0'$ for otherwise, we would have $\dgh(\Sp^m,\Sp^n)\leq\frac{1}{2}\dis(R_{m,n})= \frac{1}{2}\dis(\psi_{m,n})=0$, which would imply that $\dgh(\Sp^m,\Sp^n)=0$, i.e. that $\Sp^m$ and $\Sp^m$ are isometric, which is a contradiction since $m\neq n$.

Assume first that $x_0'\neq -x_0$. In this case,
$$\mbox{$0< d_{\Sp^m}(x_0,x_0')<\pi$ \, and \, $0\leq d_{\Sp^n}(\psi_{m,n}(x_0),\psi_{m,n}(x_0'))\leq \pi.$}$$
Thus, $$|d_{\Sp^m}(x_0,x_0')-d_{\Sp^n}(\psi_{m,n}(x_0),\psi_{m,n}(x_0'))|<\pi.$$

Assume now that $x_0'= -x_0$. In this case, $d_{\Sp^m}(x_0,x_0')=d_{\Sp^n}(\psi_{m,n}(x_0),\psi_{m,n}(x_0'))=\pi$ since $\psi_{m,n}$ is antipode preserving. Thus, in this case we also have  $$0=|d_{\Sp^m}(x_0,x_0')-d_{\Sp^n}(\psi_{m,n}(x_0),\psi_{m,n}(x_0'))|<\pi.$$
\end{proof}

\begin{remark}
Observe that the antipode preserving property of $\psi_{m,n}$ given in Theorem \ref{thm:surj'} is stronger than what we need for the purpose of proving Theorem \ref{thm:sn-sm-ub}. Indeed, all one needs is that $\psi_{m,n}(x)\neq \psi_{m,n}(-x)$  for any $x\in \Sp^m$.
\end{remark}

The goal for the rest of this section is to prove Theorem \ref{thm:surj'}.

Spherical suspensions  and space-filling curves are  key technical tools which we now review.

\subsection*{Space filling curves.} The existence of the space-filling curves is  well known \cite{peano1890courbe}:

\begin{theorem}[Space-filling curve]\label{thm:Peano}
There exist a continuous and surjective map $$H:[0,1]\longtwoheadrightarrow [0,1]^2.$$
\end{theorem}

In the sequel, we will use the notation $\mathrm{Conv}(v_1,v_2,\dots,v_d)$ to denote the convex hull of vectors $v_1,v_2,\dots,v_d$.

By resorting to space-filling curves, one can prove the following  proposition, which will be crucial in the sequel. 

\begin{figure}
\begin{center}
\includegraphics[width=0.6\linewidth]{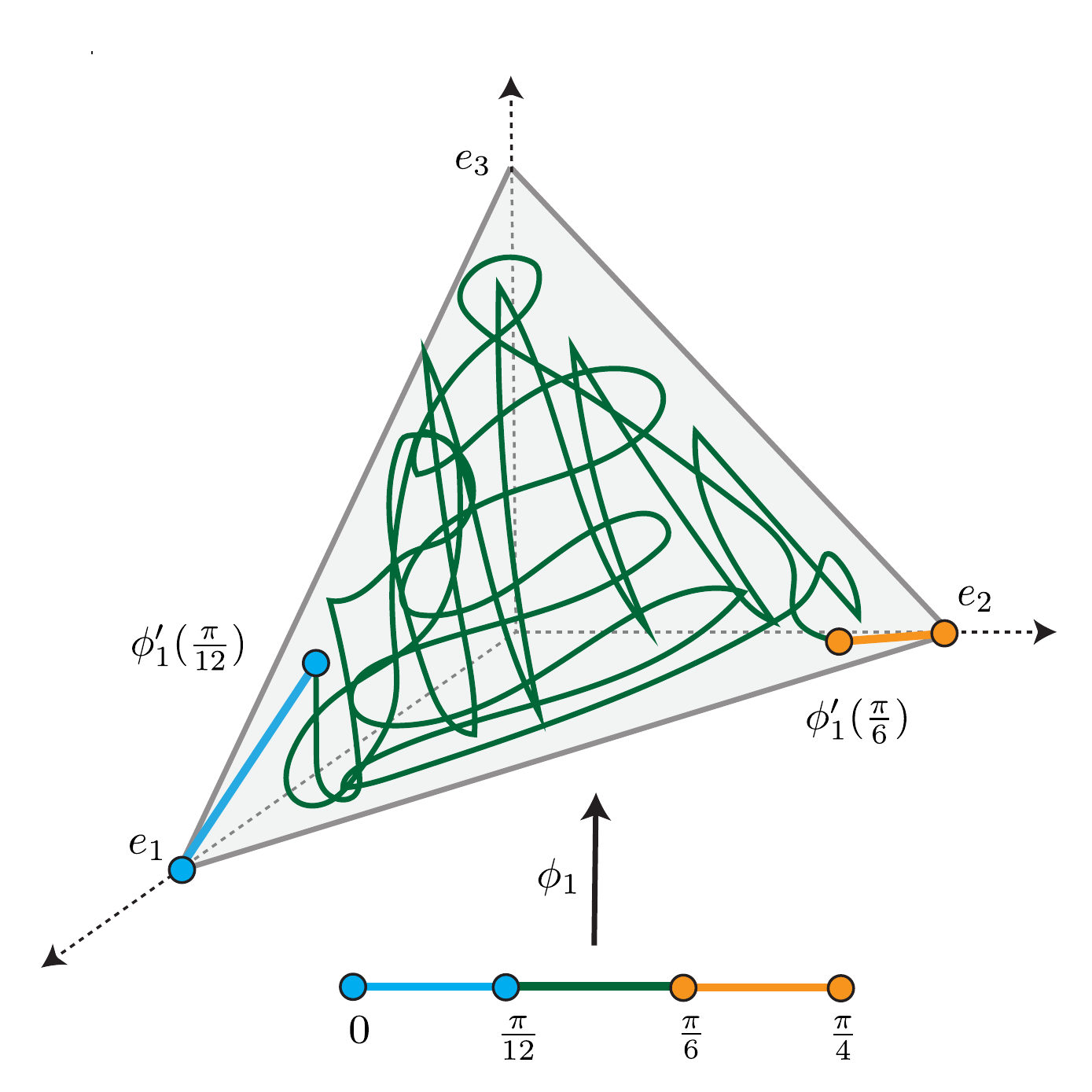}
\end{center}
\caption{The continuous surjection $\phi_1:[0,\frac{\pi}{4}]\longtwoheadrightarrow \mathrm{Conv}(e_1,e_2,e_3)$.  \label{fig:filling-curve}}
\end{figure}

\begin{proposition}\label{prop:constructionofpsi12}
There exists an antipode preserving continuous surjection $$\psi_{1,2}:\Sp^1\longtwoheadrightarrow \Sp^2.$$
\end{proposition}
\begin{proof}
Recall the definition of the $3$-dimensional cross-polytope:  $$\widehat{\mathbb{B}}^3:=\mathrm{Conv}(e_1,-e_1,e_2,-e_2,e_3,-e_3)\subset\mathbb{R}^3$$
where $e_1=(1,0,0),e_2=(0,1,0),$ and $e_3=(0,0,1)$. Then, its boundary $\partial\widehat{\mathbb{B}}^3$, which consists of eight triangles
$$\mathrm{Conv}(e_1,e_2,e_3),\mathrm{Conv}(e_1,e_2,-e_3),\dots,\mathrm{Conv}(-e_1,-e_2,-e_3)$$
is homeomorphic to $\Sp^2$.

Now, divide $\Sp^1$ into eight closed circular arcs with equal length $\frac{\pi}{4}$. In other words,  let
$$\left[0,\frac{\pi}{4}\right],\left[\frac{\pi}{4},\frac{\pi}{2}\right],\left[\frac{\pi}{2},\frac{3\pi}{4}\right],\left[\frac{3\pi}{4},\pi\right],\left[\pi,\frac{5\pi}{4}\right],\left[\frac{5\pi}{4},\frac{3\pi}{2}\right],\left[\frac{3\pi}{2},\frac{7\pi}{4}\right],\left[\frac{7\pi}{4},2\pi\right]$$
be those eight regions. Of course, we are identifying $0$ and $2\pi$ here.

Note that we are able to build a continuous and surjective map $$\phi_1:\left[0,\frac{\pi}{4}\right]\longtwoheadrightarrow\mathrm{Conv}(e_1,e_2,e_3)\text{ such that }\phi_1(0)=e_1,\phi_1\left(\frac{\pi}{4}\right)=e_2$$
as follows: Since $\mathrm{Conv}(e_1,e_2,e_3)$ is homeomorphic to $[0,1]^2$, by Theorem \ref{thm:Peano} there exists a continuous and surjective map $\phi_1'$ from $\left[\frac{\pi}{12},\frac{\pi}{6}\right]$ to $\mathrm{Conv}(e_1,e_2,e_3)$. Then, we extend its domain by using linear interpolation between $e_1$ and $\phi_1'(\frac{\pi}{12})$, and $e_2$ and $\phi_1'(\frac{\pi}{6})$ to give rise to $\phi_1$; see Figure \ref{fig:filling-curve}.

By using an analogous procedure, we construct continuous and surjective maps:
\begin{align*}
    &\phi_2:\left[\frac{\pi}{4},\frac{\pi}{2}\right]\longtwoheadrightarrow\mathrm{Conv}(-e_1,e_2,e_3)\text{ such that }\phi_2\left(\frac{\pi}{4}\right)=e_2,\phi_2\left(\frac{\pi}{2}\right)=e_3,\\
    &\phi_3:\left[\frac{\pi}{2},\frac{3\pi}{4}\right]\longtwoheadrightarrow\mathrm{Conv}(e_1,-e_2,e_3)\textrm{ such that }\phi_3\left(\frac{\pi}{2}\right)=e_3,\phi_3\left(\frac{3\pi}{4}\right)=-e_2,\\
    &\phi_4:\left[\frac{3\pi}{4},\pi\right]\longtwoheadrightarrow\mathrm{Conv}(-e_1,-e_2,e_3)\textrm{ such that }\phi_4\left(\frac{3\pi}{4}\right)=-e_2,\phi_4(\pi)=-e_1.
\end{align*}

Next, we construct the remaining continuous and surjective maps  by suitably reflecting the ones already constructed:
\begin{align*}
    &\phi_5:\left[\pi,\frac{5\pi}{4}\right]\longtwoheadrightarrow\mathrm{Conv}(-e_1,-e_2,-e_3)\textrm{ such that }\phi_5(x):=-\phi_1(-x),\\
    &\phi_6:\left[\frac{5\pi}{4},\frac{3\pi}{2}\right]\longtwoheadrightarrow\mathrm{Conv}(e_1,-e_2,-e_3)\textrm{ such that }\phi_6(x):=-\phi_2(-x),\\
    &\phi_7:\left[\frac{3\pi}{2},\frac{7\pi}{4}\right]\longtwoheadrightarrow\mathrm{Conv}(e_1,e_2,-e_3)\textrm{ such that }\phi_7(x):=-\phi_3(-x),\\
    &\phi_8:\left[\frac{7\pi}{4},2\pi\right]\longtwoheadrightarrow\mathrm{Conv}(-e_1,e_2,-e_3)\textrm{ such that }\phi_8(x):=-\phi_4(-x).
\end{align*}

Finally, by gluing all the eight maps $\phi_i$s, we  build an antipode preserving continuous and surjective map $\overline{\psi}_{1,2}:\Sp^1\longtwoheadrightarrow\partial\widehat{\mathbb{B}}^3$. Using the canonical (closest point projection) homeomorphism between $\partial\widehat{\mathbb{B}}^3$ and $\Sp^2$, we finally have the announced $\psi_{1,2}:\Sp^1\longtwoheadrightarrow\Sp^2$. It is clear from its construction that the map $\psi_{1,2}$ is continuous, surjective, and antipode preserving. Figure \ref{fig:psi-12} depicts the overall structure of the map $\psi_{1,2}$.
\end{proof}

\begin{figure}
    \centering
    \includegraphics[width=\linewidth]{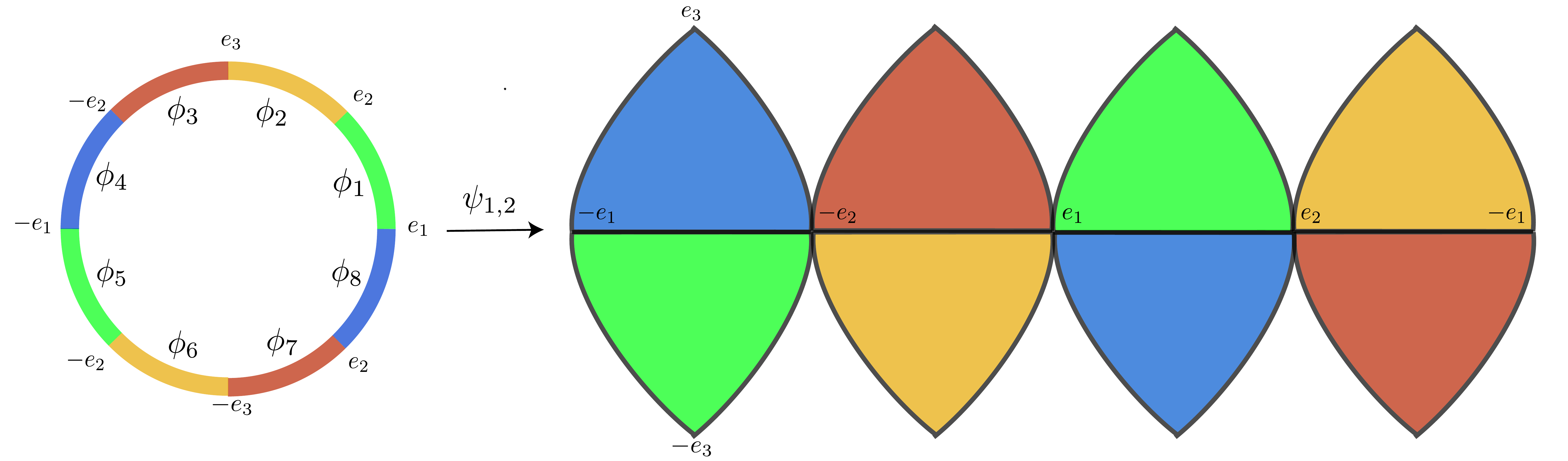}
    \caption{Structure of the map $\psi_{1,2}$ constructed in Proposition \ref{prop:constructionofpsi12}. Inside each arc, the map is defined via a space filling curve. For simplicity,  $\Sp^2$ is  ``cartographically" depicted.}
    \label{fig:psi-12}
\end{figure}

\subsection*{Spherical suspensions.} Suppose $m,n\in\N$ and a map $f:\Sp^m\longrightarrow\Sp^n$ are given. Then, one can lift this map $f$ to a map from $\Sp^{m+1}$ to $\Sp^{n+1}$ in the following way: Observe that an arbitrary point in $\Sp^{m+1}$ can be expressed as $(p\sin\theta,\cos\theta)$ for some $p\in\Sp^m$ and $\theta\in[0,\pi]$. Then,  the \emph{spherical suspension of} $f$ is the map
\begin{align*}
    Sf:\Sp^{m+1}&\longrightarrow\Sp^{n+1}\\
    (p\sin\theta,\cos\theta)&\longmapsto (f(p)\sin\theta,\cos\theta).
\end{align*}

\begin{lemma}\label{lemma:propsuspension}
If the map $f:\Sp^m\longtwoheadrightarrow\Sp^n$ is continuous, surjective, and antipode preserving, then $Sf:\Sp^{m+1}\longtwoheadrightarrow\Sp^{n+1}$ is also continuous, surjective, and antipode preserving.
\end{lemma}
\begin{proof}
Continuity and surjectivity are clear from the construction. Since $f$ is antipode preserving, we know that $f(-p)=-f(p)$ for every $p\in \Sp^m$. Hence,
\begin{align*}
    Sf(-p\sin\theta,-\cos\theta)&=Sf(-p\sin (\pi-\theta),\cos (\pi-\theta))\\
    &=(f(-p)\sin (\pi-\theta),\cos (\pi-\theta))\\
    &=(-f(p)\sin \theta,-\cos \theta)\\
    &=-(f(p)\sin \theta,\cos \theta)\\
    &=-Sf(p\sin\theta,\cos\theta)
\end{align*}
for any $p\in\Sp^m$ and $\theta\in [0,\pi]$. Thus, $Sf$ is also antipode preserving.
\end{proof}

We now use induction to obtain:
\begin{corollary}\label{cor:psim,m+1}
For any integer $m>0$, there exists a continuous, surjective, and antipode preserving map
$$\psi_{m,(m+1)}:\Sp^m\longtwoheadrightarrow\Sp^{m+1}.$$
\end{corollary}
\begin{proof}
Proposition \ref{prop:constructionofpsi12} guarantees the  existence of such $\psi_{1,2}$. For general $m$, it suffices to apply Lemma \ref{lemma:propsuspension} inductively.
\end{proof}

The following lemma is obvious:
\begin{lemma}\label{lemma:compprop}
Suppose that numbers $l,m,n\in\N$, $f:\Sp^l\longtwoheadrightarrow\Sp^m$, and maps $g:\Sp^m\longtwoheadrightarrow\Sp^n$ are given such that both $f,g$ are continuous, surjective, and antipode preserving. Then, their composition $g\circ f:\Sp^l\longtwoheadrightarrow\Sp^n$ is also continuous, surjective, and antipode preserving.
\end{lemma}

\subsection*{The proof of Theorem \ref{thm:surj'}.} We are now ready to prove Theorem \ref{thm:surj'} which states that there exists an antipode preserving continuous surjection $\psi_{m,n}:\Sp^m\longtwoheadrightarrow \Sp^n$ for any $0<m<n<\infty$:
\begin{proof}[Proof of Theorem \ref{thm:surj'}]
By Corollary \ref{cor:psim,m+1}, there are continuous, surjective, and antipode preserving maps

$\psi_{m,(m+1)},\psi_{(m+1),(m+2)},\dots$, and $\psi_{(n-1),n}$. Then, by Lemma \ref{lemma:compprop}, the map
$$\psi_{m,n}:=\psi_{(n-1),n}\circ\cdots\circ\psi_{(m+1),(m+2)}\circ\psi_{m,(m+1)}$$
is also continuous, surjective, and antipode preserving. This concludes the proof.
\end{proof}

\section{A Borsuk-Ulam theorem for discontinuous functions and the proof of Theorem \ref{thm:sn-sm-lb-DS}} \label{sec:lb}

\begin{defn}[Modulus of discontinuity]
Let $X$ be a topological space, $Y$ be a metric space, and $f:X\rightarrow Y$ be any function. Then, we define $\delta(f)$, the \emph{modulus of discontinuity of $f$} in the following way:
$$\delta(f):=\inf\{\delta\geq 0:\forall\,x\in X,\exists\,\text{an open neighborhood }U_x\text{ of }x\text{ s.t. }\diam(f(U_x))\leq\delta\}.$$
\end{defn}

\begin{remark}
Of course, $\delta(f)=0$ if and only if $f$ is continuous.
\end{remark}

It turns out that the modulus of discontinuity is a lower bound for  distortion:
\begin{proposition}\label{prop:moddis}
Let $\phi:(X,d_X)\longrightarrow (Y,d_Y)$ be a map between two metric spaces. Then, we have
$$\delta(\phi)\leq\mathrm{dis}(\phi).$$
\end{proposition}
\begin{proof}
If $\mathrm{dis}(\phi)=\infty$, then the proof is trivial. So, suppose $\mathrm{dis}(\phi)<\infty$. Now, fix arbitrary $x\in X$ and $\varepsilon>0$. Consider the open ball $U_x:=B(x,\frac{\varepsilon}{2})$. Then, for any $x',x''\in U_x$, we have
\begin{align*}
    d_Y(\phi(x'),\phi(x''))&\leq d_X(x',x'')+\vert d_X(x',x')-d_Y(\phi(x'),\phi(x'')) \vert\\
    &<\mathrm{dis}(\phi)+\varepsilon.
\end{align*}
This implies $\diam(\phi(U_x))\leq \mathrm{dis}(\phi)+\varepsilon$. Since $x$ is arbitrary, it means $\delta(\phi)\leq\mathrm{dis}(\phi)+\varepsilon$. Since $\varepsilon$ is arbitrary, we have the required inequality.
\end{proof}

The following variant of the Borsuk-Ulam theorem due to Dubins and Schwarz is the main tool for the proof of Theorem \ref{thm:sn-sm-lb-DS}. 

\begin{theorem}[{\cite[Theorem 1]{dubins1981equidiscontinuity}}]\label{thm:strBU}
For each integer $n>0$, the modulus of discontinuity of any function $f:\mathbb{B}^n\rightarrow \mathbb{S}^{n-1}$ that maps every pair of antipodal points on the boundary of $\mathbb{B}^n$ onto antipodal points on $\mathbb{S}^{n-1}$ is not less than $\zeta_{n-1}$.
\end{theorem}

In  Appendix \ref{app:proof-BU} we provide a concise self contained proof of this result based on ideas by Arnold Wa{\ss}mer; see Matou\v{s}ek \cite[page 41]{matousek2003using}.

We immediately have:

\begin{corollary}[{\label{cor:strBU}\cite[Corollary 3]{dubins1981equidiscontinuity}}]
For each integer $n>0$, the modulus of discontinuity of any function $g:\mathbb{S}^n\rightarrow \mathbb{S}^{n-1}$ which maps every pair of antipodal points on $\mathbb{S}^n$ onto antipodal points on $\mathbb{S}^{n-1}$ is not less than $\zeta_{n-1}$.
\end{corollary}

We provide a detailed proof of this result for completeness.

\begin{proof}
Consider the following map
\begin{align*}
    \Phi:\mathbb{B}^n&\longrightarrow\Sp^n\\
    (x_1,\dots,x_n)&\longmapsto\left(x_1,\dots,x_n,\sqrt{1-(x_1^2\cdots +x_n^2)}\right).
\end{align*}
Obviously, $\Phi$ is continuous and its image is $\mathbf{H}_{\geq 0}(\Sp^n)$. Now, fix an arbitrary $\delta\geq 0$ such that: $$ \mbox{($*$) for every $x\in \Sp^n$ there exists an open neighborhood $U_x$ of $x$ with $\diam(g(U_x))\leq\delta$.}$$
Now, fix arbitrary $x'\in\mathbb{B}^n$. Then, $\Phi^{-1}(U_{\Phi(x')})$ is an open neighborhood of $x'$, and $$\diam\big(g\circ\Phi(\Phi^{-1}(U_{\Phi(x')}))\big)\leq\diam(g(U_{\Phi(x')}))\leq\delta.$$
Since $x'$ is arbitrary, this means that  $\delta\geq\delta(g\circ\Phi)$. Moreover, since $g\circ\Phi$ is antipode preserving,  $\delta(g\circ\Phi)\geq\zeta_{n-1}$ by Theorem \ref{thm:strBU}. Hence, we  conclude that  $\delta\geq\zeta_{n-1}$. Finally, since $\delta$ satisfying condition ($*$) above was arbitrary, by taking the infimum we conclude that 
$$\delta(g)\geq\zeta_{n-1}$$
as we wanted.
\end{proof}

\begin{corollary}\label{cor:strBU2}
For each integer $n>0$, any function $g:\mathbb{S}^n\rightarrow \mathbb{S}^{n-1}$ which maps every pair of antipodal points on $\mathbb{S}^n$ onto antipodal points on $\mathbb{S}^{n-1}$ satisfies $\mathrm{dis}(g)\geq \zeta_{n-1}$.
\end{corollary}
\begin{proof}
Apply Corollary \ref{cor:strBU} and Proposition \ref{prop:moddis}.
\end{proof}

\subsection{The proof of Theorem \ref{thm:sn-sm-lb-DS}} \label{sec:helmets}
We are almost ready to prove Theorem \ref{thm:sn-sm-lb-DS} that establishes $\dgh(\Sp^m,\Sp^n)\geq\frac{1}{2}\zeta_m$ for any $0<m<n<\infty$.

For each integer $n\geq 1$, recall the natural isometric embedding of $\Sp^{n-1}$ to the equator $\sete(\Sp^n)$ of $\Sp^n$:
\begin{align*}
    \iota_{n-1}:\mathbb{S}^{n-1}&\longhookrightarrow \mathbb{S}^n\\
    (x_1,\dots,x_n)&\longmapsto (x_1,\dots,x_n,0).
\end{align*}

Also, let us define the sets $\seta(\Sp^n)\subset\Sp^n$ (which we will sometimes refer to as ``helmets") for $n\in\N$:

\begin{defn}[Definition of $\seta(\Sp^n)$]\label{def:seta}
Let
\begin{align*}
    &\seta(\mathbb{S}^0):=\{1\}\textrm{ and,}\\
    &\seta(\mathbb{S}^1):=\{(\cos{\theta},\sin{\theta})\in \mathbb{S}^1:\theta\in[0,\pi)\}.
\end{align*}
Moreover, for general $n\geq 1$,   define, inductively,
$$\seta(\mathbb{S}^n):=\mathbf{H}_{>0}(\mathbb{S}^n)\cup \iota_{n-1}(\seta(\mathbb{S}^{n-1})).$$
\end{defn}

See Figure \ref{fig:disjoint} for an illustration. Observe that, for any $n\geq 0$, $$\mbox{$\seta(\mathbb{S}^n)\cap \big(-\seta(\mathbb{S}^n)\big)=\emptyset$ \,and\, $\seta(\mathbb{S}^n)\cup \big(-\seta(\mathbb{S}^n)\big)=\mathbb{S}^n.$}$$

\begin{figure}
    \centering
    \includegraphics[width=0.6\linewidth]{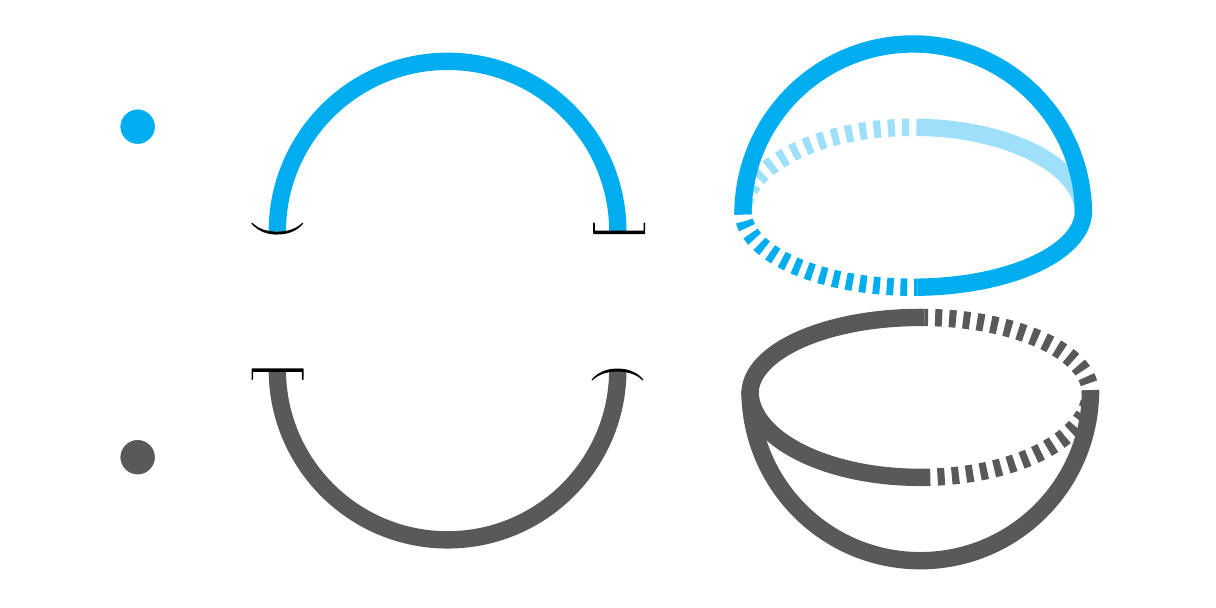}
    \caption{From left to right, the blue sets represent $\seta(\Sp^0)$, $\seta(\Sp^1)$, and $\seta(\Sp^2)$. The figure also shows their antipodes in dark grey, respectively. See Definition \ref{def:seta} for the precise definition.}
    \label{fig:disjoint}
\end{figure}

The following lemma is simple but critical. Given any map $\phi:\Sp^n\rightarrow \Sp^{n-1}$ it will permit constructing an antipode preserving map $\phi^*$ with at most the same distortion.

\begin{lemma}\label{lemma:distortion}
For any $m,n\geq 0$, let $\emptyset\neq C\subseteq \Sp^n$ satisfy $C\cap (-C)=\emptyset$ and let $\phi:C\rightarrow \Sp^m$ be any map. Then, the extension $\phi^*$ of $\phi$ to the set $C\cup(-C)$ defined  by
\begin{align*}
    {\phi^*}: C\cup(-C)&\longrightarrow \Sp^m\\
    C \ni x&\longmapsto \phi(x)\\
    -x&\longmapsto-\phi(x)
\end{align*}
is antipode preserving and satisfies $\dis({\phi^*})=\dis(\phi)$.
\end{lemma}
\begin{proof}
$\phi^*$ is obviously antipode preserving by the definition. Now, fix arbitrary $x,x'\in C$. Then, 

\begin{align*}
    \big\vert d_{\Sp^n}(x,-x')-d_{\Sp^m}({\phi^*}(x),{\phi^*}(-x')) \big\vert&=\big\vert (\pi-d_{\Sp^n}(x,x'))-(\pi-d_{\Sp^m}(\phi(x),\phi(x'))) \big\vert\\
    &=\big\vert d_{\Sp^n}(x,x')-d_{\Sp^m}(\phi(x),\phi(x'))\vert
    \\&\leq\dis(\phi)
\end{align*}
and,
$$\vert d_{\Sp^n}(-x,-x')-d_{\Sp^m}({\phi^*}(-x),{\phi^*}(-x')) \vert=\vert d_{\Sp^n}(x,x')-d_{\Sp^m}(\phi(x),\phi(x')) \vert\leq\dis(\phi).$$

This implies $\dis({\phi^*})=\dis(\phi)$ as we wanted to prove.
\end{proof}

\begin{corollary}\label{coro:preservedistortion}
For each $n\in\Z_{>0}$ and any map $\phi:\Sp^n\rightarrow \Sp^{n-1}$ there exists an antipode preserving map  $\phi^*:\Sp^n\rightarrow \Sp^{n-1}$ such that  $\dis(\phi^*)\leq\dis(\phi)$.
\end{corollary}
\begin{proof}
Consider the restriction of $\phi$ to $\seta(\Sp^n)$ and apply Lemma \ref{lemma:distortion}.
\end{proof}

Finally, we are ready to prove Theorems \ref{thm:sn-sm-lb-DS}.

\begin{proof}[Proof of Theorem \ref{thm:sn-sm-lb-DS}]
Let $0<m<n<\infty$. We first prove the second claim of Theorem \ref{thm:sn-sm-lb-DS} that $\dis(\phi)\geq\zeta_m$ for any map $\phi:\Sp^n\rightarrow\Sp^m$.  Suppose to the contrary so that there is a map $\tilde{g}:\Sp^n\rightarrow\Sp^m$ with $\dis(\tilde{g})<\zeta_m$. By restriction, this map induces a map $g:\Sp^{m+1}\rightarrow\Sp^m$ such that $\dis(g)<\zeta_m$. By applying Corollary \ref{coro:preservedistortion}, one can modify  $g$ into an antipode preserving map $\widehat{g}:\Sp^{m+1}\rightarrow \Sp^m$ with $\dis(\widehat{g})<\zeta_m$, which  contradicts Corollary \ref{cor:strBU2}. This yields the proof of the second claim of Theorem \ref{thm:sn-sm-lb-DS}.

Now, in order to prove the first claim of Theorem \ref{thm:sn-sm-lb-DS} that $\dgh(\Sp^m,\Sp^n)\geq\frac{1}{2}\zeta_m$,  suppose that $\Gamma$ is a correspondence between $\Sp^m$ and $\Sp^n$ with $\dis(\Gamma)<\zeta_m$. Pick any function $g:\Sp^n\rightarrow\Sp^m$ such that $(g(x),x)\in\Gamma$ for every $x\in\Sp^n$. This implies that $\dis(g)\leq\dis(\Gamma)<\zeta_m$, which contradicts the second claim. This proves the first claim.
\end{proof}
 
\subsection{The proof of Theorem \ref{thm:stronggeneralization}}
By carefully inspecting the proof of Theorem \ref{thm:sn-sm-lb-DS}, one can extract the much stronger Theorem \ref{thm:stronggeneralization}. 

\begin{proof}[Proof of Theorem \ref{thm:stronggeneralization}]
We will actually prove slightly stronger result. Suppose (i)  $X$ can be isometrically embedded into $\Sp^m$ and (ii) $\seta(\Sp^{m+1})$ (note that $\seta(\Sp^{m+1})\subset\seth_{\geq 0}(\Sp^{m+1})$) can be isometrically embedded into $Y$. Now, we prove that $\dgh(X,Y)\geq\frac{1}{2}\zeta_m$. Moreover, $\dis(\phi)\geq\zeta_m$ for any map $\phi:Y\rightarrow X$.

We first prove the second claim. Suppose to the contrary so that there is a map $\tilde{g}:Y\rightarrow X$ with $\dis(\tilde{g})<\zeta_m$. Then, since $\seta(\Sp^{m+1})$ is isometrically embedded in $Y$ and $X$ is isometrically embedded in $\Sp^m$ by the assumption, one can construct a map $g:\seta(\Sp^{m+1})\longrightarrow\Sp^m$ with $\dis(g)<\zeta_m$. Hence, with the aid of Lemma \ref{lemma:distortion}, one can modify this $g$ into an antipode preserving map $\widehat{g}:\Sp^{m+1}\rightarrow \Sp^m$ with $\dis(\widehat{g})<\zeta_m$, which contradicts Corollary \ref{cor:strBU2}. This yields the proof of the second claim.

Now, in order to prove the first claim, use the same argument used in the proof of Theorem \ref{thm:sn-sm-lb-DS}.
\end{proof}

\section{The proofs of Proposition \ref{prop:ub} and Proposition \ref{prop:ub-n-m}} \label{sec:proofs-ub}

To prove Proposition \ref{prop:ub} and Proposition \ref{prop:ub-n-m}, we need to define a few notions.

\begin{defn}
For any nonempty $U\subseteq\Sp^{n-1}$, we define \emph{the cone of }$U$, as the following subset  of $\Sp^n\subset \R^{n+1}$:
$$\mathcal{C}(U):=\left\{\cos\theta\cdot e_{n+1}+\sin\theta\cdot \iota_{n-1}(u)\in \mathbf{H}_{\geq 0}(\Sp^n):u\in U\text{ and }\theta\in\left[0,\frac{\pi}{2}\right]\right\}$$
where $e_{n+1}=(0,0,\cdots,0,1)\in\R^{n+1}$ is the north pole of $\Sp^n$. See Figure \ref{fig:cone}.
\end{defn}

\begin{figure}
    \centering
    \includegraphics[width=0.4\linewidth]{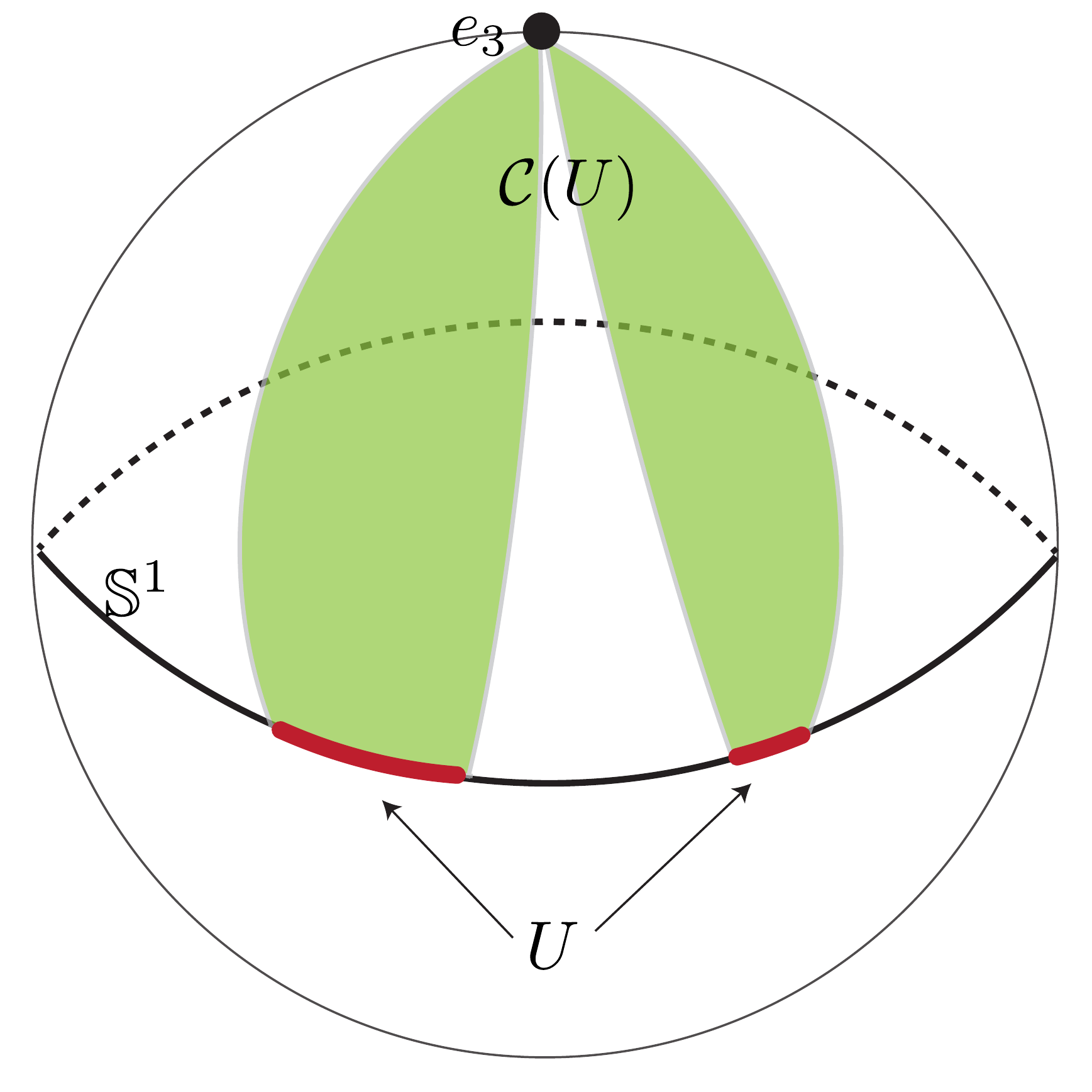}
    \caption{The cone $\mathcal{C}(U)$ for a subset $U$ of $\Sp^1$.}
    \label{fig:cone}
\end{figure}

\begin{lemma}\label{lemma:diamofcone}
For any nonempty $U\subseteq\Sp^{n-1}$,
$$\diam(\mathcal{C}(U))=\begin{cases}\frac{\pi}{2}&\text{if }\diam(U)\leq\frac{\pi}{2},\\ \diam(U)&\text{if }\diam(U)\geq\frac{\pi}{2}.\end{cases}$$
\end{lemma}
\begin{proof}
Recall that
$$\mathcal{C}(U):=\left\{\cos\theta\cdot e_{n+1}+\sin\theta\cdot \iota_{n-1}(u)\in \mathbf{H}_{\geq 0}(\Sp^n):u\in U\text{ and }\theta\in\left[0,\frac{\pi}{2}\right]\right\}.$$
Now, for $u,v\in U$ and $\theta,\theta'\in[0,\frac{\pi}{2}]$, consider the following inner product:
$$\langle \cos\theta\cdot e_{n+1}+\sin\theta\cdot\iota_{n-1}(u),\cos\theta'\cdot e_{n+1}+\sin\theta'\cdot\iota_{n-1}(v)\rangle=\cos\theta\cos\theta'+\langle u,v\rangle\cdot\sin\theta\sin\theta'.$$

Hence, if $\langle u,v\rangle\geq 0$,
$$\langle \cos\theta\cdot e_{n+1}+\sin\theta\cdot\iota_{n-1}(u),\cos\theta'\cdot e_{n+1}+\sin\theta'\cdot\iota_{n-1}(v)\rangle\geq 0$$
so that $d_{\Sp^n}(\cos\theta\cdot e_{n+1}+\sin\theta\cdot u,\cos\theta'\cdot e_{n+1}+\sin\theta'\cdot v)\leq\frac{\pi}{2}$.

If $\langle u,v\rangle\leq 0$, $\cos\theta\cos\theta'+\langle u,v\rangle\cdot\sin\theta\sin\theta$ becomes decreasing in $\theta,\theta'$. Hence, it is minimized for $\theta=\theta'=\frac{\pi}{2}$. Therefore,
$$\langle \cos\theta\cdot e_{n+1}+\sin\theta\cdot\iota_{n-1}(u),\cos\theta'\cdot e_{n+1}+\sin\theta'\cdot\iota_{n-1}(v)\rangle\geq\langle u,v \rangle$$
so that $d_{\Sp^n}(\cos\theta\cdot e_{n+1}+\sin\theta\cdot\iota_{n-1}(u),\cos\theta'\cdot e_{n+1}+\sin\theta'\cdot\iota_{n-1}(v))\leq d_{\Sp^{n-1}}(u,v)$ which completes the proof. 
\end{proof}

\begin{defn}[Geodesic convex hull]\label{def:gch}
Given a nonempty subset $A\subset \Sp^n$, its \emph{geodesic convex hull $\mathrm{conv}_{\Sp^n}(A)$} is defined to be the smallest subset of $\Sp^n$ containing $A$ such that for any two points in the set, all minimizing geodesics between them are also contained in the set. It is clear that when $A$ is contained in an open hemisphere,  $$\mathrm{conv}_{\Sp^n}(A) = \{\Pi_{\Sp^n}(c)|\,c\in \mathrm{conv}(A)\}$$ where $\Pi_{\Sp^n}(p) := \frac{p}{\|p\|}$ for $p\neq 0$ and $\Pi_{\Sp^n}(p):=0$ otherwise.
\end{defn}


In what follows we will prove Proposition \ref{prop:ub-n-m} after proving Proposition \ref{prop:ub}. The proof of the former proposition generalizes the construction used in the proof of the latter one, and as a consequence Proposition \ref{prop:ub} (which exhibits a correspondence between $\Sp^2$ and $\Sp^1$) is a special case of Proposition \ref{prop:ub-n-m} (which constructs a correspondence between $\Sp^{m+1}$ and $\Sp^m$). 

With the goal of making the construction more understandable, we have however decided to first present a detailed proof of Proposition \ref{prop:ub} since the optimal $R_{2,1}$ correspondence constructed therein is used in the proof of Proposition \ref{prop:ub13}  in order to construct an optimal correspondence $R_{3,1}$.  After this we provide a streamlined proof of Proposition \ref{prop:ub-n-m}.

\subsection{The proof of Proposition \ref{prop:ub}}\label{sec:proof-ub}

We will find an upper bound for $\dgh(\Sp^1,\mathbf{H}_{\geq 0}(\Sp^2))$ (resp. $\dgh(\Sp^1,\Sp^2)$) by constructing a specific correspondence between $\Sp^1$ and $\mathbf{H}_{\geq 0}(\Sp^2)$ (resp. $\Sp^1$ and $\Sp^2$). This correspondence is inspired by the case $m=1$ of certain surjective maps from $\Sp^{m+1}$ to $\Sp^m$  \cite[Scholium 1]{dubins1981equidiscontinuity} developed in the course of the authors' study of the modulus of discontinuity of antipode preserving maps between spheres. In spite of the fact that these maps will in general fail to yield tight upper bounds for distortion, they still permit giving non-trivial upper bounds for $\mathfrak{g}_{m,m+1}$. This  will be explained in \S\ref{sec:proof-gen-ub}.

\begin{remark}
We provide an alternative construction of a correspondence $R$ between $\Sp^1$ and $\Sp^2$ with $\dis(R)\leq\frac{2\pi}{3}$ in Appendix \ref{sec:alternative-method}.
\end{remark}

\begin{figure}
    \centering
    \includegraphics[width=0.7\linewidth]{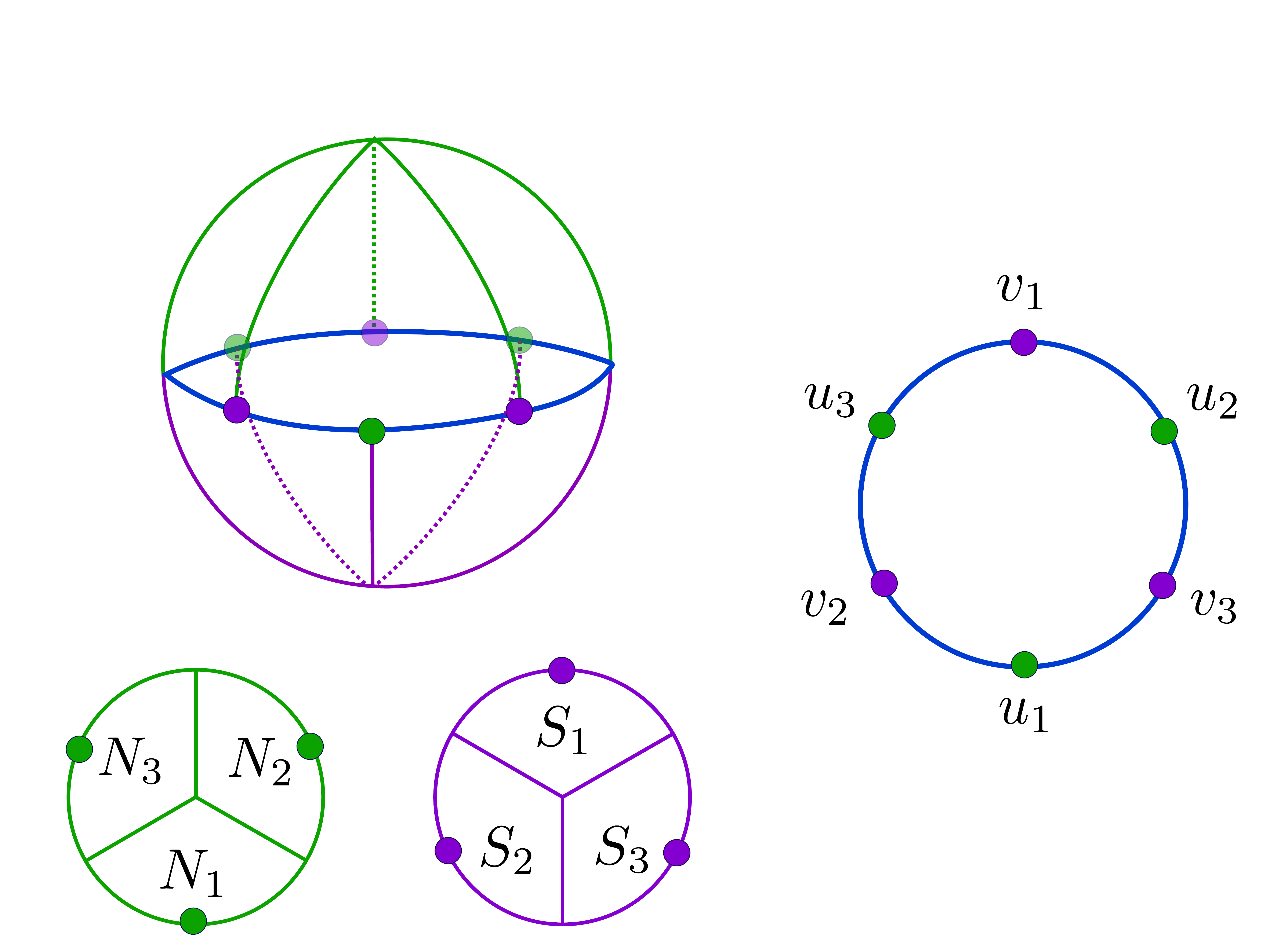}
    \caption{The surjection $\phi_{2,1}:\Sp^2\twoheadrightarrow \Sp^1$ constructed in Proposition \ref{prop:ub}. In the figure, $S_i := -N_i$ and $v_i:=-u_1$ for $i=1,2,3$. The equator of $\Sp^2$ is mapped to itself under the map (via the identity map).}
    \label{fig:phi21}
\end{figure}

\begin{proof}[Proof of Proposition \ref{prop:ub}]
We will prove both claims that there exists (1) a correspondence between $\Sp^1$ and $\mathbf{H}_{\geq0}(\Sp^2)$, and (2) a correspondence between $\Sp^1$ and $\Sp^2$, both of which have distortion at most $\frac{2\pi}{3}$ in an intertwined way.

In order to prove the first claim, note that it is enough to find a surjective map $\widetilde{\phi}_{2,1}:\mathbf{H}_{\geq 0}(\Sp^2)\twoheadrightarrow\Sp^1$ (resp. $\phi_{2,1}:\Sp^2\twoheadrightarrow\Sp^1$) such that $\dis(\widetilde{\phi}_{2,1})\leq\zeta_1=\frac{2\pi}{3}$ (resp. $\dis(\phi_{2,1})\leq\zeta_1=\frac{2\pi}{3}$) since this map gives rise to a correspondence $\widetilde{R}_{2,1}:=\mathrm{graph}(\widetilde{\phi}_{2,1})$ (resp. $R_{2,1}:=\mathrm{graph}(\phi_{2,1})$) with $\dis(\widetilde{R}_{2,1}) = \dis(\widetilde{\phi}_{2,1})\leq \zeta_1$ (resp. $\dis(R_{2,1}) = \dis(\phi_{2,1})\leq \zeta_1$).

Let
$$u_1:=(1,0,0),\,u_2:=\left(-\frac{1}{2},\frac{\sqrt{3}}{2},0\right),\text{ and }u_3:=\left(-\frac{1}{2},-\frac{\sqrt{3}}{2},0\right).$$

Note that $\{u_1,u_2,u_3\}$ are the vertices of a regular triangle inscribed in $\sete(\Sp^2)$. We divide the open upper hemisphere $\mathbf{H}_{>0}(\mathbb{S}^2)$ into three regions by using the Voronoi partitions induced by these three points. Precisely, for each $i=1,2,3$ we define the following set:
$$N_i:=\{x\in \mathbf{H}_{>0}(\mathbb{S}^2):d_{\mathbb{S}^2}(x,u_i)\leq d_{\mathbb{S}^2}(x,u_j)\, \forall j\neq i\textit{ and }d_{\mathbb{S}^2}(x,u_i)< d_{\mathbb{S}^2}(x,u_j)\, \forall j < i \}.$$

See Figure \ref{fig:phi21} for an illustration of the construction.

Observe that $\overline{N_i}=\mathcal{C}(\mathrm{Conv}_{\Sp^1}(\{\iota_1^{-1}(-u_j)\in\Sp^1:j\neq i\}))$ for each $i=1,2,3$. Since $\mathrm{Conv}_{\Sp^1}(\{\iota_1^{-1}(-u_j)\in\Sp^1:j\neq i\})$ is just the shortest geodesic between the two points $\{\iota_1(-u_j)\in\Sp^1:j\neq i\}$ with length $\zeta_1=\frac{2\pi}{3}$, $\diam(\overline{N_i})\leq\zeta_1$ by Lemma \ref{lemma:diamofcone} for any $i=1,2,3$.

We now construct a map $\widetilde{\phi}_{2,1}:\mathbf{H}_{\geq 0}(\Sp^2)\rightarrow \Sp^1$ in the following way:
\begin{align*}
    \widetilde{\phi}_{2,1}(p):=\begin{cases} \iota_1^{-1}(u_i)\textit{ if }p\in N_i,\\ \iota_1^{-1}(p)\textit{ if }p\in \sete(\Sp^2).\end{cases}
\end{align*}

Let us prove that the distortion of $\widetilde{\phi}_{2,1}$ is less than or equal to $\zeta_1$. We break the study of the value of
$$\big\vert d_{\Sp^2}(p,q)-d_{\Sp^1}(\widetilde{\phi}_{2,1}(p),\widetilde{\phi}_{2,1}(q))\big\vert$$ for $p,q\in \mathbf{H}_{\geq 0}(\Sp^2)$ into several cases:
\begin{enumerate}
    \item \textbf{Case $p\in N_i$ and $q\in N_j$}: If $i=j$, then $0\leq d_{\Sp^2}(p,q)\leq\zeta_1$ and $\widetilde{\phi}_{2,1}(p)=\widetilde{\phi}_{2,1}(q)=\iota_m^{-1}(u_i)$ so that $d_{\Sp^1}(\widetilde{\phi}_{2,1}(p),\widetilde{\phi}_{2,1}(q))=0$. Hence,
    $$\vert d_{\Sp^2}(p,q)-d_{\Sp^1}(\widetilde{\phi}_{2,1}(p),\widetilde{\phi}_{2,1}(q))\vert\leq\zeta_1.$$ If $i\neq j$, then $0\leq d_{\Sp^2}(p,q)\leq\pi$ and $d_{\Sp^1}(\widetilde{\phi}_{2,1}(p),\widetilde{\phi}_{2,1}(q))=\zeta_1$ so that
    $$\vert d_{\Sp^2}(p,q)-d_{\Sp^1}(\widetilde{\phi}_{2,1}(p),\widetilde{\phi}_{2,1}(q))\vert\leq \zeta_1.$$
    
    \item \textbf{Case $p\in N_i$ and $q\in \sete(\Sp^2)$}: Then,
    \begin{align*}
        \vert d_{\Sp^2}(p,q)-d_{\Sp^1}(\widetilde{\phi}_{2,1}(p),\widetilde{\phi}_{2,1}(q))\vert&=\vert d_{\Sp^2}(p,q)-d_{\Sp^1}(\iota_1^{-1}(u_i),\iota_1^{-1}(q))\vert\\
        &=\vert d_{\Sp^2}(p,q)-d_{\Sp^2}(u_i,q)\vert\\
        &\leq d_{\Sp^2}(p,u_i)\leq\zeta_1.
    \end{align*}
    
    \item \textbf{Case $p,q\in \sete(\Sp^2)$}: Then, $\widetilde{\phi}_{2,1}(p)=\iota_1^{-1}(p)$ and $\widetilde{\phi}_{2,1}(q)=\iota_1^{-1}(q)$. Hence,
    $$\vert d_{\Sp^2}(p,q)-d_{\Sp^1}(\widetilde{\phi}_{2,1}(p),\widetilde{\phi}_{2,1}(q))\vert=0\leq\zeta_1.$$
\end{enumerate}
This implies that  $\dis(\widetilde{\phi}_{2,1})\leq\zeta_1$. Observe that $\widetilde{\phi}_{2,1}$ is the identity on $\sete(\Sp^2)$, so $\widetilde{\phi}_{2,1}$ is surjective.

For the second claim, by applying Lemma \ref{lemma:distortion} to $\widetilde{\phi}_{2,1}\vert_{\seta(\Sp^2)}$, we construct a map $\phi_{2,1}:\Sp^2\longtwoheadrightarrow\Sp^1$ such that $\dis(\phi_{2,1})=\dis(\widetilde{\phi}_{2,1})\leq\zeta_1$. Moreover, by construction, $\phi_{2,1}$ is obviously surjective and antipode preserving.
\end{proof}

\begin{remark}
The antipode preserving property of $\phi_{2,1}$ will be useful for the proof of Proposition \ref{prop:ub13}.
\end{remark}

\subsection{The proof of Proposition \ref{prop:ub-n-m}}\label{sec:proof-gen-ub}

One can prove Proposition \ref{prop:ub-n-m} using a generalization of the approach used in the proof of Proposition \ref{prop:ub}.

\begin{remark}[Diameter of faces of geodesic simplices]\label{rmk:diamoneface}
Let $\{u_1,\dots,u_{m+2}\}$ be the vertices of a regular $(m+1)$-simplex inscribed in $\Sp^m$. Let
$$F_m:=\mathrm{Conv}_{\Sp^m}(\{u_1,\dots,u_{m+1}\}).$$
In other words, $F_m$ is just a \emph{face} of the geodesic regular  simplex inscribed in $\Sp^m$ where the length of each edge is $\zeta_m=\arccos\left(-\frac{1}{m+1}\right)$.

The diameter of $F_m$ can be determined by applying a result by Santal\'o \cite[Lemma 1]{santalo1946convex}:
\begin{align*}
    \diam(F_m)=\eta_m:=\begin{cases}\arccos\left(-\frac{m+1}{m+3}\right)&\text{for }m\text{ odd},\\\arccos\left(-\sqrt{\frac{m}{m+4}}\right)&\text{for }m\text{ even}.\end{cases}
\end{align*}
As proved by Santal\'o, this diameter is realized either by the distance between the circumcenter of the geodesic convex hull of $A_m^\mathrm{odd}:=\{u_1,\dots,u_{\frac{m+1}{2}}\}$ and the circumcenter of the geodesic convex hull of $B^\mathrm{odd}_m:=\{u_{\frac{m+3}{2}},\dots,u_{m+1}\}$ if $m$ is odd, or by the distance between the circumcenter of the geodesic convex hull of $A_m^\mathrm{even}:=\{u_1,\dots,u_{\frac{m}{2}}\}$ and the circumcenter of the geodesic convex hull of $B_m^\mathrm{even}:=\{u_{\frac{m+2}{2}},\dots,u_{m+1}\}$ if $m$ is even. See Figure \ref{fig:eta-m}.

\begin{figure}
    \centering
    \includegraphics[width=0.7\linewidth]{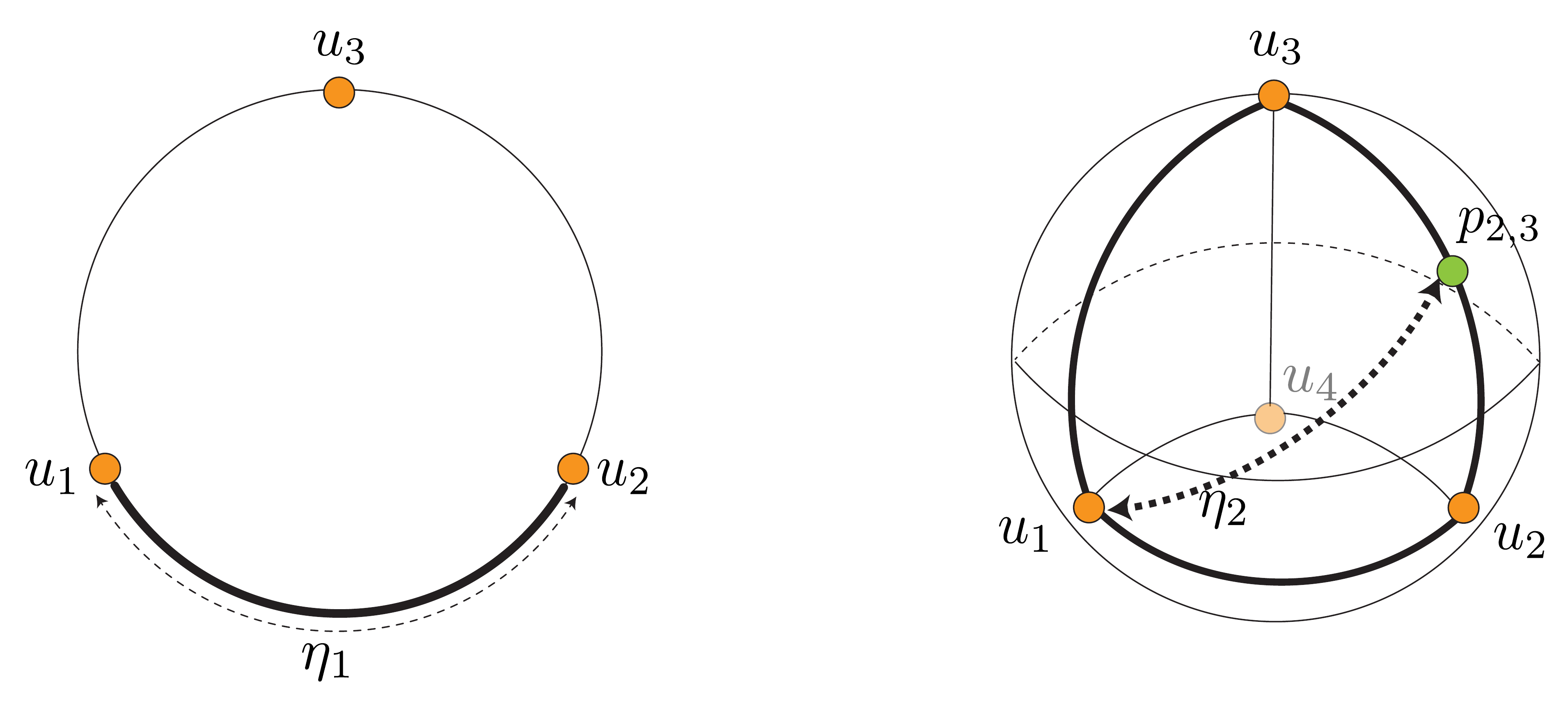}
    \caption{The diameter of a face $F_m$ of a geodesic simplex: the cases $m=1$ and $m=2$. When $m=1$, $A_1^\mathrm{odd} = \{u_1\}$ and $B_1^\mathrm{odd}=\{u_2\}$. When $m=2$ (on the right), $A_2^\mathrm{even} = \{u_1\}$,  $B_2^\mathrm{even} = \{u_2,u_3\}$ and the circumcenter of the geodesic convex hull of $B_2^\mathrm{even} $ is the point $p_{2,3}$, i.e. $\diam(F_2) = \eta_2 =  d_{\Sp^2}(u_1,p_{2,3})$.}
    \label{fig:eta-m}
\end{figure}

Observe that, in general,
$$\zeta_m\leq\eta_m\leq 2\,(\pi-\zeta_m).$$
Note that as $m$ goes to infinity, $\zeta_m$ goes to $\frac{\pi}{2}$, $\eta_m$ goes to $\pi$, and $2\,(\pi-\zeta_m)$ also goes to $\pi$.
\end{remark}

\begin{remark}\label{rem:Voronoi-cvx}
Let $\{u_1,\dots,u_{m+2}\}\subset\Sp^m$ be the vertices of a regular $(m+1)$-simplex inscribed in $\Sp^m$. Let $V_1,\dots,V_{m+2}$ be the Voronoi partition of $\Sp^m$ induced by $\{u_1,\dots,u_{m+2}\}$. Then, $\overline{V_i}=\mathrm{Conv}_{\Sp^m}(\{-u_j:j\neq i\})$ (so, $\overline{V_i}$ is congruent to $F_m$ in Remark \ref{rmk:diamoneface}) for each $i=1,\dots,m+2$. Here is a proof:

Without loss of generality, one can assume $i=1$. Observe that $$\overline{V_1}=\{x\in\Sp^m:d_{\Sp^m}(x,u_1)\leq d_{\Sp^m}(x,u_j)\,\forall j\neq 1\}.$$

Now fix arbitrary $x\in \mathrm{Conv}_{\Sp^m}(\{-u_j:j\neq 1\})$. Then, $x=\frac{v}{\Vert v\Vert}$ where $v=\sum_{j=2}^{m+2}\lambda_j (-u_j)$ and $\lambda_j$'s are non-negative coefficients such that $\sum_{j=2}^{m+2}\lambda_j=1$. Then,
$$\langle x,u_1 \rangle=\frac{1}{\Vert v \Vert}\cdot\frac{1}{m+1}\cdot\sum_{j=2}^{m+2}\lambda_j=\frac{1}{\Vert v \Vert}\cdot\frac{1}{m+1}$$
and for any $k\neq 1$,
$$\,\,\,\,\,\,\,\,\langle x,u_k \rangle=\frac{1}{\Vert v \Vert}\cdot\left(-1+\frac{1}{m+1}\cdot\sum\limits_{2\leq j\leq m+2, j\neq k}\lambda_j\right).$$ Hence, this implies $\langle x,u_1 \rangle\geq\langle x,u_k \rangle$ so that $d_{\Sp^m}(x,u_1)\leq d_{\Sp^m}(x,u_k)$ for any $k\neq 1$. Therefore, $x\in\overline{V_1}$ and $\mathrm{Conv}_{\Sp^m}(\{-u_j:j\neq 1\})\subseteq\overline{V_1}$.

For the other direction, fix arbitrary $x\in\overline{V_1}$. Since $\{-u_2,\dots,-u_{m+2}\}$ is a basis of $\R^{m+1}$, there are a unique set of coefficients $\{c_i\}_{i=2}^{m+2}$ such that $x=\sum_{i=2}^{m+2} c_i(-u_i)$. Then, one can check $c_i=\frac{m+1}{m+2}(\langle x,u_1 \rangle-\langle x,u_i \rangle)$ for $i=2\dots,m+2$ by using the fact $\sum_{i=1}^{m+2} \langle x,u_i \rangle=\langle x, \sum_{i=1}^{m+2} u_i \rangle=\langle x,0 \rangle=0$, and \cite[5.27 Theorem]{folland1999real} (the fact that $\sum_{i=1}^{m+2} u_i=0$ can be easily checked by the induction on $m$). Note that $c_i\geq 0$ since $\langle x,u_1 \rangle\geq\langle x,u_i \rangle$. Hence, if we define $$\lambda_i:=\frac{c_i}{\sum_{j=2}^{m+2}c_j}=\frac{1}{m+2}\left(1-\frac{\langle x,u_i \rangle}{\langle x,u_1 \rangle}\right)$$ for each $i=2\dots,m+2$ and $v:=\sum_{i=2}^{m+2}\lambda_i(-u_i)$, then $x=\frac{v}{\Vert v \Vert}$. Therefore, $x\in\mathrm{Conv}_{\Sp^m}(\{-u_j:j\neq 1\})$ and $\overline{V_1}\subseteq\mathrm{Conv}_{\Sp^m}(\{-u_j:j\neq 1\})$. Hence, $\overline{V_1}=\mathrm{Conv}_{\Sp^m}(\{-u_j:j\neq 1\})$ as we claimed.
\end{remark}

\begin{proof}[Proof of Proposition \ref{prop:ub-n-m}]
We construct a surjective and antipode preserving map $$\phi_{(m+1),m}:\Sp^{m+1}\longtwoheadrightarrow\Sp^m$$ with $$\mathrm{dis}(\phi_{(m+1),m})\leq\eta_m.$$

Let $\{u_1,\dots,u_{m+2}\}$ be the vertices of a regular $(m+1)$-simplex inscribed in $\sete(\Sp^{m+1})$. We divide open upper hemisphere $\mathbf{H}_{>0}(\Sp^{m+1})$ into $(m+2)$ regions by using the Voronoi partitions induced by these $(m+2)$ vertices. Precisely, for each $i=1,\dots,m+2$ we define the following set:

$$N_i:=\left\{p\in \mathbf{H}_{>0}(\Sp^{m+1})\left| \begin{array}{l}d_{\Sp^{m+1}}(p,u_i)\leq d_{\Sp^{m+1}}(p,u_j)\, \forall j\neq i,\\\hspace{.8in}\text{and}\\d_{\Sp^{m+1}}(p,u_i)< d_{\Sp^{m+1}}(p,u_j)\, \forall j < i \end{array}\right.\right\}.$$

Observe that $\overline{N_i}=\mathcal{C}(\overline{V_i})$ where $\{V_1,\dots,V_{m+2}\}$ is the Voronoi partition of $\Sp^m$ induced by $$\{\iota_m^{-1}(u_1),\dots,\iota_m^{-1}(u_{m+2})\}.$$ Hence, by Lemma \ref{lemma:diamofcone}, Remark \ref{rmk:diamoneface}, and Remark \ref{rem:Voronoi-cvx}, one  concludes that $\diam(\overline{N_i})\leq\eta_m$ for any $i=1,\dots,m+2$ .

We now construct a map $\widetilde{\phi}_{(m+1),m}:\seta(\Sp^{m+1})\rightarrow \Sp^m$ in the following way:
\begin{align*}
    \widetilde{\phi}_{(m+1),m}(p):=\begin{cases} \iota_m^{-1}(u_i)\textit{ if }p\in N_i\\ \iota_m^{-1}(p)\textit{ if }p\in \iota_m(\seta(\Sp^m)).\end{cases}
\end{align*}

In order to  prove that the distortion of $\widetilde{\phi}_{(m+1),m}$ is less than or equal to $\eta_m$ we break the study of the value of
$$\big\vert d_{\Sp^{m+1}}(p,q)-d_{\Sp^m}(\widetilde{\phi}_{(m+1),m}(p),\widetilde{\phi}_{(m+1),m}(q))\big\vert$$ for $p,q\in \seta(\Sp^{m+1})$ into several cases:
\begin{enumerate}
    \item \textbf{Case $p\in N_i$ and $q\in N_j$}: If $i=j$, then $d_{\Sp^{m+1}}(p,q)\leq\eta_m$ and $\widetilde{\phi}_{(m+1),m}(p)=\widetilde{\phi}_{(m+1),m}(q)=\iota_m^{-1}(u_i)$ so that $d_{\Sp^m}(\widetilde{\phi}_{(m+1),m}(p),\widetilde{\phi}_{(m+1),m}(q))=0$. Hence,
    $$\vert d_{\Sp^{m+1}}(p,q)-d_{\Sp^m}(\widetilde{\phi}_{(m+1),m}(p),\widetilde{\phi}_{(m+1),m}(q))\vert\leq\eta_m.$$ If $i\neq j$, then $d_{\Sp^{m+1}}(p,q)\leq\pi$ and $d_{\Sp^m}(\widetilde{\phi}_{(m+1),m}(p),\widetilde{\phi}_{(m+1),m}(q))=\zeta_m$ so that
    $$\vert d_{\Sp^{m+1}}(p,q)-d_{\Sp^m}(\widetilde{\phi}_{(m+1),m}(p),\widetilde{\phi}_{(m+1),m}(q))\vert\leq \zeta_m\leq\eta_m.$$
    
    \item \textbf{Case $p\in N_i$ and $q\in \iota_m(\seta(\Sp^m))$}: Then,
    \begin{align*}
        \vert d_{\Sp^{m+1}}(p,q)-d_{\Sp^m}(\widetilde{\phi}_{(m+1),m}(p),\widetilde{\phi}_{(m+1),m}(q))\vert&=\vert d_{\Sp^{m+1}}(p,q)-d_{\Sp^m}(\iota_m^{-1}(u_i),\iota_m^{-1}(q))\vert\\
        &=\vert d_{\Sp^{m+1}}(p,q)-d_{\Sp^{m+1}}(u_i,q)\vert\\
        &\leq d_{\Sp^{m+1}}(p,u_i)\leq\eta_m.
    \end{align*}
    
    \item \textbf{Case $p,q\in \iota_m(\seta(\Sp^m))$}: Then, $\widetilde{\phi}_{(m+1),m}(p)=p$ and $\widetilde{\phi}_{(m+1),m}(p)=q$. Hence,
    $$\vert d_{\Sp^{m+1}}(p,q)-d_{\Sp^m}(\widetilde{\phi}_{(m+1),m}(p),\widetilde{\phi}_{(m+1),m}(q))\vert=0\leq\eta_m.$$
\end{enumerate}
This implies that  $\dis(\widetilde{\phi}_{(m+1),m})\leq\eta_m$. Finally, by applying Lemma \ref{lemma:distortion} to $\widetilde{\phi}_{(m+1),m}$, we  construct the map $\phi_{(m+1),m}:\Sp^{m+1}\longtwoheadrightarrow\Sp^m$ such that $\dis(\phi_{(m+1),m})=\dis(\widetilde{\phi}_{(m+1),m})\leq\eta_m$. Moreover, by  construction,   $\phi_{(m+1),m}$ is obviously surjective and antipode preserving. Therefore,
$$\dgh(\Sp^m,\Sp^{m+1})\leq\frac{1}{2}\eta_m$$
as we required.
\end{proof}

\begin{remark}\label{rem:eta-distinct} 
Observe that, even though during the proof of Proposition \ref{prop:ub-n-m} we only established the fact $\dis(\phi_{(m+1),m})\leq\eta_m$, one can  check $\dis(\phi_{(m+1),m})$ is \emph{exactly equal} to $\eta_m$, since one can choose two points $p,q\in N_i$ such that $d_{\Sp^{m+1}}(p,q)$ is arbitrarily close to $\eta_m$.
\end{remark}        

\section{The proof of Proposition \ref{prop:ub13}}\label{sec:S1S3uppbdd}

In this section, we will prove Proposition \ref{prop:ub13} by constructing a specific correspondence between $\Sp^1$ and $\Sp^3$ with distortion less than or equal to $\zeta_1=\frac{2\pi}{3}$. The construction of this correspondence is based on the optimal correspondence $R_{2,1} = \mathrm{graph}(\phi_{2,1})$ between $\Sp^1$ and $\Sp^2$ identified in the  proof of Proposition \ref{prop:ub} given in \S\ref{sec:proof-ub} and some ideas reminiscent of the Hopf fibration. We will define a surjective map $\phi_{3,1}:\Sp^3\longtwoheadrightarrow \Sp^1$ by suitably ``rotating" the (optimal) surjection $\phi_{2,1}:\Sp^2\longtwoheadrightarrow \Sp^1$; see Figure \ref{fig:phi31}. 

\subsubsection*{Overview of the construction of $\phi_{3,1}$.}

The  diagram below  describes the construction of the map $\phi_{3,1}$ at a high level: 

\setlength\mathsurround{0pt}

$$\begin{tikzcd}
\Sp^3 \arrow[dd, "h",swap] \arrow[rr, "{\phi_{3,1}}", two heads,dotted,line width=0.75pt]                  & &   \Sp^1                                  \\
& &\\
{\Sp^2\times[0,\pi)} \arrow[rr, "\phi_{2,1}\times \mathrm{id}", two heads]  & & \Sp^1\times[0,\pi)\arrow[uu, "{T_\bullet}", two heads,swap] 
\end{tikzcd}$$

To an arbitrary $q\in\Sp^3$, we will be able to assign both a corresponding point $p_q\in\Sp^2$ and an angle $\alpha_q\in[0,\pi)$ giving rise to a map $h:\Sp^3\longrightarrow\Sp^2\times[0,\pi)$  such that $h(q):=(p_q,\alpha_q)$. Also,  $T_\bullet:\Sp^1\times [0,\pi)\longtwoheadrightarrow\Sp^1$  will be a map such that for each $\alpha\in[0,\pi),$ $T_\alpha$ is a  rotation of $\Sp^1$ by an angle $\alpha$. Then, as  described in the  diagram, for $q\in\Sp^3$, $\phi_{3,1}(q)$ will be defined as $T_{\alpha_q}\big(\phi_{2,1}(p_q)\big)$. Figures \ref{fig:phi31} and \ref{fig:hopf}  illustrate the construction.

Note that there is a certain degree of similarity between the map $\pi_1\circ h:\Sp^3\longtwoheadrightarrow\Sp^2$ (where $\pi_1$ is the canonical projection from $\Sp^2\times[0,\pi)$ to $\Sp^2$) and the ``Hopf fibration'', in the sense that the set  $(\pi_1\circ h)^{-1}(\{p,-p\})$ is isometric to $\Sp^1$ for  $p\in\Sp^2\backslash\Sp^1$ (whereas, $(\pi_1\circ h)^{-1}(\{p\})=\{p\}$ for  $p\in\Sp^1$).

\subsubsection*{Details.}
The following  coordinate representations  will be used throughout  this section:\footnote{Note that in comparison to the coordinate representation specified in \S\ref{sec:preliminaries}, here we are embedding $\Sp^1$, $\Sp^2$, and $\Sp^3$ into $\R^4$ in a certain way so that the emebddings $\Sp^1\hookrightarrow\Sp^2\hookrightarrow \Sp^3$ are also specific.}

\begin{itemize}
    \item $\Sp^1:=\{(x,y,0,0)\in\R^4:x^2+y^2=1\}$,\\
    \item $\Sp^2:=\{(x,y,z,0)\in\R^4:x^2+y^2+z^2=1\}$,\\
    \item $\Sp^3:=\{(x,y,z,w)\in\R^4:x^2+y^2+z^2+w^2=1\}$.
\end{itemize}

Also, we will use the map $\phi_{2,1}:\Sp^2\twoheadrightarrow\Sp^1$ and the regions  $N_1,N_2,N_3\subset\Sp^2$ constructed in the proof of Proposition \ref{prop:ub}, cf. \S\ref{sec:proof-ub}.

\begin{remark}\label{rmk:propsof on:hi21}
The following simple observations will be useful later. See Figure \ref{fig:phi21}.

\begin{enumerate}
    \item $\diam(\overline{N_i})\leq\zeta_1 = \frac{2\pi}{3}$ for any $i=1,2,3$. (This fact has been already mentioned during the proof of Proposition \ref{prop:ub-n-m}).
    
    \item If $p=(x,y,z,0)\in N_i$ and $q=(a,b,c,0)\in N_j$ for $(i,j)=(1,2),(2,3)\text{ or }(3,1)$ (resp. $(i,j)=(2,1),(3,2)\text{ or }(1,3)$), then $bx-ay\geq 0$ (resp. $\leq 0$) and $\phi_{2,1}(p),\phi_{2,1}(q)$ are in clockwise (resp. counterclockwise) order. 
\end{enumerate}
\end{remark}

Now, for any $\alpha\in\R$, consider the following rotation matrix:

$$T_\alpha:=\left(\begin{array}{cccc}\cos{\alpha} & -\sin{\alpha} & 0 & 0 \\ \sin{\alpha} & \cos{\alpha} & 0 & 0\\ 0 & 0 & \cos{\alpha} & -\sin{\alpha}\\ 0 & 0 & \sin{\alpha} & \cos{\alpha} \end{array} \right)$$

For any $p\in\Sp^3$, $T_\alpha p$ denotes the result of matrix multiplication by viewing $p$ as a $4$ by $1$ column vector according to the coordinate system described at the beginning of this section.

The following basic properties of these rotation matrices will be useful soon.

\begin{lemma}\label{lemma:rotationprops}
Let $\alpha,\beta\in\R$. Then,
\begin{enumerate}
    \item For any $q\in\Sp^3\backslash\Sp^1$, there are a unique $p_q\in\Sp^2\backslash\Sp^1$ and a unique $\alpha_q\in[0,\pi)$ such that $q=T_{\alpha_q} p_q$. In particular, $\alpha_q=0$ if and only if $q\in\Sp^2\backslash\Sp^1$.
    
    \item Both of $\Sp^1$ and $\Sp^3$ are invariant with respect to the action of the  rotation matrices $T_\alpha$.
    
    \item $T_\alpha\,T_\beta=T_{\alpha+\beta}$.
    
    \item $d_{\Sp^3}(T_\alpha\,p,T_\alpha\,q)=d_{\Sp^3}(p,q)$ for any $p,q\in\Sp^3$.
    
    \item $d_{\Sp^3}(T_\alpha\,p,p)=\alpha$ for any $p\in\Sp^3$ and $\alpha\in[0,\pi]$. 
    
    \item $d_{\Sp^3}(T_\alpha(-p),p)=\pi-\alpha$ for any $p\in\Sp^3$ and $\alpha\in[0,\pi]$.
\end{enumerate}
\end{lemma}
\begin{proof}
\begin{enumerate}
    \item Let $q=(x',y',z',w')\in\Sp^3\backslash\Sp^1$. Since $q$ is not in $\Sp^1$, we know that $(z')^2+(w')^2>0$. Then, there exists a unique $\alpha_q\in[0,\pi)$ and $z\in\R\backslash\{0\}$ such that
    $$\left(\begin{array}{c} z' \\ w' \end{array} \right)=\left(\begin{array}{cc}\cos{\alpha_q} & -\sin{\alpha_q} \\ \sin{\alpha_q} & \cos{\alpha_q} \end{array} \right)\left(\begin{array}{c} z \\ 0 \end{array} \right)$$
    i.e. $z^2=(z')^2+(w')^2$.  Then, this $\alpha_q$ is the required angle and we choose the unique point $p_q=(x,y,z,0)\in\Sp^2\backslash\Sp^1$ so that
    $$\left(\begin{array}{c} x'\\ y'\\ z' \\ w' \end{array} \right)=\left(\begin{array}{cccc}\cos{\alpha_q} & -\sin{\alpha_q} & 0 & 0 \\ \sin{\alpha_q} & \cos{\alpha_q} & 0 & 0\\ 0 & 0 & \cos{\alpha_q} & -\sin{\alpha_q}\\ 0 & 0 & \sin{\alpha_q} & \cos{\alpha_q} \end{array} \right)\left(\begin{array}{c} x\\ y\\ z \\ 0 \end{array} \right).$$

Since $T_{\alpha_q}$ is the identity matrix when $\alpha_q=0$, then, obviously, $\alpha_q=0$ if and only if $q\in\Sp^2\backslash\Sp^1$.
    
    \item Obvious.
    
    \item Obvious.
    
    \item This item is equivalent to the condition $\langle T_\alpha p,T_\alpha q\rangle=\langle p,q \rangle$, and it can be easily checked by direct computation.
    
    \item This item is equivalent to the condition $\langle T_\alpha p, p\rangle=\cos\alpha$, and it can be easily checked by direct computation.
    
    \item This item is equivalent to the condition $\langle T_\alpha(-p), p\rangle=-\cos\alpha$, and it can be easily checked by direct computation.
\end{enumerate}
\end{proof}

\subsubsection*{Additional details and the proof of Proposition \ref{prop:ub13}.}
We need a few more definitions and technical lemmas for the proof of Proposition \ref{prop:ub13}. We in particular make the following definitions for notational convenience:
\begin{itemize}
    \item For any $p,q\in\Sp^2$,
    \begin{align*}
        E_{p,q}:[0,\pi]&\longrightarrow[-1,1]\\
        \alpha&\longmapsto\langle T_\alpha\,p,q\rangle
    \end{align*}
    
    \item For any $p,q\in\Sp^2$,
    \begin{align*}
        F_{p,q}:[0,\pi]&\longrightarrow\R\\
        \alpha&\longmapsto d_{\Sp^3}(T_\alpha\,p,q)-\alpha
    \end{align*}
    
    \item For any $p,q\in\Sp^2$,
    \begin{align*}
        G_{p,q}:[0,\pi]&\longrightarrow\R\\
        \alpha&\longmapsto d_{\Sp^3}(T_\alpha\,p,q)+\alpha
    \end{align*}
\end{itemize}

\begin{lemma}\label{lemma:propsofEFG}
For any $p=(x,y,z,0)\in\Sp^2\backslash\Sp^1$ and $q=(a,b,c,0)\in\Sp^2$,
\begin{enumerate}
    \item $E_{p,q}(\alpha)\in (-1,1)$ for any $\alpha\in (0,\pi)$.
    
    \item $(E_{p,q}'(\alpha))^2+(E_{p,q}(\alpha))^2\leq 1$ for any $\alpha\in [0,\pi]$.\footnote{Here $E'_{pq}$ denotes the derivative of $E_{p,q}.$} 
    
    \item $F_{p,q}$ is a  non-increasing function. Therefore, $-d_{\Sp^2}(p,q)\leq F_{p,q}(\alpha)\leq d_{\Sp^2}(p,q)$ for any $\alpha\in [0,\pi]$.
    
    \item $G_{p,q}$ is a  non-decreasing function. Therefore, $d_{\Sp^2}(p,q)\leq G_{p,q}(\alpha)\leq 2\pi-d_{\Sp^2}(p,q)$ for any $\alpha\in [0,\pi]$.
\end{enumerate}
\end{lemma}
\begin{proof}
\begin{enumerate}
    \item Suppose not so that $E_{p,q}(\alpha)=\pm 1$. This implies that $T_\alpha p=q\text{ or }-q\in\Sp^2$, but that cannot be true because $T_\alpha p\in \Sp^3\backslash\Sp^2$ by Lemma \ref{lemma:rotationprops} item (1) and because of the range of $\alpha$. So, it is contradiction hence we have $E_{p,q}(\alpha)\in (-1,1)$ as we required.
    
    \item As a result of direct computation, we know that
    $$E_{p,q}(\alpha)=\langle p,q \rangle\cos{\alpha}+(bx-ay)\sin{\alpha}.$$
    Here, observe that $bx-ay$ is the 3rd coordinate of the cross product $(x,y,z)\times(a,b,c)$. In particular, this implies $\vert bx-ay \vert\leq\Vert (x,y,z)\times(a,b,c) \Vert=\sin\beta$ where $\langle p,q \rangle=\cos\beta$. Therefore,
    \begin{align*}
    (E_{p,q}'(\alpha))^2+(E_{p,q}(\alpha))^2&=\langle p,q \rangle^2+(bx-ay)^2\leq \cos^2\beta+\sin^2\beta=1.
    \end{align*}
    
    \item Note that $F_{p,q}(\alpha)=\arccos(E_{p,q}(\alpha))-\alpha$. Hence, for any $\alpha\in (0,\pi)$, 
    \begin{align*}
        F_{p,q}'(\alpha)=-\frac{E_{p,q}'(\alpha)}{\sqrt{1-(E_{p,q}(\alpha))^2}}-1.
    \end{align*}
    Observe that this expression is well-defined by (1). Also, by (2),
    \begin{align*}
        (E_{p,q}'(\alpha))^2+(E_{p,q}(\alpha))^2\leq 1&\Leftrightarrow -E_{p,q}'(\alpha)\leq\sqrt{1-(E_{p,q}(\alpha))^2}\\
        &\Leftrightarrow F_{p,q}'(\alpha)=-\frac{E_{p,q}'(\alpha)}{\sqrt{1-(E_{p,q}(\alpha))^2}}-1\leq 0.
    \end{align*}
    Hence, $F_{p,q}$ is a non-increasing function. Also, since $F_{p,q}(0)=d_{\Sp^2}(p,q)$ and $F_{p,q}(\pi)=d_{\Sp^3}(T_\pi p,q)-\pi=d_{\Sp^2}(-p,q)-\pi=(\pi-d_{\Sp^2}(p,q))-\pi=-d_{\Sp^2}(p,q)$,
    $$-d_{\Sp^2}(p,q)\leq F_{p,q}(\alpha)\leq d_{\Sp^2}(p,q).$$
    
    \item Note that $G_{p,q}(\alpha)=\arccos(E_{p,q}(\alpha))+\alpha$. Hence, for any $\alpha\in (0,\pi)$, 
    \begin{align*}
        G_{p,q}'(\alpha)=-\frac{E_{p,q}'(\alpha)}{\sqrt{1-(E_{p,q}(\alpha))^2}}+1.
    \end{align*}
    Observe that this expression is well-defined by equation (1). Also, by equation  (2),
    \begin{align*}
        (E_{p,q}'(\alpha))^2+(E_{p,q}(\alpha))^2\leq 1&\Leftrightarrow E_{p,q}'(\alpha)\leq\sqrt{1-(E_{p,q}(\alpha))^2}\\
        &\Leftrightarrow G_{p,q}'(\alpha)=-\frac{E_{p,q}'(\alpha)}{\sqrt{1-(E_{p,q}(\alpha))^2}}+1\geq 0.
    \end{align*}
    Hence, $G_{p,q}$ is non-decreasing function. Also, since $G_{p,q}(0)=d_{\Sp^2}(p,q)$ and $G_{p,q}(\pi)=d_{\Sp^3}(T_\pi p,q)+\pi=d_{\Sp^2}(-p,q)+\pi=(\pi-d_{\Sp^2}(p,q))+\pi=2\pi-d_{\Sp^2}(p,q)$,
    $$d_{\Sp^2}(p,q)\leq G_{p,q}(\alpha)\leq 2\pi-d_{\Sp^2}(p,q).$$
\end{enumerate}
\end{proof}

\begin{lemma}\label{lemma:rotationdistprops}
For any $p=(x,y,z,0),q=(a,b,c,0)\in\Sp^2\backslash\Sp^1$,
\begin{enumerate}
    \item If $p\in N_i$ and $q\in N_j$ for $(i,j)=(1,2),(2,3)\text{ or }(3,1)$, then we have
    $$d_{\Sp^3}(T_{\frac{2\pi}{3}}p,q)\leq\frac{2\pi}{3}.$$
    
    \item If $p\in N_i$ and $q\in N_j$ for $(i,j)=(2,1),(3,2)\text{ or }(1,3)$, then we have
    $$d_{\Sp^3}(T_{\frac{\pi}{3}}p,q)\geq\frac{\pi}{3}.$$
\end{enumerate}
\end{lemma}
\begin{proof}
\begin{enumerate}
    \item First, observe that $bx-ay\geq 0$ by the item (2) of Remark \ref{rmk:propsof on:hi21}. Hence, 
    \begin{align*}
        E_{p,q}\left(\frac{2\pi}{3}\right)=\langle T_{\frac{2\pi}{3}} p,q \rangle&=-\frac{1}{2}\langle p,q \rangle+\frac{\sqrt{3}}{2}(bx-ay)\\
        &\geq -\frac{1}{2}\langle p,q \rangle\geq -\frac{1}{2}.
    \end{align*}
    Therefore,
    $$d_{\Sp^3}(T_{\frac{2\pi}{3}}p,q)=\arccos \left(E_{p,q}\left(\frac{2\pi}{3}\right)\right)\leq\arccos\left(-\frac{1}{2}\right)=\frac{2\pi}{3}.$$
    
    \item The proof of this case is  similar to the proof of the case (1) of this Lemma, so we omit it.
\end{enumerate}
\end{proof}

\begin{figure}
	\centering
    \includegraphics[width=0.6\textheight]{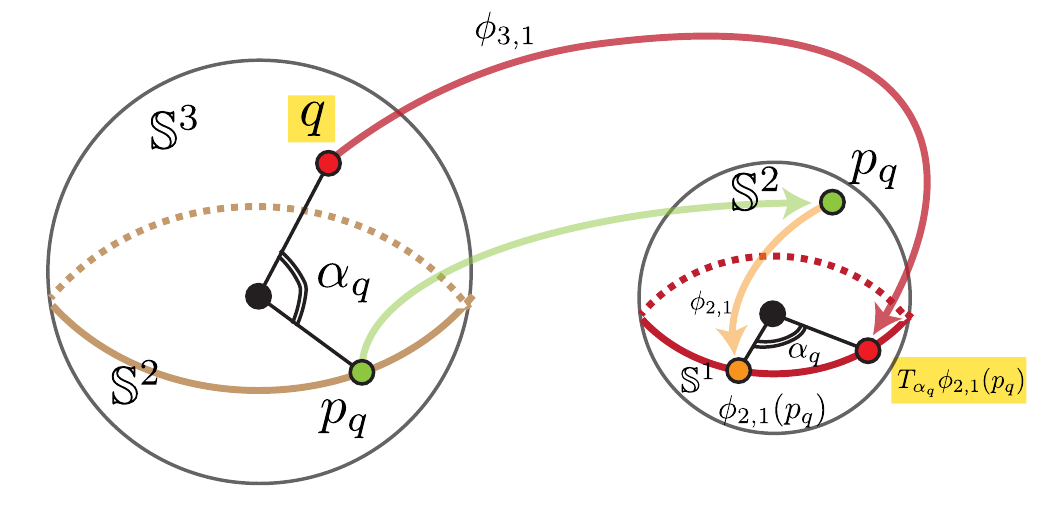}
\caption{The definition of $\phi_{3,1}$: given $q\in \Sp^3\backslash\Sp^2$ there exists a unique angle $\alpha_q\in(0,\pi)$ and unique point $p_q \in \Sp^2\backslash\Sp^1$ such that $q = T_{\alpha_q} p_q$. Then, we consider the point $\phi_{2,1}(p_q)\in\Sp^1$ and define $\phi_{3,1}(q):= T_{\alpha_q}\phi_{2,1}(p_q)$. That $\phi_{3,1}(q)\in \Sp^1$ follows from Lemma \ref{lemma:rotationprops} item (2).}
\label{fig:phi31}
\end{figure}

\begin{proof}[Proof of Proposition \ref{prop:ub13}]
Note that it is enough to find a surjective map $\phi_{3,1}:\Sp^3\twoheadrightarrow\Sp^1$ such that $\dis(\phi_{3,1})\leq\zeta_1=\frac{2\pi}{3}$ since this map gives rise to a correspondence $R_{3,1}:=\mathrm{graph}(\phi_{3,1})$ with $\dis(R_{3,1}) = \dis(\phi_{3,1})\leq \zeta_1$.

We  construct the required surjective map \mbox{$\phi_{3,1}:\Sp^3\longtwoheadrightarrow\Sp^1$}  as follows:
\begin{align*}
    q&\longmapsto\begin{cases}\phi_{2,1}(q)&\textit{ if }q\in\Sp^2 \\ T_{\alpha_q} \phi_{2,1}(p_q)&\textit{ if }q\in\Sp^3\backslash\Sp^2\text{ and }q=T_{\alpha_q}\, p_q\\
    & \quad\textit{for the \emph{unique} such  }\alpha_q\in(0,\pi)\textit{ and }p_q\in\Sp^2\backslash\Sp^1.\end{cases}
\end{align*}

Note that $\phi_{3,1}$ is surjective, since $\phi_{3,1}\vert_{\Sp^2}=\phi_{2,1}$ and $\phi_{2,1}$ is surjective.

See Figures \ref{fig:phi31} and \ref{fig:hopf} for an explanation of the construction of the map $\phi_{3,1}$.

\begin{figure}
    \centering
    \includegraphics[width = 0.7\linewidth]{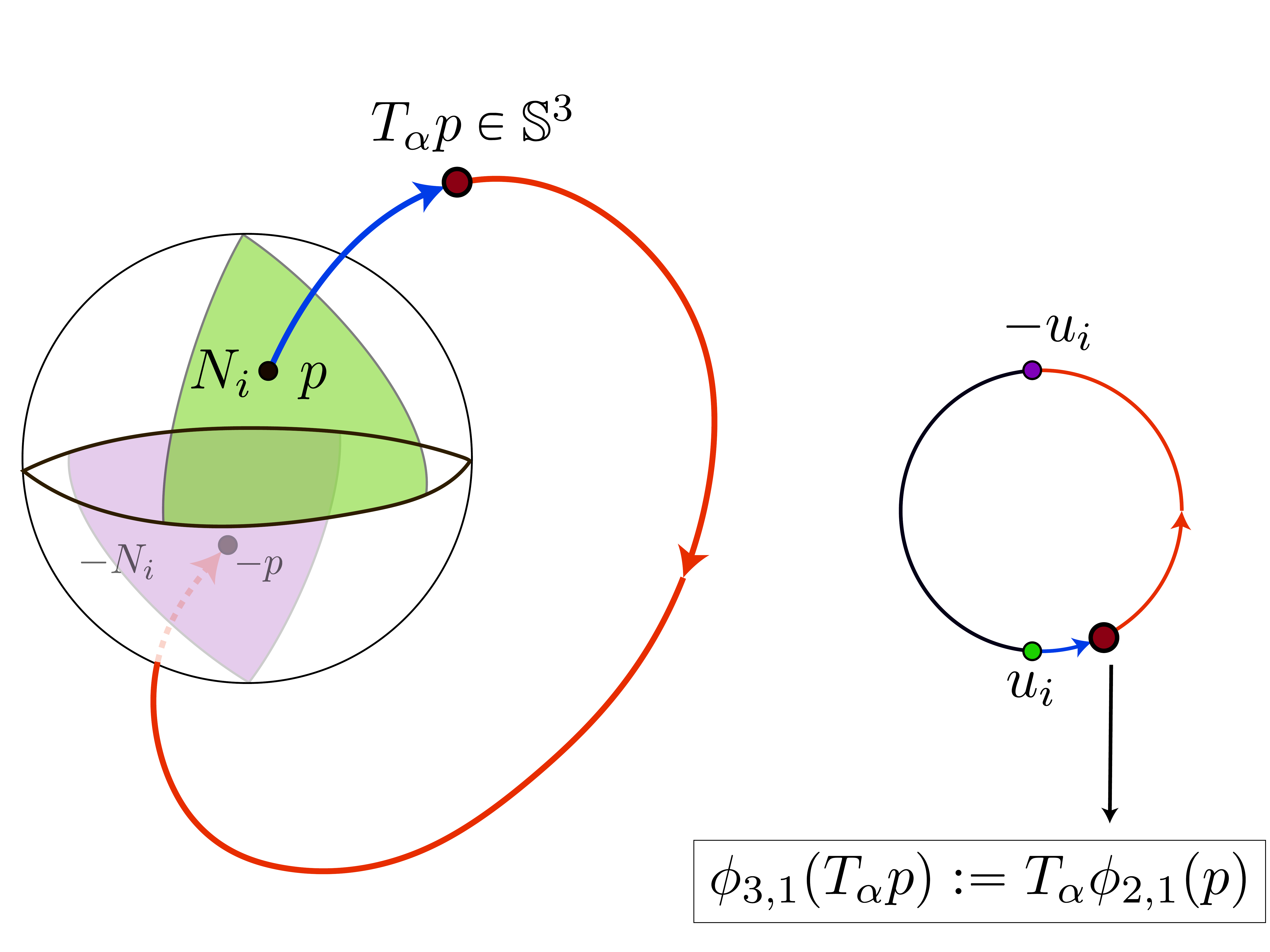}
    \caption{The definition of the map $\phi_{3,1}$ via the map $\phi_{2,1}$. The point $T_\alpha p$ on $\Sp^3$ is mapped to the point $T_\alpha \phi_{2,1}(p)$ on $\Sp^1$. The antipode preserving map $\phi_{2,1}$ maps the whole region $N_i$ to the point $u_i$.}
    \label{fig:hopf}
\end{figure}

Let us  now verify that
$$\vert d_{\Sp^3}(q,q')-d_{\Sp^1}(\phi_{3,1}(q),\phi_{3,1}(q'))\vert\leq\zeta_1$$
for every $q,q'\in\Sp^3$. Without loss of generality, we can assume that $q=T_\alpha p$, $q'=T_\beta p'$ for some $p,p'\in\Sp^2$ and $0\leq\beta\leq\alpha<\pi$. Then,
\begin{align*}
    \vert d_{\Sp^3}(q,q')-d_{\Sp^1}(\phi_{3,1}(q),\phi_{3,1}(q'))\vert&=\vert d_{\Sp^3}(T_\alpha p,T_\beta p')-d_{\Sp^1}(T_\alpha \phi_{2,1}(p),T_\beta \phi_{2,1}(p'))\vert\\
    &=\vert d_{\Sp^3}(T_{(\alpha-\beta)} p,p')-d_{\Sp^1}(T_{(\alpha-\beta)} \phi_{2,1}(p),\phi_{2,1}(p'))\vert
\end{align*}

Hence, it is enough to prove
\begin{equation}\label{eq:phi31dist}
    \big\vert d_{\Sp^3}(T_\alpha p,q)-d_{\Sp^1}(T_\alpha \phi_{2,1}(p),\phi_{2,1}(q))\big\vert\leq\zeta_1
\end{equation}

for any $p,q\in\Sp^2$ and $\alpha\in[0,\pi)$.

If $p\in\Sp^1$, then $\phi_{2,1}(p)=p$. Hence,
\begin{align*}
    \vert d_{\Sp^3}(T_\alpha p,q)-d_{\Sp^1}(T_\alpha \phi_{2,1}(p),\phi_{2,1}(q))\vert&=\vert d_{\Sp^3}(T_\alpha p,q)-d_{\Sp^1}(T_\alpha p,\phi_{2,1}(q))\vert\\
    &\leq d_{\Sp^3}(q,\phi_{2,1}(q))\leq\zeta_1
\end{align*}
where the last inequality holds by item (1) of Remark \ref{rmk:propsof on:hi21}. One can carry out a similar computation if $q\in\Sp^1$. So, let's assume $p=(x,y,z,0),q=(a,b,c,0)\in\Sp^2\backslash\Sp^1$. Furthermore, since $\phi_{2,1}$ is antipode preserving, it is enough to check  inequality (\ref{eq:phi31dist}) only for $p,q\in \mathbf{H}_{>0}(\Sp^2)$. We do this by following the same idea as in the proof of Lemma \ref{lemma:distortion}.

We do a case by case analysis.
\begin{enumerate}
    \item \textbf{Case $p\in N_i$ and $q\in N_j$ for $(i,j)=(1,2),(2,3)\text{ or }(3,1)$}: By  item (2) of Remark \ref{rmk:propsof on:hi21}, the two points $\phi_{2,1}(p)$ and $\phi_{2,1}(q)$ are in clockwise order. Hence,
    \begin{align*}
        d_{\Sp^1}(T_\alpha \phi_{2,1}(p),\phi_{2,1}(q))=\begin{cases}\frac{2\pi}{3}-\alpha&\text{if }\alpha\in\left[0,\frac{2\pi}{3}\right]\\ \alpha-\frac{2\pi}{3}&\text{if }\alpha\in\left[\frac{2\pi}{3},\pi\right).\end{cases}
    \end{align*}
    
    Consider first the case when $\alpha\in\left[0,\frac{2\pi}{3}\right]$. We have to prove that 
    $$-\frac{2\pi}{3}\leq d_{\Sp^3}(T_\alpha p,q)-\left(\frac{2\pi}{3}-\alpha\right)\leq\frac{2\pi}{3}.$$
    Equivalently, we have to prove
    $$0\leq G_{p,q}(\alpha)\leq\frac{4\pi}{3}.$$
    The left-hand side inequality is obvious since $G_{p,q}(\alpha)\geq d_{\Sp^2}(p,q)\geq 0$ by Lemma \ref{lemma:propsofEFG} item (4). The right-hand side inequality is true by Lemma \ref{lemma:propsofEFG} item (4) and Lemma \ref{lemma:rotationdistprops} item (1).

    Next, consider the case when $\alpha\in\left[\frac{2\pi}{3},\pi\right)$. We have to prove
    $$-\frac{2\pi}{3}\leq d_{\Sp^3}(T_\alpha p,q)-\left(\alpha-\frac{2\pi}{3}\right)\leq\frac{2\pi}{3}.$$
    Equivalently, we have to prove
    $$-\frac{4\pi}{3}\leq F_{p,q}(\alpha)\leq 0.$$
    The leftmost  inequality is obvious since $F_{p,q}(\alpha)\geq -d_{\Sp^2}(p,q)\geq-\frac{4\pi}{3}$ by Lemma \ref{lemma:propsofEFG} item (3). The right-hand side inequality is true by Lemma \ref{lemma:propsofEFG} item (3) and Lemma \ref{lemma:rotationdistprops} item (1).
    
    \item \textbf{Case $p\in N_i$ and $q\in N_j$ for $(i,j)=(2,1),(3,2)\text{ or }(1,3)$}: Almost the same as the case (1) except we use the item (2) of Lemma \ref{lemma:rotationdistprops}.
    
    \item \textbf{Case $p,q\in N_i$ for $i=1,2,3$}: In this case, $d_{\Sp^1}(T_\alpha \phi_{2,1}(p),\phi_{2,1}(q))=\alpha$ since $\phi_{2,1}(p)=\phi_{2,1}(q)$ and Lemma \ref{lemma:rotationprops} item (5). Hence, we have to show
    $$-\frac{2\pi}{3}\leq d_{\Sp^3}(T_\alpha p,q)-\alpha=F_{p,q}(\alpha) \leq \frac{2\pi}{3}.$$
    But, this is obvious by the item (1) of Remark \ref{rmk:propsof on:hi21} and the item (3) of Lemma \ref{lemma:propsofEFG}.
\end{enumerate}

So, indeed $\dis(\phi_{3,1})\leq\zeta_1$ as we wanted.
\end{proof}

\section{The proof of Proposition \ref{prop:ub23}}\label{sec:proof-prop-ub23}

In this section we provide a construction of an optimal correspondence, $R_{3,2}$, between $\Sp^3$ and $\Sp^2$. The structure of this correspondence is  different from those of the ones described in the proofs of Propositions \ref{prop:ub} and \ref{prop:ub-n-m}. As a matter of fact, as Remark \ref{rem:eta-distinct} mentions, the distortion of the surjection $\phi_{(m+1),m}:\Sp^{m+1}\longtwoheadrightarrow\Sp^m$ constructed in Proposition \ref{prop:ub-n-m} is \emph{exactly equal} to $\eta_m$. Since $\zeta_2<\eta_2$ this means that a different construction is required for the case $m=2$.

Let $u_1,u_2,u_3,$ and $u_4$ be the vertices of a regular tetrahedron inscribed in $\Sp^2$ (i.e., $\langle u_i,u_j \rangle=-\frac{1}{3}=\cos\zeta_2$ for any $i\neq j$). We consider:
\begin{center}
\begin{tabular}{ l l }
 $u_1=(1,0,0),$ & $u_2=\left(-\frac{1}{3},\frac{2\sqrt{2}}{3},0\right),$\\ 
 $u_3=\left(-\frac{1}{3},-\frac{\sqrt{2}}{3},\frac{\sqrt{2}}{\sqrt{3}}\right),$ & $u_4=\left(-\frac{1}{3},-\frac{\sqrt{2}}{3},-\frac{\sqrt{2}}{\sqrt{3}}\right).$ 
\end{tabular}
\end{center}

Now, let $V_1,V_2,V_3,\text{ and }V_4\subset\Sp^2$ be the Voronoi partition of $\Sp^2$ induced by $u_1,u_2,u_3,$ and $u_4$. Then, for each $i$, $\overline{V_i}$ is the spherical convex hull of the set $\{-u_j\in\Sp^2:j\in\{1,2,3,4\}\backslash\{i\}\}$. Let $$r:=\arccos\left(\frac{2\sqrt{2}}{3}\right).$$ For  $i\neq j\in\{1,2,3,4\}$, let $u_{i,j}$ be the point on the shortest geodesic between $u_i$ and $-u_j$ such that $d_{\Sp^2}(u_i,u_{i,j})=r$. See Figure \ref{fig:Nidecription} for an illustration of $V_1$. 

\begin{figure}
\begin{center}
\includegraphics[width=0.6\linewidth]{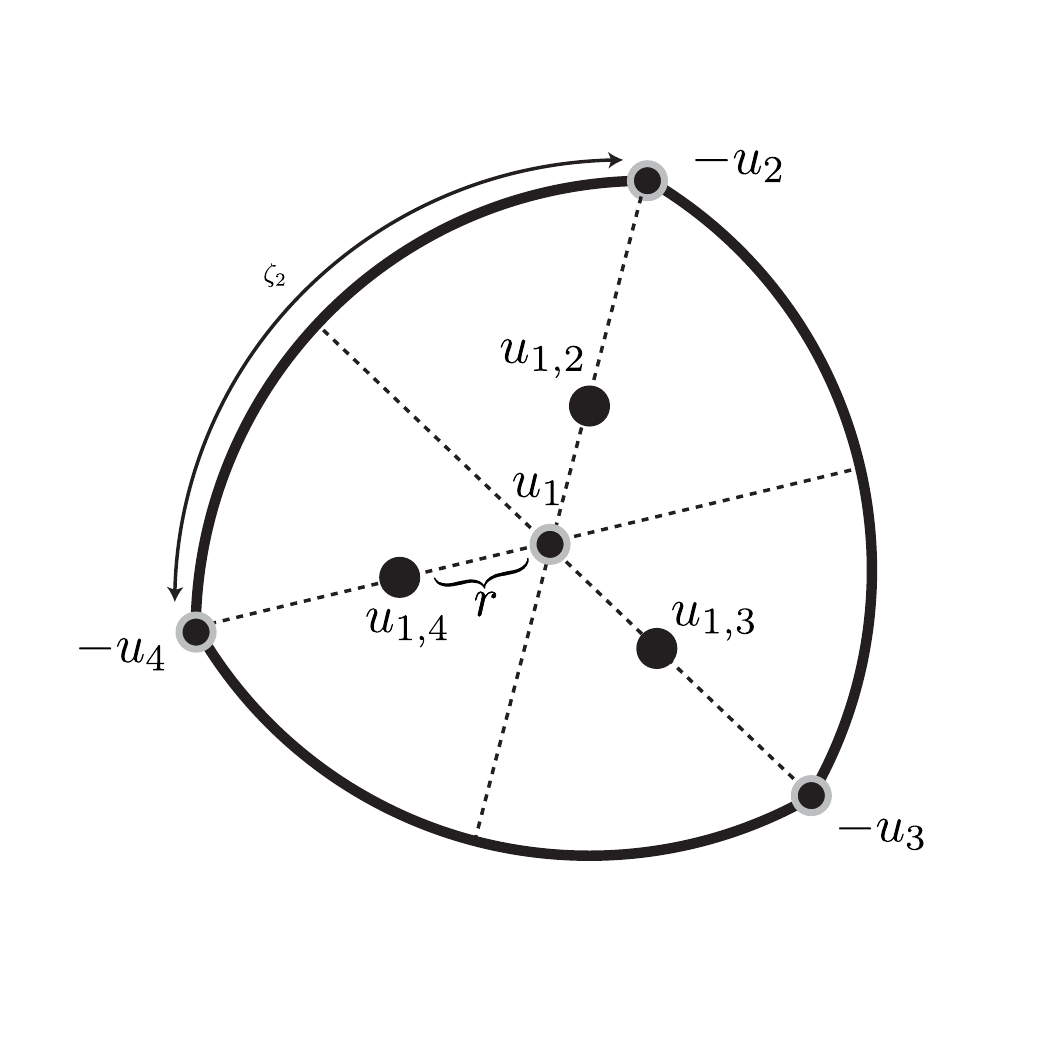}
\end{center}
\caption{Illustration of $V_i$ for $i=1$. All the sides of the spherical triangle $V_1$ (determined by the three points $-u_2$, $-u_3$, and $-u_4$) have the same length $\zeta_2$.\label{fig:Nidecription}}
\end{figure}

\begin{remark}\label{rmk:cptuij}
One can directly compute the following coordinates: 
\begin{center}
\begin{tabular}{ l l l }
 $u_{1,2}=\left(\frac{2\sqrt{2}}{3},-\frac{1}{3},0\right),$ & $u_{1,3}=\left(\frac{2\sqrt{2}}{3},\frac{1}{6},-\frac{1}{2\sqrt{3}}\right),$ & $u_{1,4}=\left(\frac{2\sqrt{2}}{3},\frac{1}{6},\frac{1}{2\sqrt{3}}\right),$ \\ 
 $u_{2,1}=\left(-\frac{4\sqrt{2}}{9},\frac{7}{9},0\right),$ & $u_{2,3}=\left(-\frac{\sqrt{2}}{9},\frac{17}{18},-\frac{1}{2\sqrt{3}}\right),$ & $u_{2,4}=\left(-\frac{\sqrt{2}}{9},\frac{17}{18},\frac{1}{2\sqrt{3}}\right).$   
\end{tabular}
\end{center}

\end{remark}

\begin{lemma}\label{lemma:distuijs}
For any $i\neq j\in\{1,2,3,4\}$, the following  hold:
\begin{enumerate}
    \item $\langle u_{i,k},u_{i,l} \rangle=\frac{5}{6}$ for any $k\neq l\in{1,2,3,4}\backslash\{i\}$.
    
    \item $\langle u_{i,k},u_{j,k} \rangle=\frac{5}{54}$ for any $k\in{1,2,3,4}\backslash\{i,j\}$.
    
    \item $\langle u_{i,k},u_{j,l} \rangle=-\frac{2}{27}$ for any $k\neq l\in{1,2,3,4}\backslash\{i,j\}$.
    
    \item $\langle u_{i,k},u_{j,i} \rangle=-\frac{25}{54}$ for any $k\in{1,2,3,4}\backslash\{i,j\}$.
    
    \item $\langle u_{i,j},u_{j,i} \rangle=-\frac{23}{27}$.
    
    \item $\langle u_i,u_{j,k} \rangle=-\frac{\sqrt{2}}{9}$ for any $k\in{1,2,3,4}\backslash\{i,j\}$.
    
    \item $\langle u_i,u_{j,i} \rangle=-\frac{4\sqrt{2}}{9}$.
\end{enumerate}
\end{lemma}
\begin{proof}
By symmetry, without loss of generality one can assume $i=1$ and $j=2$. Then, use the coordinate values given in Remark \ref{rmk:cptuij}.
\end{proof}

Next, for each $i$, let $\{V_{i,j}\subset V_i:j\in\{1,2,3,4\}\backslash\{i\}\}$ be the Voronoi partition of $V_i$ induced by $\{u_{i,j}\in V_i:j\in\{1,2,3,4\}\backslash\{i\}\}$.\vspace{\baselineskip}

From now on, in this section, we will identify $\Sp^2$ with $\sete(\Sp^3)\subset\Sp^3$. Then, obviously $$\mathbf{H}_{\geq 0}(\Sp^3)=\mathcal{C}(V_1)\cup \mathcal{C}(V_2)\cup \mathcal{C}(V_3)\cup \mathcal{C}(V_4).$$ Moreover, for any $i\in\{1,2,3,4\}$ and $\alpha\in\left[0,\frac{\pi}{2}\right]$, we  divide $\mathcal{C}(V_i)$ in the following way:

\begin{align*}
    & \mathcal{C}_\alpha^\mathrm{top}(V_i):=\{p\in \mathcal{C}(V_i):d_{\Sp^{n+1}}(e_4,p)\leq\alpha\},\\
    & \mathcal{C}_\alpha^\mathrm{bot}(V_i):=\{p\in \mathcal{C}(V_i):d_{\Sp^{n+1}}(e_4,p)>\alpha\},\\
    & \mathcal{C}_\alpha^\mathrm{bot}(V_{i,j}):=\{p\in \mathcal{C}(V_i):d_{\Sp^{n+1}}(e_4,p)>\alpha\text{ and }\Omega(p)\in V_{i,j}\}\text{ for any }j\in\{1,2,3,4\}\backslash\{i\}.
\end{align*}
where
\begin{align*}
    \Omega:\mathbf{H}_{\geq 0}(\Sp^3)\backslash\{e_4\}&\longrightarrow \sete(\Sp^3)=\Sp^2\\
    (x,y,z,w)&\longmapsto\frac{1}{\sqrt{1-w^2}}(x,y,z,0)
\end{align*}
is the orthogonal projection onto the equator. Then, obviously $$\mathcal{C}(V_i)=\mathcal{C}_\alpha^\mathrm{top}(V_i)\cup\bigcup_{j\in\{1,2,3,4\}\backslash\{i\}}\mathcal{C}_\alpha^\mathrm{bot}(V_{i,j})$$ for each $i\in\{1,2,3,4\}$. See Figure \ref{fig:CV1} (\subref{fig:CNitopbot}) and Figure \ref{fig:CV1}  (\subref{fig:CNibot}) for  illustrations of $\mathcal{C}_\alpha^\mathrm{top}(V_1)$, $\mathcal{C}_\alpha^\mathrm{bot}(V_1)$, $\mathcal{C}_\alpha^\mathrm{bot}(V_{1,2})$, $\mathcal{C}_\alpha^\mathrm{bot}(V_{1,3})$, and $\mathcal{C}_\alpha^\mathrm{bot}(V_{1,4})$.

\begin{figure}
\centering
  \begin{subfigure}[b]{0.4\textwidth}
  \centering
    \includegraphics[width=\textwidth]{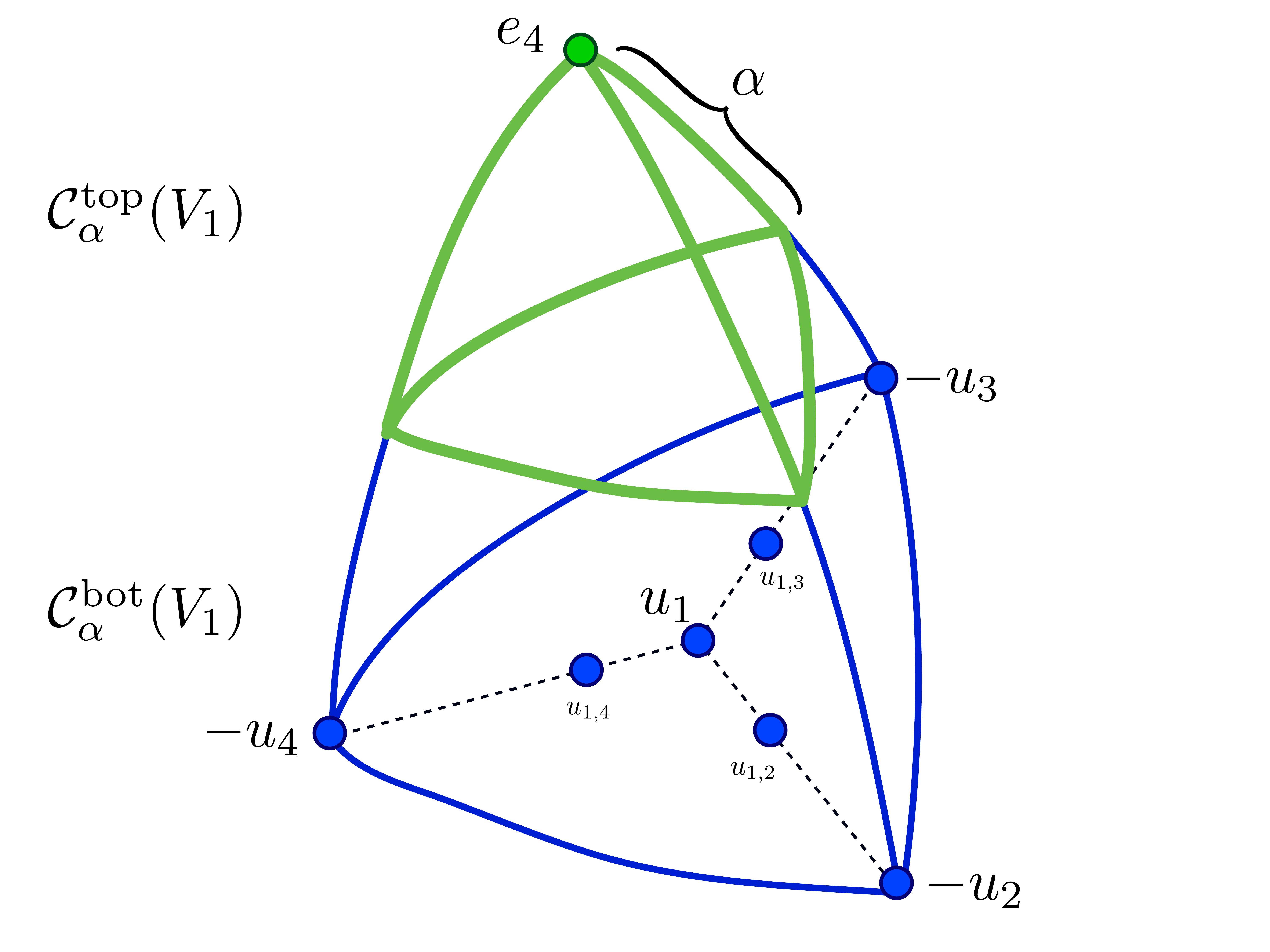}
    \caption{$\mathcal{C}_\alpha^\mathrm{top}(V_1)$ and  $\mathcal{C}_\alpha^\mathrm{bot}(V_1)$}
    \label{fig:CNitopbot}
  \end{subfigure}
  \begin{subfigure}[b]{0.4\textwidth}
  \centering
    \includegraphics[width=\textwidth]{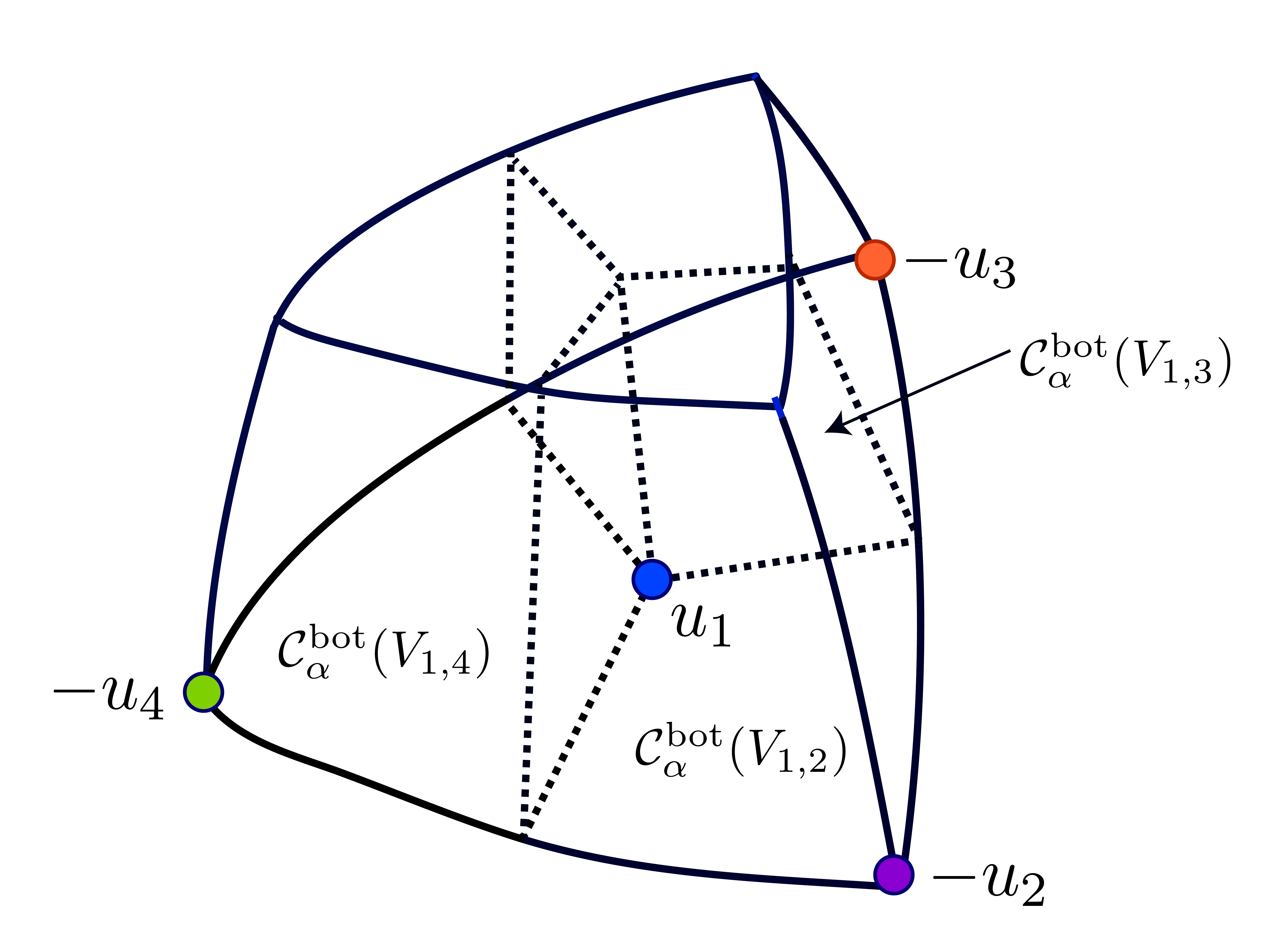}
    \caption{\mbox{$\mathcal{C}_\alpha^\mathrm{bot}(V_{1,2})$, $\mathcal{C}_\alpha^\mathrm{bot}(V_{1,3})$, and  $\mathcal{C}_\alpha^\mathrm{bot}(V_{1,4})$}}
    \label{fig:CNibot}
  \end{subfigure}
  \caption{The regions into which $\mathcal{C}(V_1)$ is split.}\label{fig:CV1}
\end{figure}

\begin{lemma}\label{lemma:S3S2techineq}
For $p,q\in \mathbf{H}_{\geq 0}(\Sp^3)$, the following inequalities hold:
\begin{enumerate}
    \item If $p,q\in \mathcal{C}_\alpha^\mathrm{top}(V_i)$ for some $i\in\{1,2,3,4\}$, then
$$\langle p,q\rangle\geq\cos^2\alpha-\frac{1}{\sqrt{3}}\sin^2\alpha=\left(1+\frac{1}{\sqrt{3}}\right)\cos^2\alpha-\frac{1}{\sqrt{3}}.$$
In particular, this is equivalent to $$d_{\Sp^3}(p,q)\leq\arccos\left(\left(1+\frac{1}{\sqrt{3}}\right)\cos^2\alpha-\frac{1}{\sqrt{3}}\right).$$

    \item If $p\in \mathcal{C}_\alpha^\mathrm{top}(V_i)$ and $q\in \mathcal{C}_\alpha^\mathrm{bot}(V_{j,i})$ for some $i\neq j\in\{1,2,3,4\}$, then
$$\langle p,q\rangle\leq\sqrt{\frac{2}{3}\cos^2\alpha+\frac{1}{3}}.$$
In particular, this is equivalent to $$d_{\Sp^3}(p,q)\geq\arccos\left(\sqrt{\frac{2}{3}\cos^2\alpha+\frac{1}{3}}\right).$$

    \item If $p\in \mathcal{C}_\alpha^\mathrm{bot}(V_{i,k})$ and $q\in \mathcal{C}_\alpha^\mathrm{bot}(V_{j,i})$ for some $i\neq j\in\{1,2,3,4\}$ and $k\in\{1,2,3,4\}\backslash\{i,j\}$, then
$$\langle p,q\rangle\leq\left(1-\frac{1}{\sqrt{3}}\right)\cos^2\alpha+\frac{1}{\sqrt{3}}$$
In particular, this is equivalent to the condition $$d_{\Sp^3}(p,q)\geq\arccos\left(\left(1-\frac{1}{\sqrt{3}}\right)\cos^2\alpha+\frac{1}{\sqrt{3}}\right).$$

    \item If $p\in \mathcal{C}_\alpha^\mathrm{bot}(V_{i,j})$ and $q\in \mathcal{C}_\alpha^\mathrm{bot}(V_{j,i})$ for some $i\neq j\in\{1,2,3,4\}$, then
$$\langle p,q\rangle\leq\cos^2\alpha.$$
In particular, this is equivalent to
$$d_{\Sp^3}(p,q)\geq\arccos(\cos^2\alpha).$$
\end{enumerate}
\end{lemma}
\begin{proof}
We  express $p$ and $q$ in the following way:
$$p=\cos\theta\cdot e_4+\sin\theta\cdot \iota_2(x), q=\cos\theta'\cdot e_4+\sin\theta'\cdot \iota_2(y)$$
where $e_4=(0,0,0,1)$ for some $\theta,\theta'\in \left[0,\frac{\pi}{2}\right]$ and $x,y\in\Sp^2$. Then,
$$\langle p,q \rangle=\cos\theta\cos\theta'+\langle x,y\rangle\sin\theta\sin\theta'.$$

\begin{enumerate}
    \item If $p,q\in \mathcal{C}_\alpha^\mathrm{top}(V_i)$ for some $i\in\{1,2,3,4\}$: Then we can assume $x,y\in V_i$ and $\theta,\theta'\in [0,\alpha]$. Hence,
    \begin{align*}
    \langle p,q \rangle&\geq\cos\theta\cos\theta'-\frac{1}{\sqrt{3}}\sin\theta\sin\theta'\quad\left(\because\langle x,y\rangle\geq-\frac{1}{\sqrt{3}}\text{ by Remark \ref{rmk:diamoneface}}\right)\\
    &\geq \cos^2\alpha-\frac{1}{\sqrt{3}}\sin^2\alpha=\left(1+\frac{1}{\sqrt{3}}\right)\cos^2\alpha-\frac{1}{\sqrt{3}}.
    \end{align*}
    where the second inequality holds since $\cos\theta\cos\theta'+\langle x,y\rangle\sin\theta\sin\theta'$ is decreasing in both of $\theta$ and $\theta'$.

    \item If $p\in \mathcal{C}_\alpha^\mathrm{top}(V_i)$ and $q\in \mathcal{C}_\alpha^\mathrm{bot}(V_{j,i})$ for some $i\neq j\in\{1,2,3,4\}$: Then we can assume $x\in V_i$, $y\in V_{j,i}$, $\theta\in [0,\alpha]$, and $\theta'\in \left[\alpha,\frac{\pi}{2}\right]$. Now, consider two cases separately.

If $\langle x,y \rangle\leq 0$, then $\cos\theta\cos\theta'+\langle x,y\rangle\sin\theta\sin\theta'$ is decreasing with respect to both of $\theta$ and $\theta'$. Hence,
$$\langle p,q \rangle \leq \cos 0\cos\alpha+\langle x,y \rangle\sin 0\sin\alpha=\cos\alpha.$$

If $\langle x,y \rangle\geq 0$, observe that
$$\langle p,q \rangle=(1-\langle x,y \rangle)\cos\theta\cos\theta'+\langle x,y \rangle\cos(\theta'-\theta).$$
If we view $\theta'$ as a variable on $\left[\alpha,\frac{\pi}{2}\right]$,

\begin{align*}
    &\frac{\partial}{\partial\theta'}\big((1-\langle x,y \rangle)\cos\theta\cos\theta'+\langle x,y \rangle\cos(\theta'-\theta)\big)\\
    &=-(1-\langle x,y \rangle)\cos\theta\sin\theta'-\langle x,y \rangle\sin(\theta'-\theta)\leq 0.
\end{align*}

Hence, $\langle p,q \rangle$ is maximized when $\theta'=\alpha$. So, $\langle p,q \rangle\leq\cos\theta\cos\alpha+\langle x,y \rangle\sin\theta\sin\alpha$. Now, if we view $\theta$ as a variable and take a derivative,
$$\frac{\partial}{\partial\theta}\big(\cos\theta\cos\alpha+\langle x,y \rangle\sin\theta\sin\alpha\big)=-\sin\theta\cos\alpha+\langle x,y \rangle\cos\theta\sin\alpha.$$

One can easily check that
$$-\sin\theta\cos\alpha+\langle x,y \rangle\cos\theta\sin\alpha=\begin{cases}\geq 0&\text{if }\theta'\in[0,\theta_0]\\ \leq 0&\text{if }\theta'\in[\theta_0,\alpha]\end{cases}$$
where $\theta_0$ is the unique critical point satisfying $\tan\theta_0=\langle x,y \rangle\tan\alpha$. Hence, $\cos\theta\cos\alpha+\langle x,y \rangle\sin\theta\sin\alpha$ is maximized when $\theta=\theta_0$. Hence,

$$\langle p,q \rangle\leq\cos\theta\cos\alpha+\langle x,y \rangle\sin\theta\sin\alpha\leq \sqrt{\cos^2\alpha+\langle x,y \rangle^2\sin^2\alpha}.$$

Note that $\langle x,y \rangle\leq \frac{1}{\sqrt{3}}$ since $x\in V_i$ and $y\in V_{ji}$ (this value $\frac{1}{\sqrt{3}}$ can be achieved when $x$ is the midpoint of $-u_k,-u_l$ for $k\neq l\in\{1,2,3,4\}\backslash\{i,j\}$ and $y=u_j$). Hence, one can conclude,
$$\langle p,q \rangle\leq\sqrt{\cos^2\alpha+\frac{1}{3}\sin^2\alpha}=\sqrt{\frac{2}{3}\cos^2\alpha+\frac{1}{3}}.$$
Since obviously $\cos\alpha\leq\sqrt{\cos^2\alpha+\frac{1}{3}\sin^2\alpha}=\sqrt{\frac{2}{3}\cos^2\alpha+\frac{1}{3}}$, this completes the proof of this case.

    \item If $p\in \mathcal{C}_\alpha^\mathrm{bot}(V_{i,k})$ and $q\in \mathcal{C}_\alpha^\mathrm{bot}(V_{j,i})$ for some $i\neq j\in\{1,2,3,4\}$ and $k\in\{1,2,3,4\}\backslash\{i,j\}$: Then one can assume $x\in V_{i,k}$, $y\in V_{j,i}$, and $\theta,\theta'\in \left[\alpha,\frac{\pi}{2}\right]$. Now, consider two cases separately.

If $\langle x,y \rangle\leq 0$, then $\cos\theta\cos\theta'+\langle x,y\rangle\sin\theta\sin\theta'$ is decreasing with respect to both of $\theta$ and $\theta'$. Hence,
$$\langle p,q \rangle \leq \cos^2\alpha+\langle x,y \rangle\sin^2\alpha\leq\cos^2\alpha.$$

If $\langle x,y \rangle\geq 0$, without loss of generality, one can assume $\theta\geq\theta'$. Also, observe that
$$\langle p,q \rangle=(1-\langle x,y \rangle)\cos\theta\cos\theta'+\langle x,y \rangle\cos(\theta-\theta').$$
If we view $\theta$ as a variable on $\left[\theta',\frac{\pi}{2}\right]$,
\begin{align*}
    &\frac{\partial}{\partial\theta}\big((1-\langle x,y \rangle)\cos\theta\cos\theta'+\langle x,y \rangle\cos(\theta-\theta')\big)\\
    &=-(1-\langle x,y \rangle)\sin\theta\cos\theta'-\langle x,y \rangle\sin(\theta-\theta')\leq 0.
\end{align*}

Hence, $\langle p,q \rangle$ is maximized when $\theta=\theta'$. So, $\langle p,q \rangle\leq\cos^2\theta'+\langle x,y \rangle\sin^2\theta'$. Now, if we view $\theta'$ as a variable and take a derivative,
$$\frac{\partial}{\partial\theta'}\big(\cos^2\theta'+\langle x,y \rangle\sin^2\theta'\big)=-2(1-\langle x,y \rangle) \cos\theta'\sin\theta'\leq 0.$$
Therefore, $\cos^2\theta'+\langle x,y \rangle\sin^2\theta'$ is maximized when $\theta'=\alpha$. Hence, $\langle p,q \rangle\leq \cos^2\alpha+\langle x,y \rangle\sin^2\alpha$. Note that $\langle x,y \rangle \leq \frac{1}{\sqrt{3}}$ as in the proof of the previous case. Hence, finally we get $\langle p,q \rangle\leq \cos^2\alpha+\frac{1}{\sqrt{3}}\sin^2\alpha=\left(1-\frac{1}{\sqrt{3}}\right)\cos^2\alpha+\frac{1}{\sqrt{3}}$. Since $\cos^2\alpha$ is obviously smaller than $\cos^2\alpha+\frac{1}{\sqrt{3}}\sin^2\alpha=\left(1-\frac{1}{\sqrt{3}}\right)\cos^2\alpha+\frac{1}{\sqrt{3}}$, this completes the proof of this case.

    \item If $p\in \mathcal{C}_\alpha^\mathrm{bot}(V_{i,j})$ and $q\in \mathcal{C}_\alpha^\mathrm{bot}(V_{j,i})$ for some $i\neq j\in\{1,2,3,4\}$: Then one can assume $x\in V_{i,j}$, $y\in V_{j,i}$, and $\theta,\theta'\in \left[\alpha,\frac{\pi}{2}\right]$. Since $\langle x,y \rangle\leq 0$ always in this case, $\cos\theta\cos\theta'+\langle x,y\rangle\sin\theta\sin\theta'$ is decreasing with respect to both of $\theta$ and $\theta'$. Hence, $\langle p,q \rangle$ is maximized when $\theta=\theta'=\alpha$. Therefore,
$$\langle p,q \rangle \leq \cos^2\alpha+\langle x,y \rangle\sin^2\alpha\leq\cos^2\alpha$$
as we wanted.
\end{enumerate}
\end{proof}

Finally, we are ready to construct the following map:
\begin{align*}
    \widetilde{\phi}_{3,2}^\alpha: \mathbf{H}_{>0}(\Sp^3)&\longrightarrow \Sp^2\\
    p&\longmapsto\begin{cases}u_i&\text{if }p\in \mathcal{C}_\alpha^\mathrm{top}(V_i)\text{ for some }i\in\{1,2,3,4\}\\ u_{i,j}&\text{if }p\in \mathcal{C}_\alpha^\mathrm{bot}(V_{i,j})\text{ for some }i\neq j\in\{1,2,3,4\}\end{cases}
\end{align*}

\begin{proposition}\label{prop:distofhalfphi32}
For $\alpha\in\left[0,\frac{\pi}{2}\right]$ such that $\cos^2\alpha\in\left[\frac{\sqrt{3}-1}{3+\sqrt{3}},\frac{7}{9}\right]$,
$$\dis(\widetilde{\phi}_{3,2}^\alpha)\leq\zeta_2.$$
\end{proposition}
\begin{proof}
We need to check
$$\vert d_{\Sp^3}(p,q)-d_{\Sp^2}(\widetilde{\phi}_{3,2}^\alpha(p),\widetilde{\phi}_{3,2}^\alpha(q))\vert\leq\zeta_2$$
for any $p,q\in \mathbf{H}_{> 0}(\Sp^3)$. We carry out a  case-by-case analysis.

\begin{enumerate}
    \item If $p,q\in \mathcal{C}(V_i)$ for some $i\in\{1,2,3,4\}$: Without loss of generality, one can assume $i=1$. Then, $d_{\Sp^2}(\widetilde{\phi}_{3,2}^\alpha(p),\widetilde{\phi}_{3,2}^\alpha(q))\leq\diam(\{u_1,u_{1,2},u_{1,3},u_{1,4}\})=\arccos\frac{5}{6}<\zeta_2$ by the first item of Lemma \ref{lemma:distuijs}. Therefore,
    $$d_{\Sp^2}(\widetilde{\phi}_{3,2}^\alpha(p),\widetilde{\phi}_{3,2}^\alpha(q))-d_{\Sp^3}(p,q)\leq\arccos\frac{5}{6}<\zeta_2.$$
    So, it is enough to prove $d_{\Sp^3}(p,q)-d_{\Sp^2}(\widetilde{\phi}_{3,2}^\alpha(p),\widetilde{\phi}_{3,2}^\alpha(q))\leq\zeta_2$. But for this direction, we need more subtle case-by-case analysis.
    \begin{enumerate}
        \item If $p,q\in \mathcal{C}_\alpha^\mathrm{top}(V_1)$: Then $\widetilde{\phi}_{3,2}^\alpha(p)=\widetilde{\phi}_{3,2}^\alpha(q)=u_1$. Also, by the item (1) of Lemma \ref{lemma:S3S2techineq} and the choice of $\alpha$,
        $$d_{\Sp^3}(p,q)\leq\arccos\left(\left(1+\frac{1}{\sqrt{3}}\right)\cos^2\alpha-\frac{1}{\sqrt{3}}\right)\leq\zeta_2.$$
        Hence,
        $$d_{\Sp^3}(p,q)-d_{\Sp^2}(\widetilde{\phi}_{3,2}^\alpha(p),\widetilde{\phi}_{3,2}^\alpha(q))=d_{\Sp^3}(p,q)\leq\zeta_2$$
        as we wanted.
        
        \item If $p\in \mathcal{C}_\alpha^\mathrm{top}(V_1)$ and $q\in \mathcal{C}_\alpha^\mathrm{bot}(V_1)$: In this case, $\widetilde{\phi}_{3,2}^\alpha(p)=u_1$ and $\widetilde{\phi}_{3,2}^\alpha(q)=u_{1,j}$ for some $j\in\{2,3,4\}$. Therefore,
        $$d_{\Sp^3}(p,q)-d_{\Sp^2}(\widetilde{\phi}_{3,2}^\alpha(p),\widetilde{\phi}_{3,2}^\alpha(q))\leq \arccos\left(-\frac{1}{\sqrt{3}}\right)-\arccos\left(\frac{2\sqrt{2}}{3}\right)<\zeta_2.$$
        
        \item If $p,q\in \mathcal{C}_\alpha^\mathrm{bot}(V_1)$:
        \begin{enumerate}
            \item If $p,q\in \mathcal{C}_\alpha^\mathrm{bot}(V_{1,j})$ for some $j\in\{2,3,4\}$: Then $\widetilde{\phi}_{3,2}^\alpha(p)=\widetilde{\phi}_{3,2}^\alpha(q)=u_{1,j}$. Also, it is easy to check the diameter of $\mathcal{C}_\alpha^\mathrm{bot}(V_{1,j})$ is $\frac{\pi}{2}$. Hence,
        $$d_{\Sp^3}(p,q)-d_{\Sp^2}(\widetilde{\phi}_{3,2}^\alpha(p),\widetilde{\phi}_{3,2}^\alpha(q))=d_{\Sp^3}(p,q)\leq\frac{\pi}{2}<\zeta_2.$$
            
            \item If $p\in \mathcal{C}_\alpha^\mathrm{bot}(V_{1,k})$ and $p\in \mathcal{C}_\alpha^\mathrm{bot}(V_{1,l})$ for some $k\neq l\in\{2,3,4\}$: Then, $$d_{\Sp^2}(\widetilde{\phi}_{3,2}^\alpha(p),\widetilde{\phi}_{3,2}^\alpha(q))=d_{\Sp^2}(u_{1,k},u_{1,l})=\arccos\left(\frac{5}{6}\right)$$
            by the item (1) of Lemma \ref{lemma:distuijs}. Therefore,
            $$d_{\Sp^3}(p,q)-d_{\Sp^2}(\widetilde{\phi}_{3,2}^\alpha(p),\widetilde{\phi}_{3,2}^\alpha(q))\leq \arccos\left(-\frac{1}{\sqrt{3}}\right)-\arccos\left(\frac{5}{6}\right)<\zeta_2.$$
        \end{enumerate}
    \end{enumerate}
    
    \item If $p\in \mathcal{C}(V_i)$ and $q\in \mathcal{C}(V_j)$ for some $i\neq j\in\{1,2,3,4\}$: Without loss of generality, one can assume $i=1$ and $j=2$. Then, by Lemma \ref{lemma:distuijs}, $d_{\Sp^2}(\widetilde{\phi}_{3,2}^\alpha(p),\widetilde{\phi}_{3,2}^\alpha(q)) \geq \arccos\left(\frac{5}{54}\right)>\arccos\left(\frac{1}{3}\right)$. Therefore,
    $$d_{\Sp^3}(p,q)-d_{\Sp^2}(\widetilde{\phi}_{3,2}^\alpha(p),\widetilde{\phi}_{3,2}^\alpha(q))<\pi-\arccos\left(\frac{1}{3}\right)=\zeta_2.$$
    So, it is enough to prove $d_{\Sp^2}(\widetilde{\phi}_{3,2}^\alpha(p),\widetilde{\phi}_{3,2}^\alpha(q))-d_{\Sp^3}(p,q)\leq\zeta_2$. But for this direction, we need more subtle case-by-case analysis.
    
    \begin{enumerate}
        \item If $p\in \mathcal{C}_\alpha^\mathrm{top}(V_1)$ and $q\in \mathcal{C}_\alpha^\mathrm{top}(V_2)$: Then, $\widetilde{\phi}_{3,2}^\alpha(p)=u_1$ and $\widetilde{\phi}_{3,2}^\alpha(q)=u_2$ so that $d_{\Sp^2}(\widetilde{\phi}_{3,2}^\alpha(p),\widetilde{\phi}_{3,2}^\alpha(q))=d_{\Sp^2}(u_1,u_2)=\zeta_2$. So, obviously $$d_{\Sp^2}(\widetilde{\phi}_{3,2}^\alpha(p),\widetilde{\phi}_{3,2}^\alpha(q))-d_{\Sp^3}(p,q)\leq\zeta_2.$$
        
        \item If $p\in \mathcal{C}_\alpha^\mathrm{top}(V_1)$ and $q\in \mathcal{C}_\alpha^\mathrm{bot}(V_2)$:
        \begin{enumerate}
            \item If $q\in \mathcal{C}_\alpha^\mathrm{bot}(V_{2,j})$ for some $j\in\{3,4\}$: Then, by the item (6) of Lemma \ref{lemma:distuijs}, $d_{\Sp^2}(\widetilde{\phi}_{3,2}^\alpha(p),\widetilde{\phi}_{3,2}^\alpha(q))=d_{\Sp^3}(u_1,u_{2,j})=\arccos\left(-\frac{\sqrt{2}}{9}\right)$. Hence,
            $$d_{\Sp^2}(\widetilde{\phi}_{3,2}^\alpha(p),\widetilde{\phi}_{3,2}^\alpha(q))-d_{\Sp^3}(p,q)\leq\arccos\left(-\frac{\sqrt{2}}{9}\right)<\zeta_2.$$
            
            \item If $q\in \mathcal{C}_\alpha^\mathrm{bot}(V_{2,1})$: Then, $d_{\Sp^2}(\widetilde{\phi}_{3,2}^\alpha(p),\widetilde{\phi}_{3,2}^\alpha(q))=d_{\Sp^2}(u_1,u_{2,1})=\arccos\left(-\frac{4\sqrt{2}}{9}\right)$ by the item (7) of Lemma \ref{lemma:distuijs}. Moreover, by the item (2) Lemma \ref{lemma:S3S2techineq} and the choice of $\alpha$,
            $$d_{\Sp^3}(p,q)\geq\arccos\left(\sqrt{\frac{2}{3}\cos^2\alpha+\frac{1}{3}}\right)>\arccos\left(\frac{2\sqrt{2}}{3}\right).$$
            It implies,
            $$d_{\Sp^2}(\widetilde{\phi}_{3,2}^\alpha(p),\widetilde{\phi}_{3,2}^\alpha(q))-d_{\Sp^3}(p,q)<\arccos\left(-\frac{4\sqrt{2}}{9}\right)-\arccos\left(\frac{2\sqrt{2}}{3}\right)=\zeta_2.$$
        \end{enumerate}
        
        \item If $p\in \mathcal{C}_\alpha^\mathrm{bot}(V_1)$ and $q\in \mathcal{C}_\alpha^\mathrm{bot}(V_2)$: Considering symmetry, there are basically four subcases.
        \begin{enumerate}
            \item If $p\in \mathcal{C}_\alpha^\mathrm{bot}(V_{1,3})$ and $q\in \mathcal{C}_\alpha^\mathrm{bot}(V_{2,3})$: Then, by the item (2) of Lemma \ref{lemma:distuijs}, $d_{\Sp^2}(\widetilde{\phi}_{3,2}^\alpha(p),\widetilde{\phi}_{3,2}^\alpha(q))=d_{\Sp^2}(u_{1,3},u_{2,3})=\arccos\left(\frac{5}{54}\right)$. Hence,
            $$d_{\Sp^2}(\widetilde{\phi}_{3,2}^\alpha(p),\widetilde{\phi}_{3,2}^\alpha(q))-d_{\Sp^3}(p,q)\leq\arccos\left(\frac{5}{54}\right)<\zeta_2.$$
            
            \item If $p\in \mathcal{C}_\alpha^\mathrm{bot}(V_{1,3})$ and $q\in \mathcal{C}_\alpha^\mathrm{bot}(V_{2,4})$: Then, by the item (3) of Lemma \ref{lemma:distuijs}, $d_{\Sp^2}(\widetilde{\phi}_{3,2}^\alpha(p),\widetilde{\phi}_{3,2}^\alpha(q))=d_{\Sp^2}(u_{1,3},u_{2,4})=\arccos\left(-\frac{2}{27}\right)$. Hence,
            $$d_{\Sp^2}(\widetilde{\phi}_{3,2}^\alpha(p),\widetilde{\phi}_{3,2}^\alpha(q))-d_{\Sp^3}(p,q)\leq\arccos\left(-\frac{2}{27}\right)<\zeta_2.$$
            
            \item If $p\in \mathcal{C}_\alpha^\mathrm{bot}(V_{1,3})$ and $q\in \mathcal{C}_\alpha^\mathrm{bot}(V_{2,1})$: Then, by the item (4) of Lemma \ref{lemma:distuijs}, $d_{\Sp^2}(\widetilde{\phi}_{3,2}^\alpha(p),\widetilde{\phi}_{3,2}^\alpha(q))=d_{\Sp^2}(u_{1,3},u_{2,1})=\arccos\left(-\frac{25}{54}\right)$. Moreover, by the item (3) of Lemma \ref{lemma:S3S2techineq} and the choice of $\alpha$,
            $$d_{\Sp^3}(p,q)\geq\arccos\left(\left(1-\frac{1}{\sqrt{3}}\right)\cos^2\alpha+\frac{1}{\sqrt{3}}\right)>\arccos\left(-\frac{25}{54}\right)-\zeta_2.$$
            Hence,
            $$d_{\Sp^2}(\widetilde{\phi}_{3,2}^\alpha(p),\widetilde{\phi}_{3,2}^\alpha(q))-d_{\Sp^3}(p,q)<\zeta_2.$$
            
            \item If $p\in \mathcal{C}_\alpha^\mathrm{bot}(V_{1,2})$ and $q\in \mathcal{C}_\alpha^\mathrm{bot}(V_{2,1})$: Then, by the item (5) of Lemma \ref{lemma:distuijs}, $d_{\Sp^2}(\widetilde{\phi}_{3,2}^\alpha(p),\widetilde{\phi}_{3,2}^\alpha(q))=d_{\Sp^2}(u_{1,2},u_{2,1})=\arccos\left(-\frac{23}{27}\right)$. Moreover, by the item (4) of Lemma \ref{lemma:S3S2techineq} and the choice of $\alpha$,
            $$d_{\Sp^3}(p,q)\geq\arccos(\cos^2\alpha)\geq\arccos\left(\frac{7}{9}\right).$$
            Hence,
            $$d_{\Sp^2}(\widetilde{\phi}_{3,2}^\alpha(p),\widetilde{\phi}_{3,2}^\alpha(q))-d_{\Sp^3}(p,q)\leq \arccos\left(-\frac{23}{27}\right)-\arccos\left(\frac{7}{9}\right)=\zeta_2.$$
        \end{enumerate}
    \end{enumerate}
\end{enumerate}
This concludes the proof.
\end{proof}

\begin{lemma}\label{lemma:pphipuppbdd}
For any $p\in \mathbf{H}_{>0}(\Sp^3)$, $d_{\Sp^3}(p,\widetilde{\phi}_{3,2}^\alpha(p))\leq\frac{\pi}{2}$.
\end{lemma}
\begin{proof}
Without loss of generality, one can assume $p\in\mathcal{C}(V_1)$. Then, one can express $p$ in the following way: $p=\cos\theta\cdot e_4+\sin\theta\cdot \iota_2(x)$ where $e_4=(0,0,0,1)$ for some $\theta\in \left[0,\frac{\pi}{2}\right]$ and $x\in V_1$. Moreover, since $\widetilde{\phi}_{3,2}^\alpha(p)\in\{u_1,u_{1,2},u_{1,3},u_{1,4}\}$, we have
$$\langle p,\widetilde{\phi}_{3,2}^\alpha(p) \rangle=\langle x,\widetilde{\phi}_{3,2}^\alpha(p) \rangle\cdot\sin\theta.$$
Also, it is easy to check $\langle x,\widetilde{\phi}_{3,2}^\alpha(p) \rangle\geq 0$ (more precisely, $\langle u_1,x \rangle\geq\frac{1}{3}$ and $\langle u_{1,j},x \rangle\geq\frac{\sqrt{2}}{9}$ for any $x\in N_1$, $j\neq 1$). This implies $\langle p,\widetilde{\phi}_{3,2}^\alpha(p) \rangle\geq 0$ hence we have the required inequality.
\end{proof}

We are now ready to prove Proposition \ref{prop:ub23}.

\begin{proof}[Proof of Proposition \ref{prop:ub23}]
Note that it is enough to find a surjective map $\phi_{3,2}:\Sp^3\twoheadrightarrow\Sp^2$ such that $\dis(\phi_{3,2})\leq\zeta_2$ since this map gives rise to the correspondence $R_{3,2}:=\mathrm{graph}(\phi_{3,2})$ with $\dis(R_{3,2}) = \dis(\phi_{3,2})\leq \zeta_2$.

Let
\begin{align*}
    \widehat{\phi}_{3,2}^\alpha: \seta(\Sp^3)&\longrightarrow\Sp^2\\
    p&\longmapsto\begin{cases}\widetilde{\phi}_{3,2}^\alpha(p)&\text{if }p\in \mathbf{H}_{>0}(\Sp^3)\\ p&\text{if }p\in \iota_2(\seta(\Sp^2)).\end{cases}
\end{align*}
We claim that $\dis(\widehat{\phi}_{3,2}^\alpha)=\dis(\widetilde{\phi}_{3,2}^\alpha)$. To check this, it is enough to show that
$$\vert d_{\Sp^3}(p,q)-d_{\Sp^2}(\widehat{\phi}_{3,2}^\alpha(p),\widehat{\phi}_{3,2}^\alpha(q))\vert\leq\zeta_2$$
for any $p\in \mathbf{H}_{>0}(\Sp^3)$ and $q\in\iota_2(\seta(\Sp^2))$. But, this is true since
$$\vert d_{\Sp^3}(p,q)-d_{\Sp^2}(\widehat{\phi}_{3,2}^\alpha(p),\widehat{\phi}_{3,2}^\alpha(q))\vert=\vert d_{\Sp^3}(p,q)-d_{\Sp^2}(\widehat{\phi}_{3,2}^\alpha(p),q)\vert\leq d_{\Sp^3}(p,\widehat{\phi}_{3,2}^\alpha(p)),$$
and $d_{\Sp^3}(p,\widehat{\phi}_{3,2}^\alpha(p))=d_{\Sp^3}(p,\widetilde{\phi}_{3,2}^\alpha(p))\leq\frac{\pi}{2}<\zeta_2$ for any $p\in \mathbf{H}_{>0}(\Sp^3)$ by Lemma \ref{lemma:pphipuppbdd} Hence, $\dis(\widehat{\phi}_{3,2}^\alpha)=\dis(\widetilde{\phi}_{3,2}^\alpha)$ as we wanted. Finally, apply Lemma \ref{lemma:distortion} to construct a surjective map $\phi_{3,2}:\Sp^3\longtwoheadrightarrow\Sp^2$. Then,
$$\dis(\phi_{3,2})=\dis(\widehat{\phi}_{3,2}^\alpha)=\dis(\widetilde{\phi}_{3,2}^\alpha)\leq\zeta_2$$ by Proposition \ref{prop:distofhalfphi32}, as we wanted.
\end{proof}

\section{The Gromov-Hausdorff distance between spheres with  Euclidean metric} \label{sec:dgh-eucl}

For any non-empty subset $X\subseteq \Sp^n$, let $X_\mathrm{E}$ denote the metric space with the inherited Euclidean metric.
 In particular, $\Sp^n_{\mathrm{E}}$ will denote the unit sphere with the Euclidean metric $d_{\mathrm{E}}$ inherited from $\R^{n+1}$. A natural  question is, what is the value of $$\mathfrak{g}^\mathrm{E}_{m,n}:=\dgh(\Sp^m_{\mathrm{E}},\Sp^n_{\mathrm{E}})$$ for $0\leq m < n \leq\infty$? We found that, interestingly, these values do always not directly follow from those of $\mathfrak{g}_{m,n}$.

Any  correspondence $R$ between $\Sp^m$ and $\Sp^n$ can of course be regarded as a correspondence between $\Sp^m_{\mathrm{E}}$ and $\Sp^n_{\mathrm{E}}$. Throughout this section, let $\dis(R)$ denote the distortion with respect to the geodesic metric (as usual), and let $\dis_{\mathrm{E}}(R)$ denote the distortion with respect to the Euclidean metric.

The following  are direct extensions of parallel results for spheres with geodesic distance: 

\begin{remark}\label{remark:ubEuclid}
As in Remark \ref{remark:ub}, for all $0\leq m\leq n\leq\infty$,
$$\dgh(\Sp^m_{\mathrm{E}},\Sp^n_{\mathrm{E}})\leq 1.$$
\end{remark}

\begin{lemma}\label{lemma:lsEuclid}
For any integer $m\geq 1$ and any finite metric space $P$ with cardinality at most $m+1$ we have $\dgh(\Sp^m_{\mathrm{E}},P)\geq 1$.
\end{lemma}
\begin{proof}
Fix an arbitrary correspondence $R$ between $\Sp^m_{\mathrm{E}}$ and $P$. Then, one can prove that $\dis_{\mathrm{E}}(R)\geq 2$ as in the proof of Lemma \ref{lemma:ls} (via the aid of Lyusternik-Schnirelmann Theorem). Since $R$ is arbitrary, one can conclude $\dgh(\Sp^m_{\mathrm{E}},P)\geq 1$.
\end{proof}

\begin{corollary}\label{coro:sinfty-finiteEuclid}
Let $R$ be any correspondence between a finite metric space $P$ and $\Sp^\infty_{\mathrm{E}}$. Then, $\dis_{\mathrm{E}}(R)\geq 2$. In particular, $\dgh(P,\Sp^\infty_{\mathrm{E}})\geq 1$.
\end{corollary}
\begin{proof}
See the proof of Corollary \ref{coro:sinfty-finite}.
\end{proof}

\begin{proposition}\label{prop:tb-sinftyEuclid}
Let $X$ be any totally bounded metric space. Then $\dgh(X,\Sp^\infty_{\mathrm{E}})\geq 1$.
\end{proposition}
\begin{proof}
Follow the idea of the proof of Proposition \ref{prop:tb-sinfty}.
\end{proof}

\begin{proposition}
For any $n\geq 1$, $\dgh(\Sp^0_{\mathrm{E}},\Sp^n_{\mathrm{E}})=1$.
\end{proposition}
\begin{proof}
Apply Remark \ref{remark:ubEuclid} and Lemma \ref{lemma:lsEuclid}.
\end{proof}

\begin{proposition}
For any integer $m\geq 0$, $\dgh(\Sp^m_{\mathrm{E}},\Sp^\infty_{\mathrm{E}})=1$.
\end{proposition}
\begin{proof}
Apply Remark \ref{remark:ubEuclid} and Proposition \ref{prop:tb-sinftyEuclid}.
\end{proof}

The following lemma permits bounding $\dis_\mathrm{E}(R)$ via $\dis(R)$:

\begin{lemma}\label{lemma:disEucGeo}
Let $0\leq m < n \leq\infty$, and let $R$ be an arbitrary non-empty relation between $\Sp^m_{\mathrm{E}}$ and $\Sp^n_{\mathrm{E}}$. Then,
$$\dis_{\mathrm{E}}(R)\leq 2\sin\left(\frac{\dis(R)}{2}\right).$$
\end{lemma}
\begin{proof}
First of all, note that $\dis(R):=\sup_{(x,y),(x',y')\in R}\vert d_{\Sp^m}(x,x')-d_{\Sp^n}(y,y')\vert\leq\pi$, since both  $\diam(\Sp^m)$ and $\diam(\Sp^n)$ are at most $\pi$. Fix arbitrary $(x,y),(x',y')\in R$. Then, 
\begin{align*}
    d_{\mathrm{E}}(x,x')&=2\sin\left(\frac{d_{\Sp^m}(x,x')}{2}\right)\\
    &=2\sin\left(\frac{d_{\Sp^m}(x,x')}{2}-\frac{d_{\Sp^n}(y,y')}{2}+\frac{d_{\Sp^n}(y,y')}{2}\right)\\
    &=2\sin\left(\frac{d_{\Sp^m}(x,x')}{2}-\frac{d_{\Sp^n}(y,y')}{2}\right)\cos\left(\frac{d_{\Sp^n}(y,y')}{2}\right)\\
    &\qquad+2\cos\left(\frac{d_{\Sp^m}(x,x')}{2}-\frac{d_{\Sp^n}(y,y')}{2}\right)\sin\left(\frac{d_{\Sp^n}(y,y')}{2}\right)\\
    &\leq 2\sin\left(\frac{\vert d_{\Sp^m}(x,x')-d_{\Sp^n}(y,y')\vert}{2}\right)+2\sin\left(\frac{d_{\Sp^n}(y,y')}{2}\right)\\
    &=2\sin\left(\frac{\vert d_{\Sp^m}(x,x')-d_{\Sp^n}(y,y')\vert}{2}\right)+d_{\mathrm{E}}(y,y').
\end{align*}
where the inequality follows since $\cos\left(\frac{d_{\Sp^m}(x,x')}{2}-\frac{d_{\Sp^n}(y,y')}{2}\right)\in[0,1]$.

Hence,
$$d_{\mathrm{E}}(x,x')-d_{\mathrm{E}}(y,y')\leq 2\sin\left(\frac{\vert d_{\Sp^m}(x,x')-d_{\Sp^n}(y,y')\vert}{2}\right).$$
Similarly, one can also prove
$$d_{\mathrm{E}}(y,y')-d_{\mathrm{E}}(x,x')\leq 2\sin\left(\frac{\vert d_{\Sp^m}(x,x')-d_{\Sp^n}(y,y')\vert}{2}\right).$$
Therefore, we have
$$\vert d_{\mathrm{E}}(x,x')-d_{\mathrm{E}}(y,y')\vert\leq 2\sin\left(\frac{\vert d_{\Sp^m}(x,x')-d_{\Sp^n}(y,y')\vert}{2}\right).$$
Since $(x,y),(x',y') \in R$ were arbitrary, this leads to the required conclusion.
\end{proof}

\begin{corollary}\label{cor:EucGeo}
For any $0\leq m < n \leq \infty$:

\begin{enumerate}

    \item $\dgh(\Sp^m_{\mathrm{E}},\Sp^n_{\mathrm{E}})\leq\sin\big(\dgh(\Sp^m,\Sp^n)\big).$
    
    \item In more generality, for any $X\subseteq \Sp^m$ and  $Y\subseteq \Sp^n$, $\dgh(X_\mathrm{E},Y_\mathrm{E})\leq \sin\big(\dgh(X,Y)\big).$
\end{enumerate}
\end{corollary}

\begin{corollary}\label{cor:sn-sm-ubEuclid}
$\dgh(\Sp^m_{\mathrm{E}},\Sp^n_{\mathrm{E}})<1$, for all $0<m\neq n<\infty$.
\end{corollary}
\begin{proof}
Invoke Corollary \ref{cor:EucGeo} and Theorem \ref{thm:sn-sm-ub}.
\end{proof}

Given the above, and the fact that we have proved that  $\mathfrak{g}_{1,2} = \frac{\pi}{3}$ and $\mathfrak{g}_{2,3}=\frac{\zeta_2}{2}$,  one might expect that $\mathfrak{g}_{1,2}^\mathrm{E} = \dgh(\Sp^1_{\mathrm{E}},\Sp^{2}_{\mathrm{E}})=\sin\left(\frac{\pi}{3}\right) =\frac{\sqrt{3}}{2}$ and similarly that  $\mathfrak{g}_{2,3}^\mathrm{E} = \frac{\sqrt{2}}{\sqrt{3}}.$  However, rather surprisingly, we were able to construct a correspondence $R_\mathrm{E}$ between $\Sp^1_{\mathrm{E}}$ and $\mathbf{H}_{\geq 0}(\Sp^2_{\mathrm{E}})$ such that $\dis_{\mathrm{E}}(R_\mathrm{E})<\sqrt{3}$ (cf. Proposition \ref{prop:dgh-s12-s2pe} below and its proof in \S\ref{sec:proof-hepta}).  This correspondence then naturally induces a function $\phi_\mathrm{E}:\mathbf{A}(\Sp^2_\mathrm{E})\rightarrow \Sp^1_\mathrm{E}$ from the ``helmet" on $\Sp^2_\mathrm{E}$ into $\Sp^1_{\mathrm{E}}$ also with $\dis_\mathrm{E}(\phi_\mathrm{E})<\sqrt{3}.$ 

\begin{proposition}\label{prop:dgh-s12-s2pe}
$\dgh\big(\Sp_\mathrm{E}^1,\mathbf{H}_{\geq 0}(\Sp_\mathrm{E}^2)\big)<\frac{\sqrt{3}}{2}.$
\end{proposition}
This  proposition was motivated by Ilya Bogdanov's answer \cite{bogdanov} to a Math Overflow question regarding the Gromov-Hausdorff distance between $\Sp^1_\mathrm{E}$ and the unit disk in $\R^2$. 

We now discuss the possibility that the  correspondence  $R_\mathrm{E}$ described above permits proving that, in fact, $\dgh(\Sp^1_\mathrm{E},\Sp^2_\mathrm{E})<\frac{\sqrt{3}}{2}$ via extending $R_\mathrm{E}$ into a correspondence between $\Sp^2_\mathrm{E}$ and $\Sp^1_\mathrm{E}$ much in the same way that we did so in the case of spheres with their geodesic distance (cf. Lemma \ref{lemma:distortion}).

By the same method of proof as that of Corollary \ref{cor:strBU2} (giving the lower bound $\dis(g)\geq \zeta_n$ for any antipode preserving map $g:\Sp^n\rightarrow \Sp^{n-1}$) one obtains the following Euclidean analogue:

\begin{corollary}\label{cor:strBU2EC}
For each integer $n>0$, any function $g:\Sp^n_{\mathrm{E}}\rightarrow\Sp^{n-1}_{\mathrm{E}}$ which maps every pair of antipodal points on $\Sp^n_{\mathrm{E}}$ onto antipodal points on $\Sp^{n-1}_{\mathrm{E}}$ satisfies $\dis_{\mathrm{E}}(g)\geq\sqrt{2+\frac{2}{n}}$.
\end{corollary}

\begin{remark}[Extending Lemma \ref{lemma:distortion} to the case of spheres with Euclidean metric] Lemma \ref{lemma:distortion} was instrumental in our quest for lower bounds for the Gromov-Hausdorff distance between spheres with the geodesic distance. It is natural to attempt to obtain a suitable version of that result to the case of the Euclidean metric. However, there is a caveat. Indeed, one should \emph{not} expect to be able to prove a version in which $\dis_\mathrm{E}(\phi^*)$ is \emph{equal} to $\dis(\phi)$ where $\phi:\seta(\Sp^n_{\mathrm{E}})\rightarrow\Sp^m_{\mathrm{E}}$ and $\phi^*$ is its antipode preserving extension obtained via the ``helmet trick"  (as described in the statement of Lemma \ref{lemma:distortion}). If this was the case, then the antipode preserving extension $\phi^*_\mathrm{E}$ of the function $\phi_\mathrm{E}$ mentioned above would satisfy \begin{equation}
\label{eq:dis-phi-euclidean}    
\dis_\mathrm{E}(\phi^*_\mathrm{E})<\sqrt{3}.\end{equation}

However, note that Corollary \ref{cor:strBU2EC} implies that, in the case of spheres with Euclidean distance, any antipode preserving map $\psi:\Sp^{m+1}_\mathrm{E}\rightarrow \Sp^m_\mathrm{E}$  must satisfy $\dis_\mathrm{E}(\psi)\geq \sqrt{2+\frac{2}{m+1}}$. In particular, it must be that $\dis_\mathrm{E}(\psi)\geq \sqrt{3}$ for any antipode preserving map $\psi:\Sp^2_\mathrm{E}\rightarrow \Sp^1_\mathrm{E}$, and this would contradict equation (\ref{eq:dis-phi-euclidean}). 

Still, as we describe next, there is a suitable generalization of Lemma \ref{lemma:distortion} which yields nontrivial lower bounds (cf. Proposition \ref{prop:sn-sm-lb-DSEuclid}).
\end{remark}

\begin{lemma}\label{lemma:Eucliddistgap}
If $\vert a-b \vert=:\delta\in[0,2]$ for some $a,b\in [0,2]$, then
$$\left\vert\sqrt{4-a^2}-\sqrt{4-b^2}\right\vert\leq\sqrt{\delta(4-\delta)},$$
and the inequality is tight.
\end{lemma}
\begin{proof}
The claim is obvious if $\delta=0$. Henceforth, we will assume that $\delta>0$. Observe that,
\begin{align*}
    \left\vert\sqrt{4-a^2}-\sqrt{4-b^2}\right\vert&=\frac{\vert a^2-b^2\vert}{\sqrt{4-a^2}+\sqrt{4-b^2}}=\vert a-b\vert\cdot\frac{a+b}{\sqrt{4-a^2}+\sqrt{4-b^2}}\\
    &\leq\delta\cdot\frac{4-\delta}{\sqrt{4\delta-\delta^2}}=\sqrt{\delta(4-\delta)}
\end{align*}
as we wanted. Finally, the equality holds if $a=2,b=2-\delta$ or $a=2-\delta,b=2$.
\end{proof}

\begin{lemma}\label{lemma:distortionEuclid}
For any $m,n\geq 0$, let $\emptyset\neq C\subseteq \Sp^n_{\mathrm{E}}$ satisfy $C\cap (-C)=\emptyset$ and let the set $\phi:C\rightarrow \Sp^m_{\mathrm{E}}$ be any map. Then, the extension $\phi^*$ of $\phi$ to the set $C\cup(-C)$ defined  by
\begin{align*}
    {\phi^*}: C\cup(-C)&\longrightarrow \Sp^m\\
    C \ni x&\longmapsto \phi(x)\\
    -x&\longmapsto-\phi(x)
\end{align*}
is antipode preserving and satisfies $\dis_{\mathrm{E}}({\phi^*})\leq\sqrt{\dis_{\mathrm{E}}(\phi)(4-\dis_{\mathrm{E}}(\phi))}$.
\end{lemma}
\begin{proof}
$\phi^*$ is obviously antipode preserving by its definition. Now, fix arbitrary $x,x'\in C$. Then, 
\begin{align*}
    &\big\vert d_{\mathrm{E}}(x,-x')-d_{\mathrm{E}}({\phi^*}(x),{\phi^*}(-x')) \big\vert\\
    &=\left\vert \sqrt{4-(d_{\mathrm{E}}(x,x'))^2}-\sqrt{4-(d_{\mathrm{E}}(\phi(x),\phi(x')))^2} \right\vert\\
    &\leq\sqrt{\big\vert d_{\mathrm{E}}(x,x')-d_{\mathrm{E}}(\phi(x),\phi(x')) \big\vert(4-\big\vert d_{\mathrm{E}}(x,x')-d_{\mathrm{E}}(\phi(x),\phi(x')) \big\vert)}\\
    &\leq\sqrt{\dis_{\mathrm{E}}(\phi)(4-\dis_{\mathrm{E}}(\phi))}
\end{align*}
and,
$$\vert d_{\mathrm{E}}(-x,-x')-d_{\mathrm{E}}({\phi^*}(-x),{\phi^*}(-x')) \vert=\vert d_{\mathrm{E}}(x,x')-d_{\mathrm{E}}(\phi(x),\phi(x')) \vert\leq\dis_{\mathrm{E}}(\phi).$$
Hence,
$$\dis_{\mathrm{E}}(\phi^*)\leq\max\left\{\dis_{\mathrm{E}}(\phi),\sqrt{\dis_{\mathrm{E}}(\phi)(4-\dis_{\mathrm{E}}(\phi))}\right\}=\sqrt{\dis_{\mathrm{E}}(\phi)(4-\dis_{\mathrm{E}}(\phi))}$$
as we wanted to prove.
\end{proof}

\begin{corollary}\label{coro:preservedistortionEuclid}
For each $n\in\Z_{>0}$ and any map $\phi:\Sp^n_{\mathrm{E}}\rightarrow \Sp^{n-1}_{\mathrm{E}}$ there exists an antipode preserving map  $\phi^*:\Sp^n_{\mathrm{E}}\rightarrow \Sp^{n-1}_{\mathrm{E}}$ such that  $\dis_{\mathrm{E}}({\phi^*})\leq\sqrt{\dis_{\mathrm{E}}(\phi)(4-\dis_{\mathrm{E}}(\phi))}$.
\end{corollary}
\begin{proof}
Consider the restriction of $\phi$ to the ``helmet" $\seta(\Sp^n)$ (cf. \S\ref{sec:helmets}) and apply Lemma \ref{lemma:distortionEuclid}.
\end{proof}

\begin{proposition}\label{prop:sn-sm-lb-DSEuclid}
For all integers $0<m<n$,
$$\dgh(\Sp^m_{\mathrm{E}},\Sp^n_{\mathrm{E}})\geq \frac{1}{2}\left(2-\sqrt{2-\frac{2}{m+1}}\right)\geq \frac{1}{2}.$$ 
\end{proposition}
\begin{proof} 
Suppose to the contrary  that $\dgh(\Sp^m_{\mathrm{E}},\Sp^n_{\mathrm{E}})<\frac{1}{2}\left(2-\sqrt{2-\frac{2}{m+1}}\right)$. This implies that there exist a correspondence $\Gamma$ between $\Sp^m_{\mathrm{E}}$ and $\Sp^n_{\mathrm{E}}$ such that $\dis_{\mathrm{E}}(\Gamma)<\frac{1}{2}\left(2-\sqrt{2-\frac{2}{m+1}}\right)$. Moreover, since $n\geq m+1$, $\Sp^{m+1}_{\mathrm{E}}$ can be isometrically embedded in $\Sp^n_{\mathrm{E}}$, so we are able to construct a map $g:\Sp^{m+1}_{\mathrm{E}}\rightarrow \Sp^m_{\mathrm{E}}$ in the following way: for each $x\in \Sp^{m+1}_{\mathrm{E}}\subseteq \Sp^n_{\mathrm{E}}$, choose $g(x)\in \Sp^m_{\mathrm{E}}$ such that $(g(x),x)\in\Gamma$. Then, $\dis_{\mathrm{E}}(g)<\left(2-\sqrt{2-\frac{2}{m+1}}\right)$ as well. By applying Corollary \ref{coro:preservedistortionEuclid}, one can modify this $g$ into an antipode preserving map $\widehat{g}:\Sp^{m+1}_{\mathrm{E}}\rightarrow \Sp^m_{\mathrm{E}}$ with $$\dis_{\mathrm{E}}(\hat{g})\leq\sqrt{\dis_{\mathrm{E}}(g)(4-\dis_{\mathrm{E}}(g))}<\sqrt{2+\frac{2}{m+1}}$$ which contradicts Corollary \ref{cor:strBU2EC}.
\end{proof}

Note that in contrast to the case of geodesic distances (where the upper bound given by Proposition \ref{prop:ub} and the lower bound given by Theorem \ref{thm:sn-sm-lb-DS} agree when $m=1$ and $n=2$), Proposition \ref{prop:sn-sm-lb-DSEuclid} yields $\mathfrak{g}^\mathrm{E}_{1,2}\geq\frac{1}{2}$ which is strictly smaller than the upper bound $\frac{\sqrt{3}}{2}$ provided by Corollary \ref{cor:EucGeo} and Proposition \ref{prop:ub}.

\subsection{The proof of Proposition \ref{prop:dgh-s12-s2pe}}\label{sec:proof-hepta}
The proof will be based on a geometric construction which is illustrated in Figures \ref{fig:S1heptagon} and \ref{fig:S2sevenregion}.
\begin{proof}
To prove the claim, note that it is enough to construct a correspondence  $R_\mathrm{E}$ between $\Sp^1_{\mathrm{E}}$ and $\mathbf{H}_{\geq 0}(\Sp_\mathrm{E}^2)$ such that $\dis_{\mathrm{E}}(R_\mathrm{E})<\sqrt{3}$.

Firstly, let $u_1,\dots,u_7$ be the vertices of a regular heptagon inscribed in $\Sp^1$. Let $v_i:=-u_i$ for $i=1,\dots,7$. See Figure \ref{fig:S1heptagon} for a description.

Secondly, divide $\mathbf{H}_{\geq 0}(\Sp_E^2)$ into seven regions $A_1,\dots,A_7$ as in Figure \ref{fig:S2sevenregion}. The precise ``disjointification" (on the boundary) of the seven regions is not relevant to the analysis that follows, as it is easy to check.

Now, choose $a_i\in A_i$ for each $i=1,\dots,7$ in the following way, where $\alpha$ is some number which is very close to $\frac{\sqrt{3}}{2}$ but still strictly smaller than $\frac{\sqrt{3}}{2}$ (for example, choose $\alpha=0.866$):

\begin{align*}
    a_1&=\left(\sqrt{1-(\sqrt{1-\alpha^2}+2-\sqrt{3})^2},\sqrt{1-\alpha^2}+2-\sqrt{3},0\right)\approx(0.640511,0.767949,0),\\
    a_2&=\left(0,\sqrt{1-\alpha^2}+2-\sqrt{3},\sqrt{1-(\sqrt{1-\alpha^2}+2-\sqrt{3})^2}\right)\approx(0,0.767949,0.640511),\\
    a_3&=(0,\sqrt{1-\alpha^2},\alpha)\approx(0,0.5,0.866),\\ a_4&=(0,0,1),\\ a_5&=\left(0,-(\sqrt{1-\alpha^2}+\rho_6-\sqrt{3}),\sqrt{1-(\sqrt{1-\alpha^2}+\rho_6-\sqrt{3})^2}\right)\approx(0,-0.717805,0.696244),\\
    a_6&=\left(0,-\big(\sqrt{1-\alpha^2}+(\rho_6-\sqrt{3})+(\rho_5-\sqrt{3})\big),\sqrt{1-\big(\sqrt{1-\alpha^2}+(\rho_6-\sqrt{3})+(\rho_5-\sqrt{3})\big)^2}\right)\\
    &\approx(0,-0.787692,0.616069),\\
    a_7&=\left(\sqrt{1-\big(\sqrt{1-\alpha^2}+(\rho_6-\sqrt{3})+(\rho_5-\sqrt{3})\big)^2},-\big(\sqrt{1-\alpha^2}+(\rho_6-\sqrt{3})+(\rho_5-\sqrt{3})\big),0\right)\\
    &\approx(0.616069,-0.787692,0),
\end{align*}

\begin{figure}[ht]
    \centering
    \includegraphics[width=0.65\linewidth]{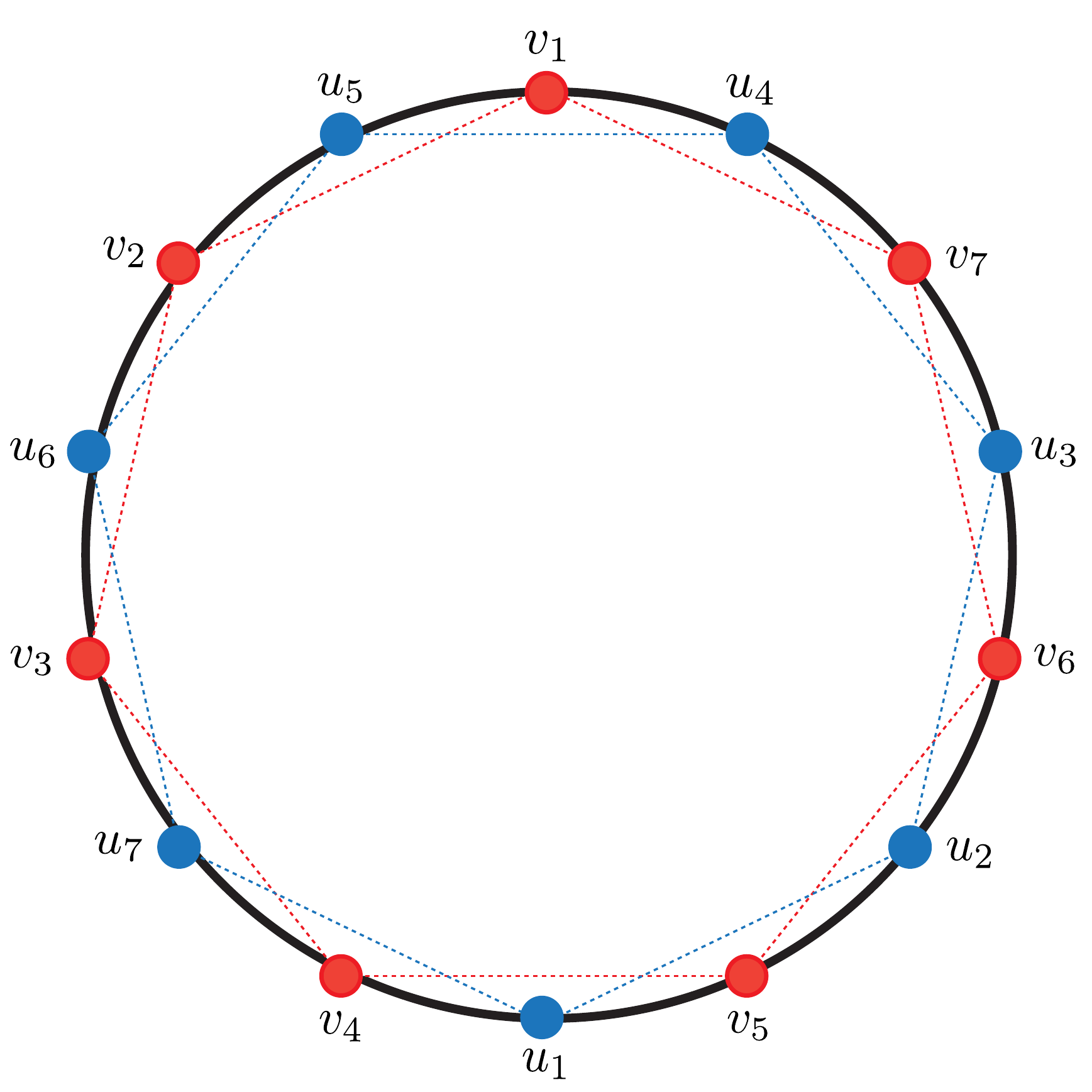}
    \caption{The points $v_1,\ldots, v_7$ and $u_1,\ldots,u_7$. These arise from two antipodal regular heptagons inscribed in of $\Sp^1$.}
    \label{fig:S1heptagon}
\end{figure}

\begin{figure}[ht]
    \centering
        \includegraphics[width=0.65\linewidth]{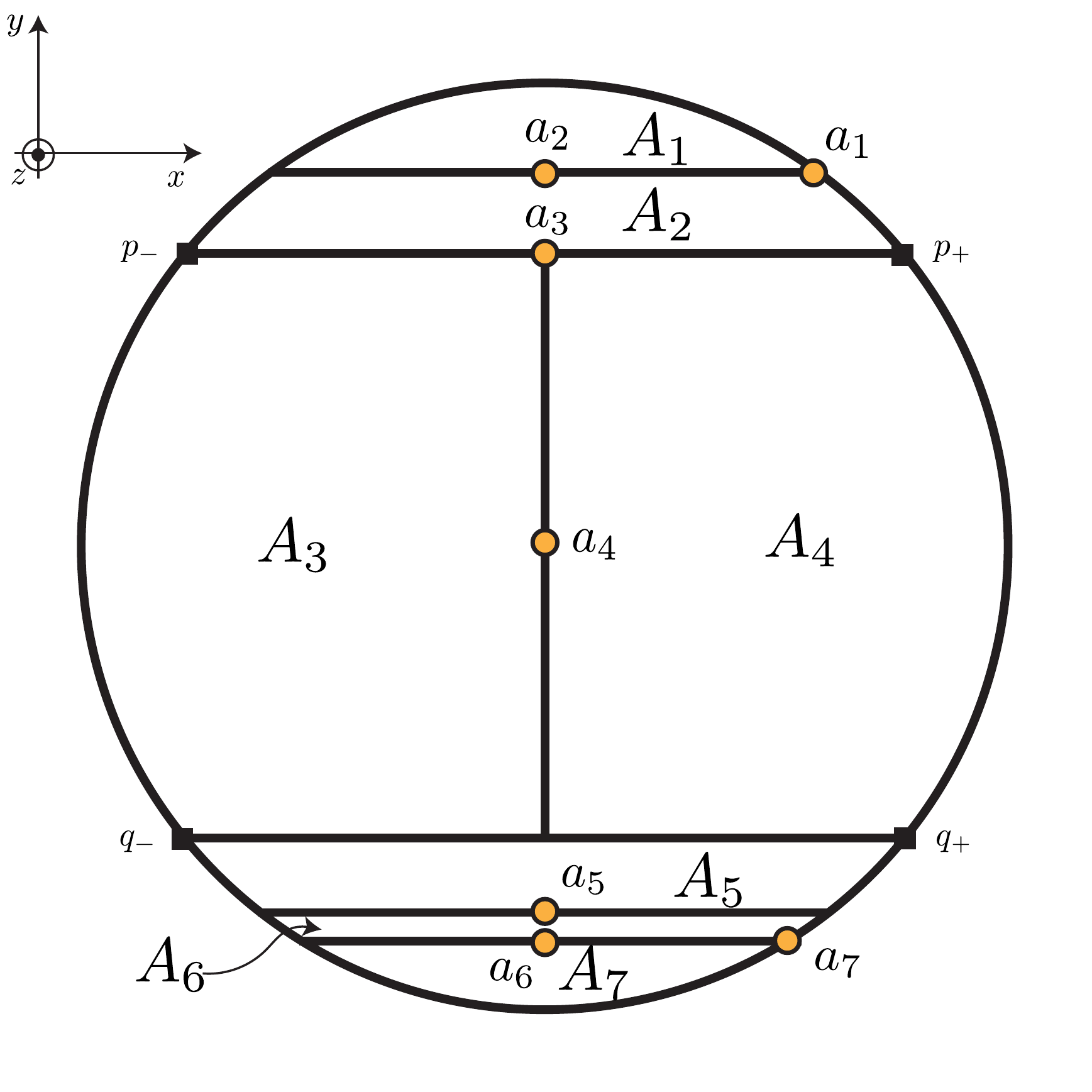}
    \caption{View from above of the seven regions  $A_1,\ldots,A_7$ of $\mathbf{H}_{\geq 0}(\Sp^2_E)$. All lines shown in the figure (which are projections of circular arcs) are  aligned with the either the $x$ or $y$ axis. Also, $p_{\pm}=(\pm\alpha,\sqrt{1-\alpha^2},0)$ and $q_{\pm}=(\pm\alpha,-\sqrt{1-\alpha^2},0)$.}
    \label{fig:S2sevenregion}
\end{figure}

where $$\rho_k:=\sqrt{2-2\cos\left(\frac{k\pi}{7}\right)} \,\,\,\mbox{for $k\in\{1,\ldots,7\}$}.$$

One can directly check that the following seven conditions are satisfied:
\begin{enumerate}
    \item $d_{\mathrm{E}}(A_i,A_j)>\rho_6-\sqrt{3}$ for any $i,j\in\{1,\dots,7\}$ with $\vert i-j \vert=3$.
    
    \item $d_{\mathrm{E}}(a_i,a_j)>\rho_6-\sqrt{3}$ for any $i,j\in\{1,\dots,7\}$ with $\vert i-j \vert=2$.
    
    \item $d_{\mathrm{E}}(a_i,a_j)>2-\sqrt{3}$ for any $i,j\in\{1,\dots,7\}$ with $\vert i-j \vert=3$.
    
    \item $d_{\mathrm{E}}(A_i,a_j)>\rho_5-\sqrt{3}$ for any $i,j\in\{1,\dots,7\}$ with $\vert i-j \vert=2$.
    
    \item $d_{\mathrm{E}}(A_i,a_j)>2-\sqrt{3}$ for any $i,j\in\{1,\dots,7\}$ with $\vert i-j \vert=3$.
    
    \item $\diam(A_i)<\sqrt{3}$ for any $i\in\{1,\dots,7\}$.
    
    \item $d_{\mathrm{E}}(a_i,a_j)<\sqrt{3}$ for any $i,j\in\{1,\dots,7\}$ with $\vert i-j \vert=1$.
\end{enumerate}

In what follows, for two points $v,w\in \Sp^1$ with $d_\mathrm{E}(v,w)<2$, $\arc{v w}$  will denote the (unique) shortest circular arc determined by these two points. 

Now, we define a correspondence $R_\mathrm{E}$ in the following way:
$$R_\mathrm{E}:=\bigcup_{i=1}^7\{(u_i,y):y\in A_i\}\cup\bigcup_{i=1}^7\{(x,a_i):x\in\arc{v_{i+3}v_{i+4}}\}.$$
We now prove that $\dis_{\mathrm{E}}(R_\mathrm{E})<\sqrt{3}$:

First, let us prove that $$\sup_{(x,y),(x',y')\in R_\mathrm{E}}(d_\mathrm{E}(x,x')-d_\mathrm{E}(y,y'))<\sqrt{3}.$$ For this we  carry out a case by case analysis.

\begin{enumerate}
    \item If $(x,y),(x',y')\in\{u_i\}\times A_i$ for some $i\in\{1,\dots,7\}$: Obvious, since $d_\mathrm{E}(x,x')=d_\mathrm{E}(u_i,u_i)=0$.
    
    \item If $(x,y)\in\{u_i\}\times A_i$ and $(x',y')\in\{u_j\}\times A_j$ for some $i,j\in\{1,\dots,7\}$ with $\vert i-j\vert=1$: Obvious, since $d_\mathrm{E}(x,x')=d_\mathrm{E}(u_i,u_j)=\rho_2<\sqrt{3}$.
    
    \item If $(x,y)\in\{u_i\}\times A_i$ and $(x',y')\in\{u_j\}\times A_j$ for some $i,j\in\{1,\dots,7\}$ with $\vert i-j\vert=2$: Obvious, since $d_\mathrm{E}(x,x')=d_\mathrm{E}(u_i,u_j)=\rho_4<\sqrt{3}$.
    
    \item If $(x,y)\in\{u_i\}\times A_i$ and $(x',y')\in\{u_j\}\times A_j$ for some $i,j\in\{1,\dots,7\}$ with $\vert i-j\vert=3$: Observe that $d_\mathrm{E}(x,x')=d_\mathrm{E}(u_i,u_j)=\rho_6>\sqrt{3}$. However, since $d_{\mathrm{E}}(A_i,A_j)>\rho_6-\sqrt{3}$ by  condition (1) above, we have $d_\mathrm{E}(x,x')-d_\mathrm{E}(y,y')<\sqrt{3}$.
    
    \item If $(x,y),(x',y')\in \arc{v_{i+3}v_{i+4}}\times\{a_i\}$ for some $i\in\{1,\dots,7\}$: Obvious, since $d_\mathrm{E}(x,x')\leq\diam(\arc{v_{i+3}v_{i+4}})=\rho_2<\sqrt{3}$.
    
    \item If $(x,y)\in \arc{v_{i+3}v_{i+4}}\times\{a_i\}$ and $(x',y')\in \arc{v_{j+3}v_{j+4}}\times\{a_j\}$ for some $i,j\in\{1,\dots,7\}$ with $\vert i-j\vert=1$: Obvious, since $d_\mathrm{E}(x,x')\leq\diam(\arc{v_{i+3}v_{i+4}}\cup\arc{v_{j+3}v_{j+4}})=\rho_4<\sqrt{3}$.
    
    \item If $(x,y)\in \arc{v_{i+3}v_{i+4}}\times\{a_i\}$ and $(x',y')\in \arc{v_{j+3}v_{j+4}}\times\{a_j\}$ for some $i,j\in\{1,\dots,7\}$ with $\vert i-j\vert=2$: Observe that $d_\mathrm{E}(x,x')\leq\diam(\arc{v_{i+3}v_{i+4}}\cup\arc{v_{j+3}v_{j+4}})=\rho_6>\sqrt{3}$. However, since $d_{\mathrm{E}}(y,y')=d_{\mathrm{E}}(a_i,a_j)>\rho_6-\sqrt{3}$ by  condition (2) above, we have $d_\mathrm{E}(x,x')-d_\mathrm{E}(y,y')<\sqrt{3}$.
    
    \item If $(x,y)\in \arc{v_{i+3}v_{i+4}}\times\{a_i\}$ and $(x',y')\in \arc{v_{j+3}v_{j+4}}\times\{a_j\}$ for some $i,j\in\{1,\dots,7\}$ with $\vert i-j\vert=3$: Since $d_{\mathrm{E}}(y,y')=d_{\mathrm{E}}(a_i,a_j)>2-\sqrt{3}$ by  condition (3) above, we have $d_\mathrm{E}(x,x')-d_\mathrm{E}(y,y')<2-(2-\sqrt{3})=\sqrt{3}$.
    
    \item If $(x,y)\in\{u_i\}\times A_i$ and $(x',y')\in \arc{v_{i+3}v_{i+4}}\times\{a_i\}$ for some $i\in\{1,\dots,7\}$: Note that $u_i\in\arc{v_{i+3}v_{i+4}}$. Hence, $d_\mathrm{E}(x,x')=d_\mathrm{E}(u_i,x')\leq\diam(\arc{v_{i+3}v_{i+4}})<\sqrt{3}$. So, we have $d_\mathrm{E}(x,x')-d_\mathrm{E}(y,y')<\sqrt{3}$.
    
    \item If $(x,y)\in\{u_i\}\times A_i$ and $(x',y')\in \arc{v_{j+3}v_{j+4}}\times\{a_j\}$ for some $i,j\in\{1,\dots,7\}$ with $\vert i-j \vert=1$: Obvious, since $d_\mathrm{E}(x,x')=d_\mathrm{E}(u_i,x')\leq\diam(\{u_i\}\cup\arc{v_{j+3}v_{j+4}})=\rho_3<\sqrt{3}$.
    
    \item If $(x,y)\in\{u_i\}\times A_i$ and $(x',y')\in \arc{v_{j+3}v_{j+4}}\times\{a_j\}$ for some $i,j\in\{1,\dots,7\}$ with $\vert i-j \vert=2$: Observe that $d_\mathrm{E}(x,x')=d_\mathrm{E}(u_i,x')\leq\diam(\{u_i\}\cup\arc{v_{j+3}v_{j+4}})=\rho_5>\sqrt{3}$. However, since $d_\mathrm{E}(A_i,a_j)>\rho_5-\sqrt{3}$ by  condition (4) above, we have $d_\mathrm{E}(x,x')-d_\mathrm{E}(y,y')<\sqrt{3}$.
    
    \item If $(x,y)\in\{u_i\}\times A_i$ and $(x',y')\in \arc{v_{j+3}v_{j+4}}\times\{a_j\}$ for some $i,j\in\{1,\dots,7\}$ with $\vert i-j \vert=3$: Since $d_\mathrm{E}(A_i,a_j)>2-\sqrt{3}$ by  condition (5) above, we have $d_\mathrm{E}(x,x')-d_\mathrm{E}(y,y')<\sqrt{3}$.
\end{enumerate}

Next, we prove $$\sup_{(x,y),(x',y')\in R_\mathrm{E}}(d_\mathrm{E}(y,y')-d_\mathrm{E}(x,x'))<\sqrt{3},$$
for which need to do a case-by case analysis.

\begin{enumerate}
    \item If $(x,y),(x',y')\in\{u_i\}\times A_i$ for some $i\in\{1,\dots,7\}$: Since $\diam(A_i)<\sqrt{3}$ by  condition (6) above, we have $d_\mathrm{E}(y,y')<\sqrt{3}$ so that $d_\mathrm{E}(y,y')-d_\mathrm{E}(x,x')<\sqrt{3}$.
    
    \item If $(x,y)\in\{u_i\}\times A_i$ and $(x',y')\in\{u_j\}\times A_j$ for some $i,j\in\{1,\dots,7\}$ with $\vert i-j\vert=1$: Obvious, since $d_\mathrm{E}(x,x')=d_\mathrm{E}(u_i,u_j)=\rho_2$ and $d_\mathrm{E}(y,y')-d_\mathrm{E}(x,x')\leq 2-\rho_2<\sqrt{3}$.
    
    \item If $(x,y)\in\{u_i\}\times A_i$ and $(x',y')\in\{u_j\}\times A_j$ for some $i,j\in\{1,\dots,7\}$ with $\vert i-j\vert=2$: Obvious, since $d_\mathrm{E}(x,x')=d_\mathrm{E}(u_i,u_j)=\rho_4$ and $d_\mathrm{E}(y,y')-d_\mathrm{E}(x,x')\leq 2-\rho_4<\sqrt{3}$.
    
    \item If $(x,y)\in\{u_i\}\times A_i$ and $(x',y')\in\{u_j\}\times A_j$ for some $i,j\in\{1,\dots,7\}$ with $\vert i-j\vert=3$: Obvious, since $d_\mathrm{E}(x,x')=d_\mathrm{E}(u_i,u_j)=\rho_6$ and $d_\mathrm{E}(y,y')-d_\mathrm{E}(x,x')\leq 2-\rho_6<\sqrt{3}$.
    
    \item If $(x,y),(x',y')\in \arc{v_{i+3}v_{i+4}}\times\{a_i\}$ for some $i\in\{1,\dots,7\}$: Obvious, since $d_\mathrm{E}(x,x')=d_\mathrm{E}(a_i,a_i)=0$.
    
    \item If $(x,y)\in \arc{v_{i+3}v_{i+4}}\times\{a_i\}$ and $(x',y')\in \arc{v_{j+3}v_{j+4}}\times\{a_j\}$ for some $i,j\in\{1,\dots,7\}$ with $\vert i-j\vert=1$: Since $d_\mathrm{E}(a_i,a_j)<\sqrt{3}$ by  condition (7) above, we have $d_\mathrm{E}(y,y')-d_\mathrm{E}(x,x')=d_\mathrm{E}(a_i,a_j)-d_\mathrm{E}(x,x')<\sqrt{3}$.
    
    \item If $(x,y)\in \arc{v_{i+3}v_{i+4}}\times\{a_i\}$ and $(x',y')\in \arc{v_{j+3}v_{j+4}}\times\{a_j\}$ for some $i,j\in\{1,\dots,7\}$ with $\vert i-j\vert=2$: Obvious, since $d_\mathrm{E}(x,x')\geq\rho_2$. Hence, $d_\mathrm{E}(y,y')-d_\mathrm{E}(x,x')\leq 2-\rho_2<\sqrt{3}$.
    
    \item If $(x,y)\in \arc{v_{i+3}v_{i+4}}\times\{a_i\}$ and $(x',y')\in \arc{v_{j+3}v_{j+4}}\times\{a_j\}$ for some $i,j\in\{1,\dots,7\}$ with $\vert i-j\vert=3$: Obvious, since $d_\mathrm{E}(x,x')\geq\rho_4$. Hence, $d_\mathrm{E}(y,y')-d_\mathrm{E}(x,x')\leq 2-\rho_4<\sqrt{3}$.
    
    \item If $(x,y)\in\{u_i\}\times A_i$ and $(x',y')\in \arc{v_{i+3}v_{i+4}}\times\{a_i\}$ for some $i\in\{1,\dots,7\}$: Since $a_i\in A_i$ and $\diam(A_i)<\sqrt{3}$ by  condition (6),  we have $d_\mathrm{E}(y,y')-d_\mathrm{E}(x,x')<\sqrt{3}$.
    
    \item If $(x,y)\in\{u_i\}\times A_i$ and $(x',y')\in \arc{v_{j+3}v_{j+4}}\times\{a_j\}$ for some $i,j\in\{1,\dots,7\}$ with $\vert i-j \vert=1$: Observe that $d_\mathrm{E}(x,x')\geq\rho_1$. Hence, we have $d_\mathrm{E}(y,y')-d_\mathrm{E}(x,x')\leq 2-\rho_1<\sqrt{3}$.
    
    \item If $(x,y)\in\{u_i\}\times A_i$ and $(x',y')\in \arc{v_{j+3}v_{j+4}}\times\{a_j\}$ for some $i,j\in\{1,\dots,7\}$ with $\vert i-j \vert=2$: Observe that $d_\mathrm{E}(x,x')\geq\rho_3$. Hence, we have $d_\mathrm{E}(y,y')-d_\mathrm{E}(x,x')\leq 2-\rho_3<\sqrt{3}$.
    
    \item If $(x,y)\in\{u_i\}\times A_i$ and $(x',y')\in \arc{v_{j+3}v_{j+4}}\times\{a_j\}$ for some $i,j\in\{1,\dots,7\}$ with $\vert i-j \vert=3$: Observe that $d_\mathrm{E}(x,x')\geq\rho_5$. Hence, we have $d_\mathrm{E}(y,y')-d_\mathrm{E}(x,x')\leq 2-\rho_5<\sqrt{3}$.
\end{enumerate}
Hence, $\dis_{\mathrm{E}}(R_\mathrm{E})<\sqrt{3}$ as we required. This completes the proof.
\end{proof}

%
%
%
%


\appendix

\section{A succinct proof of Theorem \ref{thm:strBU}}\label{app:proof-BU}
In this subsection we provide a proof of Theorem \ref{thm:strBU} following a strategy suggested by Matou\v{s}ek  in \cite[page 41]{matousek2003using} and due to Arnold Wa{\ss}mer.

\begin{lemma}\label{lemma:origindiam}
If a simplex contains $0\in\mathbb{R}^n$ and has all vertices on $\mathbb{S}^{n-1}$, then there are vertices $u$ and $v$ of the simplex such that $d_{\mathbb{S}^{n-1}}(u,v)\geq \zeta_{n-1}$.
\end{lemma}
\begin{proof}
We give the proof here  for the completeness -- the proof is basically the same as that  of \cite[Lemma 1]{dubins1981equidiscontinuity}. Let $u_1,\dots,u_{n+1}$ be (not necessarily distinct) vertices of a simplex such that their convex hull contains the origin $0\in\R^n$. Therefore, there are nonnegative numbers $\lambda_1,\dots,\lambda_{n+1}$ such that $\sum_{i=1}^{n+1} \lambda_n=1$ and $0=\sum_{i=1}^{n+1} \lambda_iu_i$. Then,
\begin{align*}
    0=\left\Vert\sum_{i=1}^{n+1} \lambda_iu_i\right\Vert^2=\sum_{i\neq j}\lambda_i\lambda_j\langle u_i,u_j\rangle+\sum_{i=1}^{n+1} \lambda_i^2.
\end{align*}
Moreover, since $0\leq \sum_{i\neq j}(\lambda_i-\lambda_j)^2=2n\sum_{i=1}^{n+1}\lambda_i^2-2\sum_{i\neq j}\lambda_i\lambda_j$, we have $$\sum_{i=1}^{n+1}\lambda_i^2\geq\frac{1}{n}\sum_{i\neq j}\lambda_i\lambda_j.$$
Hence, we have
$$0\geq\sum_{i\neq j}\lambda_i\lambda_j\left(\langle u_i,u_j\rangle+\frac{1}{n}\right).$$
Thus, there must be some distinct $i$ and $j$ such that $\langle u_i,u_j\rangle\leq-\frac{1}{n}$ so that
$$d_{\mathbb{S}^{n-1}}(u_i,u_j)\geq\arccos\left(-\frac{1}{n}\right)=\zeta_{n-1}.$$
\end{proof}

Below, the notation $V(T)$ for a triangulation $T$ of the cross-polytope $\widehat{\mathbb{B}}^n$ will denote its set of vertices.

\begin{lemma}\label{lemma:Matousek}
Let $T$ be a triangulation of the cross-polytope $\widehat{\mathbb{B}}^n$ which is antipodally symmetric at the boundary (i.e., if $\Delta\subset\partial\,\widehat{\mathbb{B}}^n$ is a simplex in $T$, then $-\Delta\subset\partial\,\widehat{\mathbb{B}}^n$ is also in $T$), and let $g:V(T)\rightarrow \mathbb{S}^{n-1}$  be a mapping that satisfies $g(-v)=-g(v)\in \mathbb{S}^{n-1}$ for all vertices $v\in V(T)$ lying on the boundary of $\widehat{\mathbb{B}}^n$. Then, there exist vertices $u,v\in V(T)$ with $d_{\mathbb{S}^{n-1}}(g(u),g(v))\geq \zeta_{n-1}$.
\end{lemma}
\begin{proof}
By Lemma \ref{lemma:origindiam} it is enough to show that some simplex $\{v_1,\dots,v_m\}$ of $T$ satisfies $$0\in \mathrm{Conv}(g(v_1),g(v_2),\dots,g(v_m)).$$ Suppose not, then one can construct the continuous map $\phi:\widehat{\mathbb{B}}^n\rightarrow\mathbb{R}^n\backslash\{0\}$ such that $\phi(a_1u_1+\dots+a_mu_m):=a_1g(u_1)+\dots+a_mg(u_m)$ where $\{u_1,\dots,u_m\}$ is a simplex of $T$, $a_1,\dots,a_m\in [0,1]$, and $\sum_{i=1}^m a_i=1$. Next, one can construct the continuous map $\widehat{\phi}:\widehat{\mathbb{B}}^n\rightarrow\Sp^{n-1}$ such that $\widehat{\phi}(x):=\frac{\phi(x)}{\Vert \phi(x) \Vert}$ for each $x\in\widehat{\mathbb{B}}^n$. Moreover, this map $\widehat{\phi}$ is antipode preserving on the boundary since if $x\in\partial\,\widehat{\mathbb{B}}^n$ satisfies $x=a_1v_1+\dots+a_mv_m$ where $\{v_1,\dots,v_m\}$ is a simplex of $\partial\,\widehat{\mathbb{B}}^n$, $\phi(x)=a_1g(v_1)+\dots+a_mg(v_m)$ and $\phi(-x)=a_1g(-v_1)+\dots+a_mg(-v_m)$ so that $\phi(-x)=-\phi(x)$. This is contradiction to the classical Borsuk-Ulam theorem since $\widehat{\phi}\circ\alpha^{-1}:\mathbb{B}^n\rightarrow\mathbb{S}^{n-1}$ is continuous and antipode preserving on the boundary where (below, for a vector $v$ by $\|v\|_1$ we note its 1-norm):
\begin{align*}
    \alpha:\widehat{\mathbb{B}}^n&\longrightarrow\mathbb{B}^n\\
    x&\longmapsto
    \begin{cases} (0,\dots,0) & \textit{ if }x=(0,\dots,0) \\ x\,\frac{\|x\|_1}{\|x\|}& \textit{ otherwise}\end{cases}
\end{align*}
is the natural bi-Lipschitz homeomorphism between $\widehat{\mathbb{B}}^n$ and $\mathbb{B}^n$  from the unit cross-polytope to the  closed unit ball).
\end{proof}

Now we are ready to prove Theorem \ref{thm:strBU}.

\begin{proof}[Proof of Theorem \ref{thm:strBU}]
Let $f:\mathbb{B}^n\rightarrow \mathbb{S}^{n-1}$ be a map that is antipode preserving on the boundary of $\mathbb{B}^n$. Now, fix arbitrary $\delta\geq 0$ such that for any $x\in \mathbb{B}^n$, there exists an open neighborhood $U_x$ of $x$ with $\diam(f(U_x))\leq\delta$. Fix $\varepsilon>0$ smaller than the Lebesgue number of the open covering $\{U_x\}_{x\in\mathbb{B}^n}$.

Let $\alpha:\widehat{\mathbb{B}}^n\longrightarrow\mathbb{B}^n$ be the natural (fattening) homeomorphism used in the proof of Lemma \ref{lemma:Matousek}. One can construct a triangulation $T$ of $\widehat{\mathbb{B}}^n$ satisfying the following two properties.
\begin{enumerate}
    \item $T$ is antipodally symmetric on the boundary of $\widehat{\mathbb{B}}^n$.
    \item $T$ is fine enough so that $\Vert \alpha(u)-\alpha(v) \Vert\leq\varepsilon$ for any two adjacent vertices $u$ and $v$.
\end{enumerate}
Then, by  Lemma \ref{lemma:Matousek}, there exist adjacent vertices $u,v$ such that $d_{\mathbb{S}^{n-1}}(f\circ\alpha(u),f\circ\alpha(v))\geq \zeta_{n-1}$. Choose $x=\alpha(u)$ and $y=\alpha(v)$. Because of the choice of $\varepsilon$, both $x$ and $y$ are contained in some $U_x$. Hence, $\delta\geq\diam(f(U_x))\geq\zeta_{n-1}$ which concludes as in the proof of Corollary \ref{cor:strBU}.
\end{proof}

\section{The Gromov-Hausdorff distance between a sphere and an interval}\label{app:katz-lb}

To make this paper self-contained, we include a proof of the following proposition.

\begin{proposition}\label{thm:sphereinterval}
Let $n$ be any positive integer. Then $\dis(f)\geq \frac{2\pi}{3}$ for any function $f:\Sp^n\longrightarrow \R$.

As a consequence, 
$\dgh(\Sp^n,I)\geq\frac{\pi}{3}$ for any interval $I\subseteq\R.$
\end{proposition}
\begin{proof}
Note that it is enough to prove the claim for $n=1$. We adapt an argument from the proof of \cite[Lemma 2.3]{katz2020torus}.

Fix an arbitrary $\varepsilon>0$. Consider an antipodally symmetric triangulation of $\Sp^1$ with  vertex set $V\subset\Sp^1$ such that $d_{\Sp^1}(p,q)\leq\varepsilon$ for any two adjacent vertices $p,q\in V$. Then, let $\widetilde{f}:\Sp^1\longrightarrow I$ be the linear interpolation of $f\vert_V:V\longrightarrow I$. Now, by the classical Borsuk-Ulam theorem, there exists $x\in\Sp^1$ such that $\widetilde{f}(x)=\widetilde{f}(-x)$. Let $p,q\in V$ be  such that $x\in\arc{pq}$. Then $I\cap J\neq\emptyset$ where $I$ is the closed interval between $f(p)$ and $f(q)$, and $J$ is the closed interval between $f(-p)$ and $f(-q)$ (since $I$ and $J$ both contain $\widetilde{f}(x)=\widetilde{f}(-x)$). Without loss of generality, we can assume that $f(-p)\in I$. Now, let $$r:=\begin{cases}p&\text{if }\vert f(-p)-f(p) \vert\leq\vert f(-p)-f(q) \vert\\ q&\text{if }\vert f(-p)-f(p) \vert>\vert f(-p)-f(q) \vert.\end{cases}$$
Then, 
$\vert f(-p)-f(r) \vert\leq\frac{\mathrm{length}(I)}{2}\leq\frac{1}{2}(\dis(f)+\varepsilon).$
Hence,
$$\pi-\varepsilon\leq d_{\Sp^1}(-p.r)\leq\dis(f)+\vert f(-p)-f(r) \vert\leq \frac{3}{2}\dis(f)+\frac{1}{2}\varepsilon,$$
so that $\dis(f)\geq\frac{2\pi}{3}-\varepsilon$. The conclusion follows.
\end{proof}

\section{Regular polygons and $\Sp^1$.}\label{sec:polys}

In this appendix we compute the distance between regular polygons and also between the circle and a regular polygon.

The following map from metric spaces to metric spaces will be useful below. For a metric space $(X,d_X)$, consider the pseudo ultrametric space $(X,u_X)$ where $u_X:X\times X\rightarrow \R$ is defined by 
$$(x,x')\mapsto u_X(x,x'):=\inf\bigg\{\max_{0\leq i \leq n-1} d_X(x_i,x_{i+1}):\,x=x_0,\ldots,x_n=x' \text{ for some }n\geq 1\bigg\}.$$
Now, define $\mathbf{U}(X)$ to be the quotient metric space induced by $(X,u_X)$ under the equivalence $x\sim x'$ if and only if $u_X(x,x')=0$. One then has the following, whose proof we omit:

\begin{proposition}\label{prop:geod-um}
For any path connected metric space $X$ it holds that $\mathbf{U}(X)=\ast.$
\end{proposition}

We also have the following result establishing that $\mathbf{U}:\mathcal{M}_b\rightarrow \mathcal{M}_b$ is $1$-Lipschitz:

\begin{theorem}[\cite{carlsson2010characterization}]\label{thm:stab-um}
For all bounded metric spaces $X$ and $Y$ one has 
$$\dgh(X,Y)\geq\dgh(\mathbf{U}(X),\mathbf{U}(Y)).$$
\end{theorem}

For each integer $n\geq 3$, let $P_n$ be the regular polygon with $n$ vertices inscribed in $\Sp^1$. We also let $P_2=\mathbb{S}^0$.   Furthermore, we endow $P_n$ with the restriction of the geodesic distance on $\mathbb{S}^1$. We then have:

\begin{proposition}[$\dgh$ between $\Sp^1$ and inscribed regular polygons]\label{prop:s1-pn}
For all $n\geq 2$, we have that  $$\dgh(\mathbb{S}^1,P_n) = \frac{\pi}{n}.$$
\end{proposition}
\begin{proof}
That $\dgh(\mathbb{S}^1,P_n)\geq \frac{\pi}{n}$ can be obtained as follows: by Theorem \ref{thm:stab-um}, $$\dgh(\Sp^1,P_n)\geq \dgh(\mathbf{U}(\Sp^1),\mathbf{U}(P_n)).$$
But, since $\mathbf{U}(\Sp^1) = \ast$ by Proposition \ref{prop:geod-um}, and $\mathbf{U}(P_n)$ is isometric to the metric space over $n$ points with all non-zero pairwise distances equal to $\frac{2\pi}{n}$, from the above inequality and equation (\ref{eq:easy}) we have $\dgh(\Sp^1,P_n)\geq \frac{1}{2} \diam(\mathbf{U}(P_n)) =  \frac{\pi}{n}$. The inequality $\dgh(\mathbb{S}^1,P_n)\leq \frac{\pi}{n}$ follows from the fact that $\dgh(\Sp^1,P_{n})\leq \dh(\Sp^1,P_n) = \frac{\pi}{n}$.
\end{proof}

Note that if $\Sp^1$ and $P_n$ as both endowed with the Euclidean distance (respectively denoted by $\Sp^1_\mathrm{E}$ and $(P_n)_\mathrm{E}$), then in analogy with Proposition \ref{prop:s1-pn}, we have the following proposition which 
 solves a question posed in \cite{adams2018vietoris}. The proof is slightly different from that of Proposition \ref{prop:s1-pn}. 

\begin{proposition}\label{prop-pn-eucl}
For all $n\geq 2$, we have that $\dgh(\Sp^1_\mathrm{E},(P_n)_\mathrm{E}) = \sin\left(\frac{\pi}{n}\right).$
\end{proposition}
 
\begin{proof} One can prove $\dgh(\mathbb{S}^1_{\mathrm{E}},(P_n)_{\mathrm{E}})\geq\sin\left(\frac{\pi}{n}\right)$ by invoking $\mathbf{U}$  as in the proof of Proposition \ref{prop:s1-pn}. In order to prove $\dgh(\mathbb{S}^1_{\mathrm{E}},(P_n)_{\mathrm{E}})\leq\sin\left(\frac{\pi}{n}\right)$, let us construct a specific correspondence $R$ between $\mathbb{S}^1_{\mathrm{E}}$ and $(P_n)_{\mathrm{E}}$. Let $u_1,\dots,u_n$ be the vertices of $(P_n)_{\mathrm{E}}$, and $V_1,\dots,V_n$ be the Voronoi regions of $\Sp^1$ induced by $u_1,\dots,u_n$. Now, let
$$R:=\bigcup_{i=1}^n V_i\times\{u_i\}.$$
Then, we claim $\dis_{\mathrm{E}}(R)\leq 2\sin\left(\frac{\pi}{n}\right)$. To prove this, it is enough to check the following two conditions via standard trigonometric identities:
\begin{enumerate}
    \item $2\sin\left(\frac{k\pi}{n}\right)-2\sin\left(\frac{(k-1)\pi}{n}\right)\leq 2\sin\left(\frac{\pi}{n}\right)$ for $1\leq k \leq\lfloor\frac{n}{2}\rfloor$.
    \item $2-2\sin\left(\frac{\lfloor\frac{n}{2}\rfloor\pi}{n}\right)\leq 2\sin\left(\frac{\pi}{n}\right)$.
\end{enumerate}
 Hence, $\dgh(\mathbb{S}^1_{\mathrm{E}},(P_n)_{\mathrm{E}})\leq\sin\left(\frac{\pi}{n}\right)$ as we required.
\end{proof}

We now pose the following question and provide partial information about it in Proposition \ref{prop:dgh-polygons-supra-diag}:
\begin{question}
Determine, for all $m,n\in\N$ the value of $\mathfrak{p}_{m,n}:=\dgh(P_m,P_n).$
\end{question}

\begin{remark}
By simple arguments which we omit one can prove that $\mathfrak{p}_{2,3}=\frac{\pi}{3}$, $\mathfrak{p}_{2,4} = \frac{\pi}{4}$, $\mathfrak{p}_{2,5} = \frac{2\pi}{5}$ and $\mathfrak{p}_{2,6} = \frac{\pi}{3}$. Also Proposition \ref{prop:s1-pn} indicates that $\mathfrak{p}_{2,n}$ tends to $\frac{\pi}{2}$ as $n\rightarrow \infty$. Then, these calculations imply that $n\mapsto \mathfrak{p}_{2,n}$ is \emph{not} monotonically increasing towards $\frac{\pi}{2}$; cf. Question \ref{question:mono}.
\end{remark}

\begin{proposition}\label{prop:dgh-polygons-supra-diag}
For any integer $0<m<\infty$, $\mathfrak{p}_{m,m+1}=\frac{\pi}{m+1}$.
\end{proposition}
\begin{proof}
First, let us prove that  $\mathfrak{p}_{m,m+1}\leq\frac{\pi}{m+1}$. We construct a correspondence $R$ between $P_m$ and $P_{m+1}$ such that $\dis(R)\leq\frac{2\pi}{m+1}$. Let $u_1,\dots,u_m$ be the vertices of $P_m$ and $v_1,\dots,v_m,v_{m+1}$ be the vertices of $P_{m+1}$. Consider the correspondence
$$R:=\bigcup_{i=1}^m \{(u_m,v_m)\}\cup\{(u_m,v_{m+1})\}.$$
Then, for any $i,j\in\{1,\dots,m\}$,
\begin{align*}
    &\vert d_{\Sp^1}(u_i,u_j)-d_{\Sp^1}(v_i,v_j)\vert\\
    &=\left\vert\frac{2\pi}{m}\cdot\min\{\vert i-j\vert,m-\vert i-j\vert\}-\frac{2\pi}{m+1}\cdot\min\{\vert i-j\vert,m+1-\vert i-j\vert\}\right\vert\\
    &=\left\vert\frac{2\pi k}{m}-\frac{2\pi k}{m+1}\right\vert\text{ or }\left\vert\frac{2\pi k}{m}-\frac{2\pi (k+1)}{m+1}\right\vert\,\,\,\left(\text{for some }0\leq k\leq\left\lfloor\frac{m}{2}\right\rfloor\right)\\
    &=\frac{2\pi k}{m(m+1)}\text{ or }\frac{2\pi}{m+1}\left(1-\frac{k}{m}\right)\,\,\,\left(\text{for some }0\leq k\leq\left\lfloor\frac{m}{2}\right\rfloor\right)\\
    &\leq\frac{2\pi}{m+1}.
\end{align*}
Also, for any $i\in\{1,\dots,m\}$,
\begin{align*}
    &\vert d_{\Sp^1}(u_i,u_m)-d_{\Sp^1}(v_i,v_{m+1})\vert\\
    &=\left\vert\frac{2\pi}{m}\cdot\min\{m-i,i\}-\frac{2\pi}{m+1}\cdot\min\{m+1-i,i\}\right\vert\\
    &=\left\vert\frac{2\pi k}{m}-\frac{2\pi k}{m+1}\right\vert\text{ or }\left\vert\frac{2\pi k}{m}-\frac{2\pi (k+1)}{m+1}\right\vert\,\,\,\left(\text{for some }0\leq k\leq\left\lfloor\frac{m}{2}\right\rfloor\right)\\
    &=\frac{2\pi k}{m(m+1)}\text{ or }\frac{2\pi}{m+1}\left(1-\frac{k}{m}\right)\,\,\,\left(\text{for some }0\leq k\leq\left\lfloor\frac{m}{2}\right\rfloor\right)\\
    &\leq\frac{2\pi}{m+1}.
\end{align*}
Hence, one  concludes that  $\dis(R)\leq\frac{2\pi}{m+1}$ as we wanted.\vspace{\baselineskip}

Next, let us prove that  $\mathfrak{p}_{m,m+1}\geq\frac{\pi}{m+1}$. Fix an arbitrary correspondence $R$ between $P_m$ and $P_{m+1}$. Then, there must be a vertex $u_i$ of $P_m$, and two vertices $v_j,v_k$ of $P_{m+1}$ such that $(u_i,v_j),(u_i,v_k)\in R$. Hence,
$$\dis(R)\geq\vert d_{\Sp^1}(u_i,u_i)-d_{\Sp^1}(v_j,v_k)\vert=\frac{2\pi}{m+1}.$$
Since $R$ is arbitrary, one  concludes that $\mathfrak{p}_{m,m+1}\geq\frac{\pi}{m+1}$ as we wanted.
\end{proof}

\section{The proof of Proposition \ref{prop:ub} -- an alternative method}\label{sec:alternative-method}

We construct for $\Sp^1$ and $\Sp^2$,  both a finite cover by closed sets and associated finite set of points $\mathbb{X}\subset \Sp^1$ and $\mathbb{Y}\subset \Sp^2$. These will be the main ingredients for constructing a correspondence between with  distortion at most $\frac{2\pi}{3}$.

\subsection{$\mathbb{X}$ and a cover of $\Sp^1$ }
Consider points $\mathbb{X}=\{x_1,x_2,x_3,x_4,x_5,x_6\}$ on $\Sp^1$ given (in clockwise order) by the vertices of an inscribed regular hexagon. We denote for each $i=1,\ldots,6$ by $A_i$ the set of those points on $\Sp^1$ whose (geodesic) distance to $x_i$ is at most $\frac{\pi}{6}$. The interiors of these six closed sets constitute non-intersecting arcs of length $\frac{\pi}{6}$; see Figure \ref{fig:S1}. Then, the sets $\{A_1,\ldots,A_6\}$ cover $\Sp^1$ and all have diameter $\diam(A_i)=\frac{\pi}{3}$. The interpoint distances between all points in $\mathbb{X}$ is given by the  matrix in equation  (\ref{eq:dX}).

\begin{figure}
\begin{center}
\includegraphics[width=0.5\linewidth]{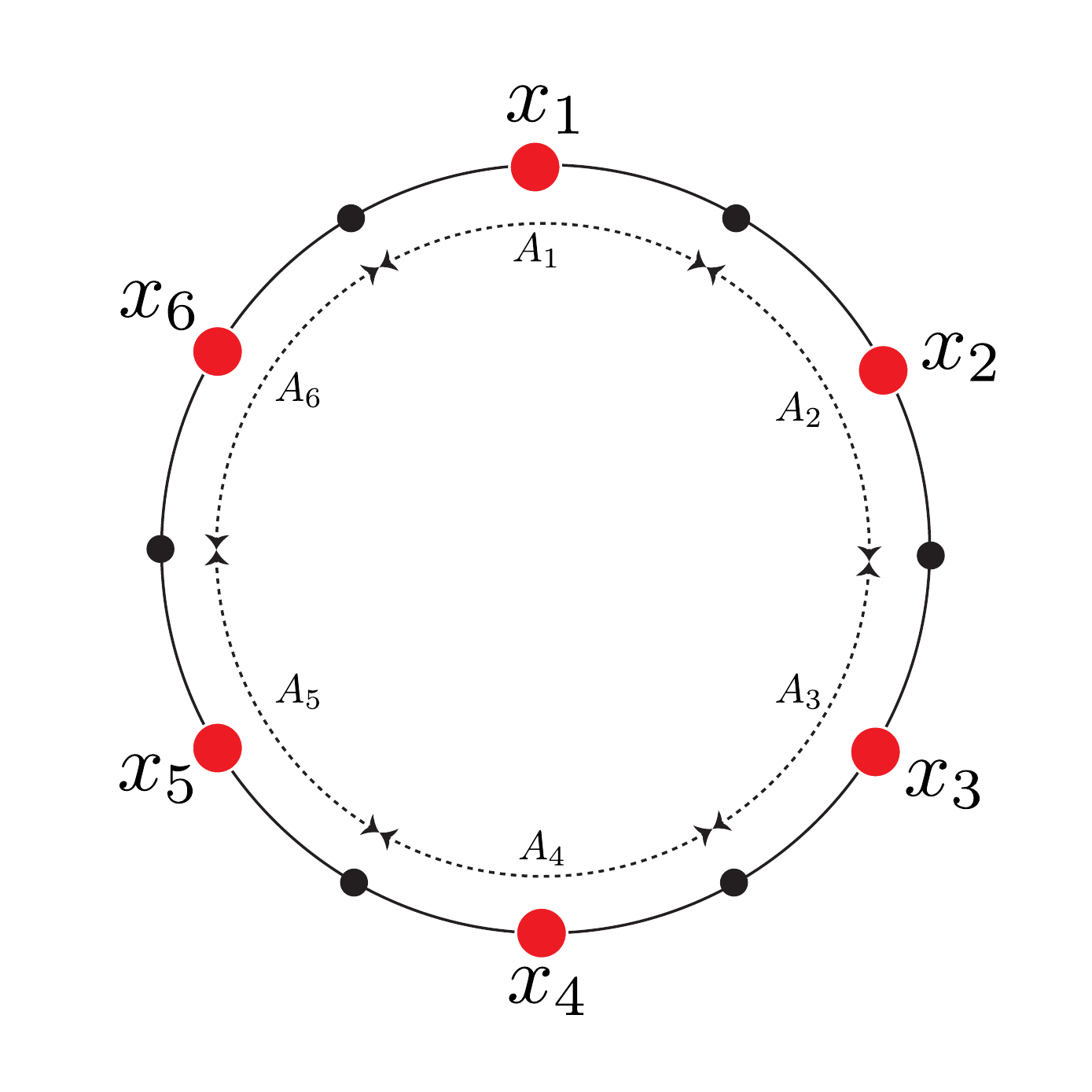}
\end{center}
\caption{Closed cover $\{A_1,\ldots,A_6\}$ of $\Sp^1$ and the set of centers $\mathbb{X}$. \label{fig:S1}}
\end{figure}

\begin{equation}\label{eq:dX}
\big(\!\big(d_{\Sp^1}(x_i,x_j)\big)\!\big) = \bbordermatrix{
~ & x_1 & x_2 & x_3 & x_4 & x_5 & x_6 \cr
x_1 & 0 & \frac{\pi}{3} & \frac{2\pi}{3} & \pi & \frac{2\pi}{3} & \frac{\pi}{3}\cr
x_2 & \ast & 0 & \frac{\pi}{3}  & \frac{2\pi}{3} & \pi & \frac{\pi}{3}\cr
x_3 & \ast & \ast &  0 & \frac{\pi}{3}  & \frac{2\pi}{3} & \pi\cr
x_4 & \ast & \ast &  \ast & 0 & \frac{\pi}{3}  & \frac{2\pi}{3} \cr
x_5 & \ast & \ast &  \ast & \ast & 0 & \frac{\pi}{3} \cr
x_6 & \ast & \ast &  \ast & \ast & \ast & 0 \cr
}\end{equation}

\subsection{$\mathbb{Y}$ and the cover of $\Sp^2$}
Consider the parametrization of $\Sp^2\subset \R^3$ via polar coordinates $\varphi\in [0,2\pi]$ and $\theta\in[0,\pi]$ given by $$(\varphi,\theta)\mapsto \Phi(\varphi,\theta):=\big(\cos(\varphi)\sin(\theta),\sin(\varphi)\sin(\theta),\cos(\theta)\big) \in \R^3.$$

Consider the closed subsets $C_1,C_2,C_3,C_4,C_5,$ and $C_6$ of $[0,2\pi]\times [0,\pi]$ given by:
\begin{itemize}
\item $C_1=[0,\frac{2\pi}{3}]\times[0,\frac{\pi}{2}].$ 
\item $C_2=[\frac{\pi}{3},\pi]\times[\frac{\pi}{2},\pi].$ 
\item $C_3=[\frac{2\pi}{3},\frac{4\pi}{3}]\times[0,\frac{\pi}{2}].$ 
\item $C_4=[\pi,\frac{5\pi}{3}]\times[\frac{\pi}{2},\pi].$ 
\item $C_5=[\frac{4\pi}{3},2\pi]\times[0,\frac{\pi}{2}].$ 
\item $C_6=[\frac{5\pi}{3},\frac{7\pi}{3}]\times[\frac{\pi}{2},\pi].$ 
\end{itemize}

Notice that $[0,2\pi]\times [0,\pi]\subset \bigcup_{i=1}^6 C_i.$ Let $\theta_0:=\arcsin(1/\sqrt{3}).$ Notice that $\theta_0\in[\frac{\pi}{6},\frac{\pi}{4}]$. Define the following points in the $(\varphi,\theta)$ plane:
\begin{itemize}
\item $c_1=(\frac{\pi}{3},\theta_0).$ 
\item $c_2= (\frac{2\pi}{3},\pi-\theta_0).$ 
\item $c_3= (\pi,\theta_0).$ 
\item $c_4= (\frac{4\pi}{3},\pi-\theta_0).$ 
\item $c_5= (\frac{5\pi}{3},\theta_0).$ 
\item $c_6= (2\pi,\pi-\theta_0).$ 
\end{itemize}
Notice that $c_i\in C_i$ for each $i\in\{1,\ldots,6\}.$ The sets $C_i$ and the points $c_i$ are depicted in Figure \ref{fig:S2}.

Now, for each $i\in\{1,\ldots,6\}$ define $B_i = \Phi(C_i)$ and 
$y_i = \Phi(c_i)$. Let $\mathbb{Y}=\{y_1,\ldots,y_6\}$. By construction, the sets $B_1,\ldots,B_6$ form a closed cover of $\Sp^2$ such that $\diam(B_i)= \frac{2\pi}{3}$ for all $i$. Also, by direct computation we find the interpoint geodesic distances between points in $\mathbb{Y}$ to be:

\begin{equation}\label{eq:dm2} \big(\!\big(d_{\Sp^2}(y_i,y_j)\big)\!\big) = \bbordermatrix{
~ & y_1 & y_2 & y_3 & y_4 & y_5 & y_6 \cr
y_1 & 0 & \frac{2\pi}{3} & \frac{\pi}{3} & \pi & \frac{\pi}{3} & \frac{2\pi}{3}\cr
y_2 & \ast & 0 & \frac{2\pi}{3}  & \frac{\pi}{3} & \pi & \frac{\pi}{3}\cr
y_3 & \ast & \ast &  0 & \frac{2\pi}{3}  & \frac{\pi}{3} & \pi\cr
y_4 & \ast & \ast &  \ast & 0 & \frac{2\pi}{3}  & \frac{\pi}{3} \cr
y_5 & \ast & \ast &  \ast & \ast & 0 & \frac{2\pi}{3} \cr
y_6 & \ast & \ast &  \ast & \ast & \ast & 0 \cr
}
\end{equation}

\begin{figure}
\begin{center}
\includegraphics[width=0.5\linewidth]{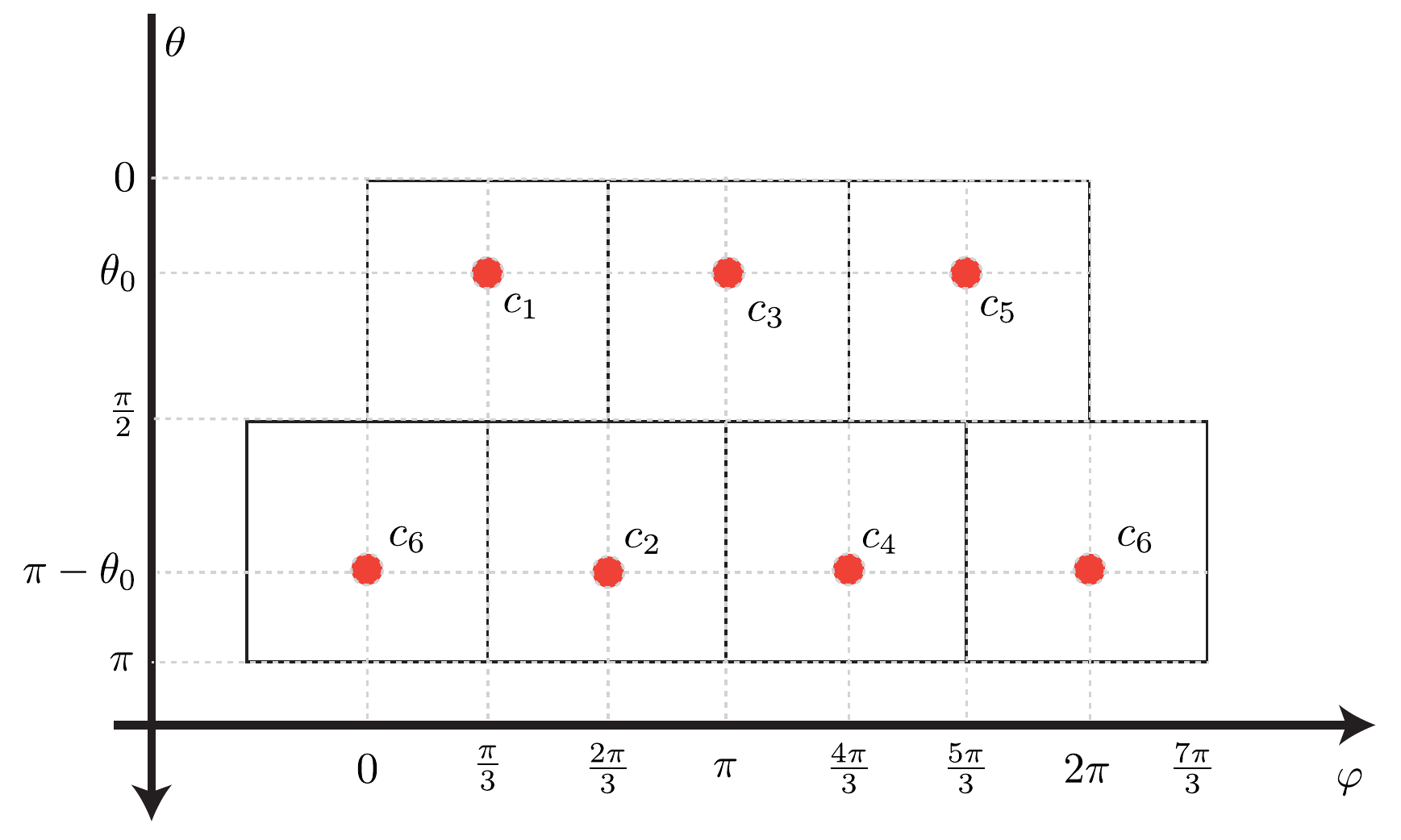}
\end{center}
\caption{Spherical coordinate representation of the closed cover $\{B_1,\ldots,B_6\}$ of $\Sp^2$ and the corresponding center points $\mathbb{Y}$. \label{fig:S2}}
\end{figure}

\subsection{The correspondence and its distortion}

Let 

\begin{equation}\label{eq:R}
R:=\bigcup_{i=1}^6 A_i\times\{y_i\} \cup\bigcup_{i=1}^6 \{x_i\}\times B_i \subset \Sp^1\times \Sp^2.
\end{equation}
Since $\{A_i,\,i=1,\ldots,6\}$ and $\{B_i,\,i=1,\ldots,6\}$  cover $\Sp^1$ and $\Sp^2$, respectively, it follows that $R$ is a correspondence between  $\Sp^1$ and $\Sp^2.$

\begin{proposition}\label{prop:S13altcrspdc}
$\dis(R)\leq\frac{2\pi}{3}$.
\end{proposition}
\begin{proof}[Proof of Proposition \ref{prop:S13altcrspdc}]
We'll prove that $\dis(R)\leq \frac{2\pi}{3}.$ Note that $\dis(R) = \max(\alpha,\beta,\gamma),$
where
$$\alpha:=\max_{i,j}\max_{\substack{a\in A_i \\ a'\in A_j}}\big|d_{\Sp^1}(a,a')-d_{\Sp^2}(y_i,y_j)\big|,$$
$$\beta:=\max_{i,j}\max_{\substack{b\in B_i \\ b'\in B_j}}\big|d_{\Sp^1}(x_i,x_j)-d_{\Sp^2}(b,b')\big|,$$
$$\gamma:=\max_{i,j}\max_{\substack{a\in A_i \\ b\in B_j}}\big|d_{\Sp^1}(a,x_j)-d_{\Sp^2}(b,y_i)\big|.$$

In what follows we use the following notation: for each $i,j\in\{1,2,3,4,5,6\}$ we write
$$\alpha_{i,j} := \max_{\substack{a\in A_i \\ a'\in A_j}}\big|d_{\Sp^1}(a,a')-d_{\Sp^2}(y_i,y_j)\big|,$$
$$\beta_{i,j} := \max_{\substack{b\in B_i \\ b'\in B_j}}\big|d_{\Sp^1}(x_i,x_j)-d_{\Sp^2}(b,b')\big|,$$
$$\gamma_{i,j}:=\max_{\substack{a\in A_i \\ b\in B_j}}\big|d_{\Sp^1}(a,x_j)-d_{\Sp^2}(b,y_i)\big|,$$
so that $\alpha=\max_{i,j}\alpha_{i,j},$ $\beta=\max_{i,j}\beta_{i,j},$ and $\gamma=\max_{i,j}\gamma_{i,j}.$

In what follows we will be using the following simple fact repeated times:
\begin{fact}
Let $I=[m,M]$ and $I'=[m',M']$ be intervals on $\R_+$. Then, for all $t\in I$ and $t'\in I'$ we have
$$\max(m-M',m'-M)\leq |t-t'|\leq \max(M-m',M'-m).$$
\end{fact}

\begin{claim}
$\alpha\leq \frac{2\pi}{3}.$
\end{claim}

\begin{proof}
We now verify that $\alpha_{i,j}\leq \frac{2\pi}{3}$ for all $i,j$. Notice that by symmetry it suffices to do so for $i=1$ and $j\in\{1,2,3,4\}.$

\begin{description}

\item[($\alpha_{1,1}$)]  Follows from $\diam(A_1)=\frac{\pi}{3}$.

\item[($\alpha_{1,2}$)]  Notice that for all $a\in A_1$ and $a'\in A_2$, $d_{\Sp^1}(a,a')\in[0,\frac{2\pi}{3}].$ Since $d_{\Sp^2}(y_1,y_2)=\frac{2\pi}{3}$ it follows that $\alpha_{1,2}\leq \frac{2\pi}{3}.$

\item[($\alpha_{1,3}$)]  For all $a\in A_1$ and $a'\in A_3$, $d_{\Sp^1}(a,a')\in[\frac{\pi}{3},\pi].$ Since $d_{\Sp^2}(y_1,y_3)=\frac{\pi}{3}$ it follows that $\alpha_{1,3}\leq \frac{2\pi}{3}.$

\item[($\alpha_{1,4}$)]  For all $a\in A_1$ and $a'\in A_4$, $d_{\Sp^1}(a,a')\in[\frac{2\pi}{3},\pi].$ Since $d_{\Sp^2}(y_1,y_4)=\pi$ it follows that $\alpha_{1,4}\leq \frac{\pi}{3}.$
\end{description}
\end{proof}

\begin{claim}\label{claim:beta}
$\beta\leq \frac{2\pi}{3}$.
\end{claim}

\begin{proof} Notice that by symmetry it suffices to do so for $i=1$ and $j\in\{1,2,3,4\}.$
\begin{description}

\item[($\beta_{1,1}$)]  Follows from $\diam(B_1)=\frac{2\pi}{3}$.

\item[($\beta_{1,2}$)]  Notice that for all $b\in A_1$ and $b'\in B_2$, $d_{\Sp^2}(b,b')\in[0,\pi].$ Since $d_{\Sp^1}(x_1,x_2)=\frac{\pi}{3}$ it follows that $\beta_{1,2}\leq \frac{2\pi}{3}.$

\item[($\beta_{1,3}$)]  For all $b\in B_1$ and $b'\in B_3$, $d_{\Sp^2}(b,b')\in[0,\frac{2\pi}{3}].$ Since $d_{\Sp^1}(x_1,x_3)=\frac{2\pi}{3}$ it follows that $\beta_{1,3}\leq \frac{2\pi}{3}.$

\item[($\beta_{1,4}$)]  For all $b\in B_1$ and $b'\in B_4$, $d_{\Sp^2}(b,b')\in[\frac{\pi}{3},\pi].$ Since $d_{\Sp^1}(x_1,x_4)=\pi$ it follows that $\beta_{1,4}\leq \frac{2\pi}{3}.$
\end{description}
\end{proof}

\begin{claim}
$\gamma\leq \frac{2\pi}{3}$.
\end{claim}
\begin{proof}

Notice that by symmetry we can fix $i=1$. Also, by periodicity, the cases ($i=1$, $j=2$)  and  ($i=1$, $j=6$)  are isomorphic. Similarly, the cases ($i=1$, $j=3$) and ($i=1$, $j=5$) are also isomorphic. Thus, it is enough to check the cases ($i=1$, $j$), for $j=1,2,3$ and $4$.

\begin{description}

\item[($\gamma_{1,1}$)] In this case notice that for all $a\in A_1$, $d_{\Sp^1}(a,x_1)\in[0,\frac{\pi}{3}]$, and for all $b\in B_1$, $d_{\Sp^2}(b,y_1)\in[0,\frac{2\pi}{3}].$ From this it follows that $\gamma_{1,1}\in[0,\frac{2\pi}{3}].$

\item[($\gamma_{1,2}$)] Notice that for $a\in A_1$ and $b\in B_2$, $d_{\Sp^1}(a,x_2)\in[\frac{\pi}{6},\frac{\pi}{2}]$ and $d_{\Sp^2}(b,y_1)\in[\frac{\pi}{2}-\theta_0,\pi-\theta_0]\subset[\frac{\pi}{4},\frac{5\pi}{6}]$. Thus, $\gamma_{1,2}\leq \frac{2\pi}{3}$.

\item[($\gamma_{1,3}$)] The case $\gamma_{1,3}$ is solved as follows. Parametrize the set $B_3$ via spherical coordinates and consider the squared euclidean distance function: 
$$f(\theta)=\|y_1-p(\theta)\|^2,\,\,\mbox{$\theta\in[0,\frac{\pi}{2}]$},$$
where  $y_1=\big(\sin(\theta_0),0,\cos(\theta_0)\big)$, $\theta_0 = \arcsin(1/\sqrt{3})$, 
$$p(\theta) = \big(\sin(\theta)\cos(\frac{\pi}{3}),\sin(\theta)\sin(\frac{\pi}{3}),\cos(\theta)\big)\,\,\mbox{for $\theta\in[0,\frac{\pi}{2}].$}$$ Then, for all $\theta\in[0,\frac{\pi}{2}],$
$$f(\theta)=2-\left(2\sqrt{\frac{2}{3}}\cos(\theta)+\frac{1}{\sqrt{3}}\sin(\theta)\right).$$

The minimizer of $f$ in $[0,\frac{\pi}{2}]$ is $2\arctan(3-2\sqrt{2})$ and the minimum value is $2-\sqrt{3}.$ The corresponding spherical distance is $\frac{\pi}{6}.$ This means that for all $b\in B_3$, $d_{\Sp^2}(b,y_1)\in[\frac{\pi}{6},\pi].$ On the other hand, for all $a\in A_1$, $d_{\Sp^1}(a,x_3)\in[\frac{\pi}{2},\frac{5\pi}{6}].$ Thus, $\gamma_{1,3}\leq \frac{2\pi}{3}$.

\item[($\gamma_{1,4}$)] 
Now, let $\mu:\Sp^d\rightarrow \Sp^d$ be the antipodal map. Since for all $b\in B_4$ one has $d_{\Sp^d}(b,y_1)+d_{\Sp^d}(b,\mu(y_1))=\pi$ and by construction $y_4 = \mu(y_1),$ we have $$\pi\geq d_{\Sp^d}(b,y_1)\geq \pi - \max\{d_{\Sp^d}(b,y_4),\,b\in B_4\} \geq \pi-\diam(B_4) = \frac{\pi}{3}.$$ Thus, for all $b\in B_4$, $d_{\Sp^2}(b,y_1)\in[\frac{\pi}{3},\pi]$. Finally, note that for $a\in A_1$, $d_{\Sp^1}(a,x_4)\in[\pi,\frac{5\pi}{6}].$  Thus, $\gamma_{1,4} \leq \frac{2\pi}{3}.$

\end{description}
\end{proof}
\end{proof}

\end{document}